\documentclass[11pt,reqno]{article}
\usepackage{amsmath, amssymb, amsthm, mathrsfs}
\usepackage{amsfonts}
\usepackage{cases}
\usepackage{graphicx}
\usepackage{xcolor}
\usepackage[margin=0.8in]{geometry}
\usepackage[utf8]{inputenc}
\usepackage{verbatim}
\usepackage{caption,subcaption}
\usepackage[colorlinks,citecolor=blue,urlcolor=blue,linkcolor=blue]{hyperref}
\usepackage{bm}
\usepackage{algorithm}
\usepackage{algpseudocode}
\usepackage[normalem]{ulem}
\makeatletter
\renewcommand{\fnum@algorithm}{\fname@algorithm}

\numberwithin{equation}{section}
\newtheorem{Definition}{Definition}[section]
\newtheorem{Remark}{Remark}[section]
\newtheorem{Theorem}{Theorem}[section]
\newtheorem{Lemma}{Lemma}[section]

\newtheorem{Corollary}{Corollary}[section]
\newtheorem{Assumption}{Assumption}[section]

\newcommand{\be}{\begin{equation}}
	\newcommand{\ee}{\end{equation}}
\newcommand{\bee}{\begin{equation*}}
	\newcommand{\eee}{\end{equation*}}
\newcommand{\bi}{\begin{itemize}}
	\newcommand{\ei}{\end{itemize}}

\usepackage{algorithm}
\usepackage{algpseudocode}
\usepackage{soul}

\allowdisplaybreaks

\usepackage{accents}

\def \rightarrow{{\to}}
\def \E{\mathbb{E}}

\def \N{\mathbb{N}}
\def \P{\mathbb{P}}

\def \R{\mathbb{R}}

\def \T{\mathbb{T}}

\def \Cc{{\mathcal C}}

\def \Ac{{\mathcal A}}
\def \Ec{{\mathcal E}}
\def \Pc{{\mathscr P}}

\def \Hc{\mathcal{H}}
\def \Mc{{\mathcal M}}
\def \eps{\varepsilon}
\def \Leb{\operatorname{\texttt{Leb}}}
\def \cone{\operatorname{\text{cone}}}
\def \ul{\operatorname{\textbf{ul}}}
\def \wg{\operatorname{\textbf{wg}}}

\def \bmpi{\bm{\pi}}

\title{Equilibrium for Time-Inconsistent Mean Field Games: A Systematic Analysis via Entropy Regularization}
\author{Erhan Bayraktar\thanks{Department of Mathematics, University of Michigan, Ann Arbor, MI, USA. \url{erhan@umich.edu}} \and
 Zhenhua Wang\thanks{Zhongtai Securities Institute for Financial Studies, Shandong University; Shandong Province Key Laboratory of Financial Risk, Jinan, Shandong, China. \url{zhenhuaw@sdu.edu.cn}}
\and Xiang Yu\thanks{Department of Applied Mathematics,  The Hong Kong Polytechnic University, Kowloon, Hong Kong. \url{xiang.yu@polyu.edu.hk}}
	\and Keyu Zhang\thanks{
	Department of Mathematical Sciences, Tsinghua University, Beijing, China. \url{zhangky21@mails.tsinghua.edu.cn}}}
\date{\vspace{-0.3cm}}

\begin{document}
	\maketitle

	\begin{abstract}
This paper studies the existence and approximation of equilibria for general time-inconsistent mean field game (MFG) problems in continuous time. To handle the intricate nonlocal equilibrium Hamilton–Jacobi–Bellman (EHJB) system arising from initial-time dependence, such as non-exponential discounting, we develop a vanishing entropy regularization approach. Using entropy regularization, we first characterize the regularized equilibrium through a coupled exploratory equilibrium HJB (EEHJB) equation and a law-dependent stochastic differential equation. By exploiting Schauder fixed-point arguments and tailored parabolic regularity estimates in a suitable functional space involving both value functions and measure flows, we establish the global existence of regularized equilibria under mild assumptions. We then establish convergence as the entropy regularization vanishes. By employing compactness arguments, Young measure techniques, and a duality tool for divergence-form Fokker–Planck equations, we prove that the regularized equilibria converge, up to subsequences, to a mean-field equilibrium of the original MFG. Furthermore, under entropy regularization, we propose a policy iteration algorithm and establish its convergence under short-time-horizon and weak-terminal-interaction conditions.

		\ \\
		\noindent\textbf{Keywords}: Time-inconsistent mean field games, entropy regularization, exploratory equilibrium HJB equation, convergence analysis, policy iteration convergence
        \ \\

		\noindent\textbf{2020 MSC}: 91A16, 60H30
	\end{abstract}

\section{Introduction}
Mean field game (MFG) theory, pioneered independently by Lasry and Lions \cite{lasry2007mean} and Huang, Malham\'e, and Caines \cite{huang2006large}, provides a powerful framework for analyzing and approximating strategic interactions in large populations. The theory of mean field models has developed rapidly over the past few decades. In the existing literature, there are two main approaches to establishing the existence of mean field equilibria. The first is the analytical approach, initially proposed by Lasry and Lions \cite{lasry2007mean}, which characterizes the MFE through a coupled system of Hamilton-Jacobi-Bellman (HJB) and Fokker-Planck (FP) equations. In this setup, the backward HJB equation determines the optimal value function of a representative agent given the flow of the population distribution, while the forward FP equation governs the macroscopic evolution of the population under the optimal control. The second approach, introduced by Carmona and Delarue \cite{carmona_mean_2013,carmona_probabilistic_2013}, is probabilistic and relies on the stochastic maximum principle, where the forward-backward partial differential equation (FBPDE) is replaced by a forward-backward stochastic differential equation (FBSDE). Both methods recast equilibrium existence as the well-posedness of a coupled forward-backward system, which often requires stringent regularity assumptions on model coefficients. Later, alternative probabilistic methods were developed to alleviate these strong model assumptions, inspired by relaxation techniques originally used to ensure the existence of solutions in classical stochastic control problems. Rather than seeking a classical equilibrium, these relaxation techniques formulate the problem as an optimization over the laws of controlled state processes (the compactification method, see, e.g., \cite{lacker2015mean,common_2016,campiFis, bo2025meanfieldgamecontrols}) or over occupation measures subject to linear constraints (the linear programming method, see, e.g., \cite{bouveret_mean-field_2020,dumitrescu_control_2021,GHZZ25,liang2025}), together with a mean field consistency condition. This formulation greatly facilitates existence proofs by requiring milder model assumptions.

The aforementioned studies on MFGs crucially assume time consistency in the sense that, given the flow of the population distribution, the problem reduces to a standard time-consistent stochastic control problem for a representative agent. However, time inconsistency commonly arises in a wide range of economic and financial applications, often due to non-exponential discounting adopted by the decision maker. In this setting, a policy deemed optimal today may no longer remain optimal at future dates. To address the failure of global optimality, a widely used alternative, pioneered by Strotz \cite{Strotz}, is to treat the time-inconsistent control problem as an intra-personal game among the decision maker's current and future selves and to seek a subgame-perfect Nash equilibrium. Over the past decade, a substantial body of literature has been devoted to single-agent time-inconsistent control problems (see, e.g., \cite{Bjork17,Yong2012,huang2021strong,bayraktar2025relaxed}). In discrete-time stochastic control, \cite{BayraktarHan2023MarkovEquilibrium} proves the existence of Markov equilibrium controls under inf-compactness conditions. Related sophisticated-equilibrium formulations have also appeared in other time-inconsistent problems; for instance, \cite{BayraktarHan2025EquilibriumTransport} develops an equilibrium transport problem with time-inconsistent costs under bicausal constraints. In the continuous-time setting under non-exponential discounting, equilibrium existence can be characterized either by an extended HJB equation in \cite{Bjork17}, formulated as a system of nonlinear PDEs, or equivalently by an equilibrium HJB (EHJB) equation in \cite{Yong2012}, formulated as a nonlinear and nonlocal PDE. Standard verification proofs based on the classical It\^{o} lemma are often invoked to establish the existence and characterization of equilibria, and therefore require classical solutions to the EHJB equation. Because of the complexity of this equation, most existing studies focus on specific models that admit explicit solutions and use a guess-and-verification procedure to derive equilibria. In the general case, \cite{lei_nonlocal_2023} investigated the short-time existence of classical solutions to the general EHJB equation under strong model assumptions, and \cite{lei_nonlocal_2024} established the existence of global solutions by imposing additional conditions on the model coefficients and cost functions. When these restrictive model assumptions are not fulfilled, classical solutions to the extended HJB or EHJB equation may fail to exist, leaving open questions about equilibrium existence.

On the other hand, entropy regularization has become popular in reinforcement learning (RL) because it encourages exploration during trial-and-error learning; see the comprehensive discussion in \cite{wang2020reinforcement} for continuous-time reinforcement learning. Recently, beyond the context of RL, entropy regularization has also attracted considerable attention in various game problems thanks to the explicit and usually unique Gibbs-form characterization of the best response control, which can significantly simplify the fixed-point analysis of equilibria in the regularized problem. To establish equilibria in the original game problem, a key step is to prove convergence as the entropy regularization vanishes. \cite{bayraktar2025relaxed} initiated this approach in studying time-inconsistent MDPs under general discounting and obtained the existence of relaxed equilibria for the intra-personal game in both discrete-time and continuous-time models; \cite{YuYuan26} investigated the time-inconsistent mean field stopping problem for a social planner in a discrete-time model and proved the existence of relaxed equilibria; \cite{YZZZ} used the vanishing entropy regularization approach for the major player's problem and established the existence of mean field equilibria for the discrete-time major-minor MFG of stopping; \cite{DDFX25} studied the continuous-time MFG of stopping in the time-consistent setting using the vanishing entropy regularization approach and addressed the existence of mean field equilibria; \cite{DZ25} tackled the time-inconsistent mean-variance stopping problem as an intra-personal game by establishing convergence as entropy vanishes. Recently, \cite{WYZZ2026} analyzed an infinite-horizon time-inconsistent control problem using this vanishing entropy regularization approach. They showed that entropy regularization improves regularity and significantly helps establish the existence of a classical solution to the exploratory equilibrium Hamilton-Jacobi-Bellman (EEHJB) equation. Furthermore, they proved that the classical solutions of the EEHJB equations converge, up to a subsequence, to a strong solution of the original EHJB equation as entropy tends to zero, and developed new verification arguments to conclude the existence of an equilibrium in the original problem without resorting to the existence of a classical solution.

Despite its success in the above time-inconsistent control problems and time-consistent game problems, whether entropy regularization can be used to study MFGs in the time-inconsistent setting remains an interesting open problem. In fact, time-inconsistent MFGs remain underexplored in the literature, with most existing studies focusing on linear-quadratic structures or specific models that admit explicit equilibria (see, e.g., \cite{ni2018,jun2021,Wang2023,liang2024}). Beyond these special setups, only a handful of works have analyzed the general existence of equilibria. For instance, \cite{BW2025} investigated time-inconsistent MFGs in a discrete-time, infinite-horizon setting and analyzed the convergence of $N$-player game equilibria toward the mean field limit; \cite{WZ2024} studied the so-called sharp equilibrium strategies for mean field stopping games and showed the existence of such equilibria in two classes of examples; \cite{LYZ2025} examined a class of time-inconsistent mean field control problems under non-exponential discounting, where all agents cooperate to optimize a collective objective, and characterized the equilibrium via an EHJB equation on the Wasserstein space. Recently, \cite{zhou2026existence} investigated the existence of time-consistent equilibria for general time-inconsistent stochastic games in discrete time with uncountable
state space.

This paper addresses general finite-horizon time-inconsistent MFGs in diffusion models by developing the method of vanishing entropy regularization. Specifically, when the population aggregation is given, the representative agent seeks an intra-personal Nash equilibrium as a best response to future selves. The mean field equilibrium is then formulated through a consistency condition with respect to the population aggregation. In sharp contrast to \cite{WYZZ2026} for the single-agent problem, this work requires a delicate analysis of the mean field state distribution for convergence as entropy tends to zero. It is also worth mentioning \cite{MeiZhu2022}, which investigates the existence of classical mean field equilibria in time-inconsistent MFGs, where the equilibrium policy takes values in the action space. Establishing the existence of such classical equilibria relies heavily on restrictive structural assumptions (see Remark \ref{rmk:comparison_literature} for further discussion). In contrast, we consider the general model setup and work with equilibria in the \textit{relaxed} sense, so that policies take values in the space of probability measures over the action space (see Definition \ref{def:equi.relaxed}). This paper conducts a systematic study of time-inconsistent MFGs through the entropy-regularized formulation: we establish the existence of equilibria in the original problem using the convergence of regularized equilibria, propose a policy iteration rule, and analyze its convergence to the regularized equilibrium.

We summarize the key theoretical contributions of the present paper as follows:

First, we formulate the time-inconsistent MFGs under Shannon entropy regularization and show that the regularized equilibrium (see Definition \ref{def:equi.reg}) is characterized by a backward EEHJB equation coupled with a forward law-dependent SDE (see \eqref{eq:cha.HJBentropy_prime}--\eqref{eq:cha.m}). The main contribution in this step lies in establishing the global well-posedness of this fully coupled system via a fixed-point argument. Compared to the single-agent setting in \cite{WYZZ2026}, we consider a finite-horizon problem with time-dependent mean field interactions. Consequently, the coefficients are inherently time-dependent, which forces us to analyze functions defined on the domain $\Delta_{[0,T]}\times \mathbb{R}^d$ with $\Delta_{[0,T]}:=\{(t,s):0\leq t\leq s\leq T\}$ due to the mismatch between the initial time $t$ and the running time $s$. To handle this class of functions, we introduce tailored $t$-norms (see \eqref{eq:t_norms} and \eqref{eq:t_norms2}). The existence proof relies on a carefully constructed fixed-point operator $\Phi_\lambda$ defined in \eqref{eq:Phi.fixedpoint}. By leveraging the sublinear growth of the entropy term in Lemma \ref{lm:entropy.ln}, we first derive H\"older estimates for the Gibbs-form operator $\pi(t,x,a):= \Gamma_{\lambda}(t,x, D_x w(t,x), m_t, a)$, provided that the input pair $(w,m)$ possesses sufficient regularity (see Lemma \ref{lm:pi.est}). Subsequently, we prove that the resulting auxiliary value function $V_{\lambda}^{\pi,m}$ serves as the classical solution to the linear parabolic PDE \eqref{eq:prop.iteraPDE} and satisfies certain uniform estimates, while the resulting measure flow $\mu^{\pi,m}$ preserves the same regularity as the input flow $m$ (see Lemma \ref{lm:iteration.estimateentropy}). In contrast to the single-agent setting, the mean field nature requires our fixed-point operator $\Phi_\lambda$ to incorporate an additional measure flow input. Building upon the established estimates, we identify a suitable compact subset $\mathcal{E}_{\lambda}$ that encompasses both sufficiently regular value functions and well-behaved measure flows. While the compact set for the value functions can be chosen following arguments similar to those in \cite{WYZZ2026}, a key technical contribution here is identifying the precise set for the measure flow component.
This careful construction ensures that $\Phi_{\lambda}$ is a continuous self-map on $\mathcal{E}_{\lambda}$ (see Lemma \ref{lm:Phi.continuous}). With these preparations, we invoke Schauder's fixed-point theorem and the verification argument to establish the existence of a regularized equilibrium.

Second, we carry out a delicate convergence analysis of the regularized equilibria toward an equilibrium of the original problem as the entropy parameter $\lambda \to 0$, thereby establishing the existence of a relaxed equilibrium in the original time-inconsistent MFGs. To this end, we consider a sequence $(V^n,\pi^n,m^n)$ satisfying the regularized system \eqref{eq:HJBn1}--\eqref{eq:consisn}. Using the estimates established in the previous step and Young measure theory, we extract a subsequence that converges to a candidate limit triplet $(V^\infty, \pi^\infty, m^\infty)$ (see Lemma \ref{lm:thm.C12andyoung}). A key distinction from the single-agent setting is that the sequence of auxiliary value functions $V^n$ is defined on $\Delta_{[0,T]}\times \mathbb{R}^d$. Consequently, subsequential convergence must be established in our tailor-made $t$-norms, which requires new arguments.
Furthermore, to show that the limit pair $(\pi^\infty,m^{\infty})$ constitutes an equilibrium of the original problem, we proceed in two core steps. The first step is a verification proof based on the convergence analysis, showing that $\pi^\infty$ satisfies the equilibrium condition \eqref{eq:def.equipi}. The limit $V^\infty$ is identified as a \textit{strong}, rather than classical, solution to the EHJB system \eqref{lim:eqhjb}--\eqref{lim:eqhjb2}, which makes standard verification arguments inapplicable. However, because each $V^{n_k}$ is a classical solution to the EEHJB equation, we can use the convergence result, together with delicate localization arguments, to conclude that $\pi^{\infty}$ satisfies the equilibrium condition \eqref{eq:def.equipi}. The second step is to identify $m^{\infty}$ as the law of the state process driven by the limit policy $\pi^{\infty}$, which is a central issue in the mean field model. The primary difficulty here stems from the fact that Young measure theory only provides weak-$\ast$ convergence of the controlled drift coefficients (see Lemma \ref{lm:thm.C12andyoung} (ii)). This weak convergence is insufficient to guarantee the convergence of the distributions via standard SDE stability arguments. To bridge this gap, we employ a duality approach that uses solutions to linear backward PDEs as test functions, shifting our focus to the Fokker-Planck equation in divergence form. Using the regularity theory for parabolic PDEs in divergence form and compactness arguments, we prove the strong $L^1$ convergence of the corresponding probability densities. This strong convergence allows us to rigorously identify the topological limit of the measure flows with the distribution of the state process driven by the limit policy, thereby addressing the issue of measure flow consistency.

Third, using entropy regularization,
we propose a Policy Iteration Algorithm (PIA) and analyze its convergence to the established regularized equilibrium. Starting from an arbitrary admissible policy, we show that the successive iterations under the operator $\Phi_\lambda$ converge to the fixed point. This provides an effective numerical method for attaining the regularized equilibrium through greedy iterations and may also provide part of the theoretical foundation for policy iteration algorithms in continuous-time RL when the model is unknown. Inspired primarily by recent results on PIA convergence for entropy-regularized stochastic control problems \cite{ma2025convergence,HYZ26}, we employ the Bismut-Elworthy-Li representation to derive useful spatial gradient estimates. Coupled with the stability analysis of the measure flows, we rigorously prove that the PIA acts as a strict contraction and converges to the regularized equilibrium, provided that the time horizon is sufficiently short and the terminal measure dependence is weak (see Theorem  \ref{thm:PIA_convergence}).
Note that, even in the time-consistent setting, existing studies on the convergence of PIA for MFGs rely on various restrictive model assumptions, such as separable Hamiltonians and measure-independent terminal costs \cite{TangSong2024,camilli2022rates,cacace2021policy}. For more general, non-separable models, convergence is known only under a similar small-time assumption, while terminal measure dependence is still excluded in \cite{lauriere2023policy}. Our work appears to be the first study to extend PIA convergence for general MFGs to the time-inconsistent setting under weaker model assumptions. We refer readers to Remark \ref{rmk:PIA_assumptions} for a detailed discussion.

Finally, we highlight that the vanishing entropy regularization approach can be viewed as an alternative to the relaxation techniques widely used in time-consistent MFGs. In our time-inconsistent setting, a representative agent no longer solves a standard global optimization problem but instead seeks an intra-personal equilibrium. This fundamental distinction prevents the direct application of classical relaxation arguments, which inherently rely on the minimization of a cost functional over a compact set. In this context, our vanishing entropy regularization approach serves as a vital bridge, providing a novel and feasible method for addressing the existence of relaxed equilibria in time-inconsistent MFGs.

The rest of the paper is organized as follows. Section \ref{subsect:nota} collects frequently used notation and preliminary results. Section \ref{sec:model} introduces the formulation of time-inconsistent MFGs in the continuous-time setting and the definition of relaxed equilibrium. Section \ref{sect:regexist} studies the problem under entropy regularization and establishes the existence of a regularized equilibrium. Section \ref{sec:convergence} carries out the convergence analysis, demonstrating that the limit of the regularized equilibria corresponds to an equilibrium of the original time-inconsistent MFGs. Finally, Section \ref{sect:pia} introduces a PIA in the regularized problem and establishes its convergence to the regularized equilibrium.

\subsection{Notations}\label{subsect:nota}
Let $\mathbb{N}$ be the set of all positive integers and $T>0$ be the finite horizon. In the $m$-dimensional Euclidean space $\R^m$,  we denote by $| \cdot|$ the Euclidean norm, by $\Leb(\cdot)$ the Lebesgue measure, and by $B_r(x)$ the ball centered at $x$ with radius $r$.

Given a function $w(t,x):[0,T]\times\R^d\rightarrow \R$, let $\partial_tw$ denote the right derivative with respect to the time variable $t$ and let $D_xw$ (resp. $D^2_xw$) denote the first (resp. second) partial derivative vector (Hessian) with respect to the space variable $x$. 
Let $D$ be a generic domain in $[0,T] \times \R^d$. We first recall the notations of parabolic H\"older spaces for $\alpha \in (0,1)$:
\begin{align*}
	\|w\|_{\Cc^0(D)} &:= \sup_{(t,x)\in D} |w(t,x)|= \|w\|_{L^\infty(D)},\quad
	[w]_{\Cc_{\alpha}(D)} := \sup_{\substack{(t,x)\neq (s,y)\in D\\ (|t-s|+|x-y|^2)\leq 1}} \frac{|w(t,x)-w(s,y)|}{(|t-s|+|x-y|^2)^{\alpha/2}}, \\
	\|w\|_{\Cc_{\alpha}(D)} &:= [w]_{\Cc_{\alpha}(D)}+\|w\|_{\Cc^0(D)}.
\end{align*}
Given a multi-index $a=(a_1,\cdots, a_d)$ with $|a|_{l_1}:= \sum_{i=1}^d a_i$, define $D_x^a w:= \frac{\partial^{|a|_{l_1}} w }{\partial^{a_1}_{x_1}\cdots \partial^{a_d}_{x_d}}$
and we say $w\in \Cc^{l, k}_{\alpha}(D)$
(resp. $w\in \Cc^{l, k}(D)$)
if
$$
\|w\|_{\Cc^{l, k}_{\alpha}(D)}:= \sum_{0\leq b\leq l, 0\leq |a|_{l_1}\leq k} \|\partial^b_t D^a_x w \|_{\Cc_{\alpha}(D)}<\infty\; \Big(\text{resp. $\|w\|_{\Cc^{l, k}(D)}:=  \sum_{0\leq b\leq l, 0\leq |a|_{l_1}\leq k}\|\partial^b_t D^k_xw \|_{\Cc^0(D)}<\infty$} \Big).
$$
Define $D_N(t_0, x_0):= ((t_0, t_0+N)\cap [0,T])\times B_N(x_0) $ and write $D_N:= D_N(0,0)$, for $N>0$.
We further define the following weighted  H\"older norm on $[0,T]\times \R^d$:
\begin{equation*}
	\|w\|_{\Cc^{l, k, \wg}_{\alpha}([0,T]\times \R^d)}:=\sum_{N\in \mathbb{N}} \frac{1}{2^N} \|w\|_{\Cc^{l, k}_{\alpha}(D_N)}.
\end{equation*}
Given $w(t,s,x)$ defined on the domain $\Delta_{[0,T]}\times \R^d$ with $\Delta_{[0,T]}:=\{(t,s):0\leq t\leq s\leq T\}$ and a spatial subset $Q \subset \R^d$, we further define the following ``$t$-norms":
\begin{equation}\label{eq:t_norms}
    \begin{aligned}
	\|w\|_{\Cc^{l, k}_{\alpha, [t_0, t_1]\times Q}} &:= \sup_{t\in[t_0, t_1]} \|w(t,\cdot,\cdot)\|_{\Cc^{l, k}_{\alpha}([t, T]\times Q)},\,
	\|w\|_{\Cc^{l, k}_{[t_0, t_1]\times Q}} := \sup_{t\in[t_0, t_1]} \|w(t,\cdot,\cdot)\|_{\Cc^{l, k}([t, T]\times Q)},\\
	\|w\|_{\widetilde{\Cc}^{l, k}_{\alpha,[t_0, t_1]\times Q}} &:= \sup_{t\in[t_0, t_1]} \|w(t,\cdot,\cdot)\|_{\Cc^{l, k}_{\alpha}([t, T]\times Q)} + \sup_{t\in[t_0, t_1]} \|\partial_t w(t,\cdot,\cdot)\|_{\Cc^{l, k}_{\alpha}([t, T]\times Q)},\\
	\|w\|_{\widetilde{\Cc}^{l, k}_{[t_0, t_1]\times Q}} &:= \sup_{t\in[t_0, t_1]} \|w(t,\cdot,\cdot)\|_{\Cc^{l, k}([t, T]\times Q)} + \sup_{t\in[t_0, t_1]} \|\partial_t w(t,\cdot,\cdot)\|_{\Cc^{l, k}([t, T]\times Q)}.
\end{aligned}
\end{equation}
When $Q=\R^d$, we simply omit the spatial domain in the subscript. 

We recall the Sobolev norm $\|w\|_{W^{1,2}_p(D)} := \sum_{2l+ |a| \leq 2} \|\partial^l_t D^a_x w \|_{L^p(D)}$ for a bounded domain $D$, and define the following ``uniformly local Sobolev norm" on $[t,T]\times \R^d$:
$$
\|w(t,\cdot,\cdot)\|_{W^{1,2,\ul}_p([t,T]\times\R^d)}:=\sup_{(s,x)\in [t, T]\times \R^d} \|w(t,\cdot, \cdot) \|_{W^{1,2}_p([s,s+1]\cap[0,T]\times B_1(x))}.
$$
We further define the ``$t$-Sobolev norm":
\begin{equation}\label{eq:t_norms2}
    \begin{aligned}
	\|w\|_{\widetilde{W}^{1,2}_{p,[t_0, t_1]}} &:= \sup_{t\in[t_0, t_1]}\|w(t,\cdot,\cdot)\|_{W^{1,2,\ul}_p([t,T]\times\R^d)},\\
	\|w\|_{\widehat{W}^{1,2}_{p,[t_0, t_1]}} &:= \sup_{t\in[t_0, t_1]}(\|w(t,\cdot,\cdot)\|_{W^{1,2,\ul}_p([t,T]\times\R^d)}+\|\partial_tw(t,\cdot,\cdot)\|_{W^{1,2,\ul}_p([t,T]\times\R^d)}). 
\end{aligned}
\end{equation}

Throughout the paper, we use the subscript of a norm to denote the functional space consisting of all functions for which that norm is finite (e.g., $w \in \widetilde{W}^{1,2}_{p,[t_0, t_1]}$ means $\|w\|_{\widetilde{W}^{1,2}_{p,[t_0, t_1]}} < \infty$). Let $0<\alpha<1$ be an arbitrarily fixed H\"older constant, and take the power $p$ in any Sobolev norm to be $\frac{d+2}{1-\alpha}$.

Let $(\Omega, \mathcal{F}, \mathbb{F}, \mathbb{P})$ be a complete filtered probability space  supporting an $m$-dimensional standard Brownian motion $W = \{W_t\}_{t \in [0, T]}$. Let $\mathcal{F}_0$ be a given $\sigma$-algebra which is independent of the Brownian motion. Let $\mathbb{F}=\left\{\mathcal{F}_t: 0 \leq t \leq T\right\}$, where $\mathcal{F}_t:=\sigma\left\{\mathcal{F}_0, W(s): 0 \leq s \leq t\right\} \vee \mathcal{N}_0$, and $\mathcal{N}_0$ is the set of all $\mathbb{P}$-null sets. Denote by $\mathscr{P}\left(\R^d\right)$ the space of probability measures on $\R^d$.

Let $L_{\mathcal{F}}^2(\R^d)$ be the collection of $\R^d$-valued random variables with finite second moment, i.e.,
\begin{equation*}
	L_{\mathcal{F}}^2(\R^d) := \left\{\xi: \Omega \rightarrow \R^d \mid \xi \text{ is } \mathcal{F}\text{-measurable with } \mathbb{E}|\xi|^2<\infty \right\}.
\end{equation*}
$L_{\mathcal{F}}^2(\R^d)$ is equipped with the norm $\|\xi\|_{L^2} := \left(\mathbb{E}|\xi|^2\right)^{\frac{1}{2}}$. For any $\xi \in L^2(\R^d)$, denote by $\operatorname{law}(\xi)$ the distribution of $\xi$.

Let $\mathscr{P}_2(\R^d)$ be the space of probability measures with finite second moments equipped with the Wasserstein-2 metric $\mathcal{W}_2(\cdot, \cdot)$, i.e.,
\begin{equation*}
	\mathcal{W}_2^2(\rho, \gamma) := \inf_{\pi \in \Pi^{\rho, \gamma}} \int_{\R^d \times \R^d} |x-y|^2 \pi(dx, dy),
\end{equation*}
where
\begin{equation*}
	\Pi^{\rho, \gamma} := \left\{\pi \in \mathscr{P}_2(\R^d \times \R^d) : \pi(dx, \R^d)=\rho(dx), \pi(\R^d, dy)=\gamma(dy)\right\}.
\end{equation*}
It is easy to see that
\begin{equation}\label{ineq:w2}
	\mathcal{W}_2^2(\operatorname{law}(\xi), \operatorname{law}(\eta)) \leq \|\xi-\eta\|_{L^2}^2.
\end{equation}
Let $\mathscr{P}_2^{\kappa, C}(\R^d)$ be a subset of $\mathscr{P}_2(\R^d)$ defined by
\begin{equation*}
	\mathscr{P}_2^{\kappa, C}(\R^d) := \left\{\rho \in \mathscr{P}_2(\R^d) : \int_{\R^d} |x|^{2+\kappa} \rho(dx) \leq C\right\}.
\end{equation*}
Note that $\mathscr{P}_2^{\kappa, C}(\R^d)$ is a compact subset of $\left(\mathscr{P}_2(\R^d), \mathcal{W}_2\right)$ for any $\kappa, C>0$.

Throughout the paper, we suppose that $\mathcal{F}_0$ is rich enough so that for any $\rho \in \mathscr{P}_2(\R^d)$, there exists a $\xi \in \mathcal{F}_0$ such that $\operatorname{law}(\xi)=\rho$.

Let $\mathscr{M} := C\left([0, T], \mathscr{P}_2(\R^d)\right)$ be the set of $\mathscr{P}_2(\R^d)$-valued continuous curves on $[0, T]$ equipped with the uniform metric $d$, i.e.,
\begin{equation*}
	d(\mu_1, \mu_2) := \sup_{0 \leq t \leq T} \mathcal{W}_2(\mu_1(t), \mu_2(t)).
\end{equation*}
Since $\left(\mathscr{P}_2(\R^d), \mathcal{W}_2\right)$ is complete, so is $(\mathscr{M}, d)$. Write $\mathscr{M}_{\nu} := \{\mu \in \mathscr{M} : \mu(0)=\nu\}$. {Throughout this paper, let $0<\kappa<\infty$ be an arbitrarily fixed constant.} Define a subset of $\mathscr{M}_\nu$ by
\begin{equation*}
	\mathscr{M}_\nu^{\kappa, C} := \left\{\mu \in \mathscr{M}_\nu \;\middle|\;
		 \sup_{0 \leq t \neq s \leq T} \frac{\mathcal{W}_2^2(\mu(t), \mu(s))}{|t-s|} \leq C,  \;\text{and } \mu(t) \in \mathscr{P}_2^{\kappa, C}(\R^d) \text{ for any } t \in[0, T]
\right\}.
\end{equation*}
By the well-known Arzel\`a-Ascoli lemma, $\mathscr{M}_\nu^{\kappa, C}$ is a compact and convex subset of $\mathscr{M}$.

	\section{Model Setup}\label{sec:model}
Let $U$ denote the action space, which is a bounded subset of $\R^\ell$ with $\Leb(U)>0$. We also denote by $\mathscr{P}(U)$ the set of all probability measures on $U$, and by $\mathscr{P}_c(U)$ the subset of $\mathscr{P}(U)$ consisting of all probability measures absolutely continuous with respect to the Lebesgue measure. For technical convenience, we focus on the model in which only the drift coefficient of the state process is controlled. The general case in which the diffusion coefficient is also controlled is left for future study.

{ Given a population flow $m\in\mathscr{M}_{\nu}$ with $\nu\in\mathscr{P}_{2+\kappa}(\R^d)$ for some $\kappa>0$, the dynamics of a single agent using a feedback strategy under $m$ are
\begin{equation*}
	dX^\pi_s = \left( \int_U b(s, X^\pi_s, m_s, a)\pi(s, X^\pi_s, a)da \right) ds + \sigma(s, X^\pi_s, m_s) dW_s 
\end{equation*}}
where $b:[0,T]\times \R^d\times \mathscr{P}_2(\R^d)\times U\rightarrow \R^d$, $\sigma:[0,T]\times \R^d\times \mathscr{P}_2(\R^d)\rightarrow \R^{d\times m}$. The payoff functional associated with applying $\pi$ under the population flow $m$ is then defined by
\begin{equation*}
	J^{\pi, m}(t,x) := \mathbb{E}_{t,x}\left[\int_t^T \left( \int_U r(l-t, X^{\pi}_l, m_l, a)\pi(l, X^{\pi}_l, a)da \right) dl + F(t, X^{\pi}_T, m_T) \right],
\end{equation*}
where $r:[0,T]\times \R^d\times \mathscr{P}_2(\R^d)\times U\rightarrow \R$ is the running reward and $F:[0,T]\times\R^d\times \mathscr{P}_2(\R^d)\rightarrow \R$ is the terminal reward. We further define the auxiliary function, for $(t,s,x)\in\Delta_{[0,T]}\times \R^d$, by
\begin{equation*}
	V^{\pi, m}(t,s,x) := \mathbb{E}_{s,x}\left[\int_s^T  \left( \int_U r(l-t, X^{\pi}_l, m_l, a)\pi(l, X^{\pi}_l, a)da \right) dl + F(t,X^{\pi}_T, m_T)\right].
\end{equation*}
Here, $J^{\pi, m}(t,x) = V^{\pi, m}(t,t,x)$ represents the payoff functional for entering the game at $(t,x)$, while $V^{\pi, m}(t,s,x)$ represents the payoff evaluated from a future time-state pair $(s,x)$.

For the rest of the paper, we use the notation
\begin{equation*}
	\tilde{f}(t,x,m,\varpi):= \int_U f(t,x,m,a)\varpi(da).
\end{equation*}
This convention applies to any generic function $f(t, x, m, a): [0,T]\times \R^d \times \mathscr{P}_2(\R^d)\times U\rightarrow \R$ and any generic distribution $\varpi\in \mathscr{P}(U)$.

\begin{Definition}
	The set of admissible policies, denoted by $\Ac$, is defined as the collection of all Borel measurable relaxed feedback policies $\pi: [0,T] \times \R^d \rightarrow \mathscr{P}(U)$.
\end{Definition}
\begin{Assumption}\label{assume.r}
	We assume that there exist constants $K_1, K_2, \eta > 0$ such that the following conditions hold for all $(t,x,m,a) \in [0,T]\times\mathbb{R}^d\times\mathscr{P}_2(\mathbb{R}^d)\times U$:
	\begin{itemize}
		\item[(i)] \textbf{Hölder regularity in $(t,x)$:}
		\begin{equation*}\label{eq:assume.integralr}
		\begin{aligned}
			&\|r(\cdot,\cdot,m, a)\|_{\Cc^{1,0}_{\alpha}([0,T]\times \R^d)}+\|b(\cdot,\cdot,m, a)\|_{\Cc_{\alpha}([0,T]\times \R^d)}\\
			&+\|\sigma(\cdot,\cdot,m)\|_{\Cc_\alpha([0,T]\times \R^d)}+\|F(\cdot,\cdot,m)\|_{\Cc^{1,2}_\alpha([0,T]\times \R^d)}\leq K_1.		\end{aligned}
	\end{equation*}

		\item[(ii)] \textbf{Uniform ellipticity:}
		\begin{equation*}
			\eta|\xi|^2 \leq \xi \sigma\sigma^T(t,x,m) \xi^T, \quad \forall \xi \in \mathbb{R}^d \setminus \{0\}.
		\end{equation*}

		\item[(iii)] \textbf{Spatial regularity:} $b$, $\sigma$, and $r$ are Lipschitz continuous in $x$ uniformly in the other arguments, i.e.,
		\begin{equation*}
		    |b(t,x,m,a)-b(t,y,m,a)|+|\sigma(t,x,m) - \sigma(t,y,m)|+|r(t,x,m,a)-r(t,y,m,a)| \leq K_2|x-y|, \quad \forall x,y \in \mathbb{R}^d.
		\end{equation*}

		\item[(iv)] \textbf{Regularity in measure:} $b$, $\sigma$, $r$, and $\partial_t r$ are Lipschitz continuous in $m$ uniformly in the other arguments, i.e.,
		\begin{align*}
			|b(t,x,m,a)-b(t,x,\rho,a)| &+ |\sigma(t,x,m)-\sigma(t,x,\rho)| +|r(t,x,m,a)-r(t,x,\rho,a)| \\&+ |\partial_t r(t,x,m,a)-\partial_t r(t,x,\rho,a)| \leq K_2 \mathcal{W}_2(m,\rho), \quad \forall m,\rho \in \mathscr{P}_2(\R^d)
		\end{align*}
		and $F, \partial_t F$ are continuous in $m$ uniformly in the other arguments.
	\end{itemize}
\end{Assumption}
\begin{Remark}
    Under the conditions of $b$ and $\sigma$ in Assumption \ref{assume.r}, for any admissible policy $\pi \in \Ac$, any measure flow $m \in \mathscr{M}_\nu$, and any initial condition $(t,x) \in[0,T) \times \R^d$, the controlled SDE
    \begin{equation*}
        dX^\pi_s = \tilde{b}(s, X^\pi_s, m_s, \pi(s, X^\pi_s)) ds + \sigma(s, X^\pi_s, m_s) dW_s, \quad X^\pi_t = x, \quad s \in [t,T],
    \end{equation*}
    admits a unique strong solution (see, e.g., \cite[Theorem 1]{veretennikov1981strong}). Moreover, the boundedness conditions on $r$ and $F$ ensure that $J^{\pi,m}(t,x)$ is always finite.
\end{Remark}

We give the following definition of relaxed equilibrium for the time-inconsistent MFG problem.
\begin{Definition}\label{def:equi.relaxed}
	A pair $(\pi^*, m^*)$, where $\pi^*\in\Ac$  and $m^*\in\mathscr{M}_{\nu}$, is called a (relaxed) mean-field equilibrium if the following conditions are satisfied:
	\begin{itemize}
		\item[(a)]{\textit{(Equilibrium response in the intra-personal game)}} For any $(t,x)\in[0,T)\times \R^d$ and $\pi'\in\mathcal{A}$,
		\begin{equation}\label{eq:def.equipi}
			\limsup_{\epsilon\to 0^+} \frac{J^{\pi'\otimes_{t,\epsilon} \pi^*, m^*}(t,x)- J^{\pi^*, m^*}(t,x)}{\epsilon}\leq 0,
		\end{equation}
		where
		\begin{equation}\label{eq:def.pipast}
			\pi'\otimes_{t,\epsilon} \pi^*(s,x)=\begin{cases}
				\pi'(s,x), & (s,x)\in[t,t+\epsilon]\times\R^d,\\
				\pi^*(s,x), & otherwise. 
			\end{cases}
		\end{equation}

		\item[(b)]{\textit{(Consistency condition on population aggregation)}} $m^*_t=\operatorname{law}(X^{*}_t)$ for all $t\in[0,T]$, where
		\begin{equation}\label{eq:def.equim}
			dX^{*}_t=\tilde{b}(t,X^{*}_t,m^*_t,\pi^*(t,X^{*}_t))dt+\sigma(t,X^{*}_t,m^*_t)dW_t,\quad X^*_0\sim\nu.
		\end{equation}
	\end{itemize}
\end{Definition}

{\begin{Remark}\label{rmk:comparison_literature}
We summarize some key differences between our study and \cite{MeiZhu2022}. Unlike \cite{MeiZhu2022}, which considers equilibrium strategies as classical feedback controls taking values in the action space $U$, we adopt the broader definition of relaxed feedback policies taking values in $\mathscr{P}(U)$. This choice aligns with the compactification approach widely used in the time-consistent mean field game literature and allows us to establish existence results under considerably weaker assumptions.

In \cite{MeiZhu2022}, the existence theory relies on two restrictive conditions. First, the drift and running cost functions must be separable in the form
\begin{equation*}
b(t,x,m,a) = b_1(t,x,m) + b_2(t,x,a),\qquad r(t,x,m,a) = r_1(t,x,m) + r_2(t,x,a).
\end{equation*}
Second, a globally Lipschitz continuous selection map $\psi$ for the Hamiltonian is required, given by
\begin{equation*}
\psi(t,x,q) := \operatorname{argmin}_{a \in U} \big\{ q \cdot b_2(t,x,a) + r_2(0,x,a) \big\}.
\end{equation*}
The existence of such a map generally requires strong structural conditions, typically involving strict convexity of the coefficients. Without these convexity assumptions, a classical equilibrium may fail to exist even in single-agent time-inconsistent control problems \cite{huang2021strong,bayraktar2025relaxed}.

To overcome these limitations, we employ an entropy regularization approach, which accommodates general conditions on the drift and reward functions and helps ensure the existence of equilibria under significantly weaker regularity conditions (see Theorem~\ref{thm:equi.existence}).
\end{Remark}}

\section{Existence of Equilibrium under Entropy Regularization}\label{sect:regexist}
This section studies equilibrium existence through the vanishing entropy regularization approach. We first analyze the mean field game under Shannon entropy regularization on relaxed controls. Specifically, for a population flow $m\in\mathscr{M}_{\nu}$ and an admissible relaxed policy $\pi:[0,T]\times\R^d\rightarrow\mathscr{P}_{c}(U)$, the regularized payoff functional is given by
\begin{align*}
	V^{\pi, m}_\lambda (t,s,x) &:= \mathbb{E}_{s,x}\left[\int_s^T  \left(\tilde{r}(l-t, X^{\pi}_l, m_l, \pi(l, X^{\pi}_l))+\delta(l-t)\lambda\mathcal{H}(\pi(l, X^{\pi}_l))\right) dl+ F(t,X^{\pi}_T,m_T)\right],\\
	J^{\pi, m}_\lambda(t,x) &:= V^{\pi, m}_\lambda(t,t,x),
\end{align*}
where $\delta:[0,T]\rightarrow\R$ is a  strictly positive discount function satisfying $\delta(0)=1$, $\lambda>0$ is the regularization parameter, and $\mathcal{H}$ denotes the Shannon entropy $\mathcal{H}(\varpi):= -\int_U \ln(\varpi(a))\varpi(a)da$ for a given density $\varpi\in \mathscr{P}_c(U)$.
{\begin{Definition}
	A relaxed feedback policy $\pi: [0,T] \times \R^d \rightarrow \mathscr{P}_c(U)$ is said to be admissible with respect to the flow $m$ for the regularized problem if for any $(t,x)\in[0,T)\times\R^d$, the controlled SDE
	\begin{equation*}
		dX^\pi_s = \tilde{b}(s, X^\pi_s, m_s, \pi(s, X^\pi_s)) ds + \sigma(s, X^\pi_s, m_s) dW_s\; \text{ for } s \in [t,T],\quad X^\pi_t = x,
	\end{equation*}
	admits a unique strong solution, and $J^{\pi,m}_\lambda(t,x)$ is finite. The set of all such admissible policies is denoted by $\Ac_{c}$.
\end{Definition}}

As in Definition \ref{def:equi.relaxed}, we give the following definition of a regularized mean-field equilibrium.

\begin{Definition}\label{def:equi.reg}
	A pair $(\pi^*_\lambda, m^*_\lambda)$, where $\pi^*_\lambda\in\Ac_{c}$ and $m^*_\lambda\in\mathscr{M}_{\nu}$, is called a regularized mean-field equilibrium with entropy weight $\lambda>0$ if the following conditions are satisfied:
	\begin{itemize}
		\item[(a)]{\textit{(Equilibrium response in the intra-personal game)}}        For any $(t,x)\in[0,T)\times \R^d$ and $\pi'\in\mathcal{D}_{c}$,
		\begin{equation}\label{eq:def.equipi.reg}
			\limsup_{\epsilon\to 0^+} \frac{J^{\pi'\otimes_{t,\epsilon} \pi^*_\lambda, m^*_\lambda}_\lambda(t,x)- J^{\pi^*_\lambda, m^*_\lambda}_\lambda(t,x)}{\epsilon}\leq 0,
		\end{equation}
		where $\mathcal{D}_{c}$ denotes the set of all alternative policies $\pi': [0,T] \times \mathbb{R}^d \to \mathscr{P}_c(U)$ such that the perturbed policy $\pi'\otimes_{t,\epsilon} \pi^*_\lambda$ remains in $\mathcal{A}_{c}$ for sufficiently small $\epsilon > 0$.

		\item[(b)]{\textit{(Consistency condition on population aggregation)}} $m^*_{\lambda, t}=\operatorname{law}(X^{*,\lambda}_t)$ for all $t\in[0,T]$, where
		\begin{equation}\label{eq:def.equim.reg}
			dX^{*,\lambda}_t=\tilde{b}(t,X^{*,\lambda}_t,m^*_{\lambda, t},\pi^*_\lambda(t,X^{*,\lambda}_t))dt+\sigma(t,X^{*,\lambda}_t,m^*_{\lambda, t})dW_t,\quad X^{*,\lambda}_0\sim\nu.
		\end{equation}
	\end{itemize}
\end{Definition}
	\begin{Assumption}\label{assume.lipsa.U}
	The functions $a\rightarrow b(t,x,m,a), r(t,x,m,a)$ are uniformly Lipschitz, i.e.,
	$$
	\Theta:=\sup_{(t,x,m)\in [0,T]\times \R^d\times\Pc_2(\R^d)} \frac{|b(t,x,m,a)-b(t,x,m,a')|+|r(t,x, m,a)-r(t,x,m,a')|}{|a-a'|}<\infty.
	$$
	The action space $U$ satisfies a uniform inner-cone test condition: when $\ell>1$, there exist $\zeta>0$ and $\gamma\in (0,\pi/2]$ such that, for any $a\in U$, one can find a cone with vertex $a$ and angle $\gamma$ (i.e., $\cone(a,\gamma)$) satisfying $\cone(a,\gamma)\cap B_\zeta(a)\subset U$. When $\ell=1$, there exists $\zeta>0$ such that, for any $a\in U$, either $[a-\zeta, a]$ or $[a, a+\zeta]$ is contained in $U$.
\end{Assumption}
\begin{Theorem}\label{thm:existence_regu}
	Let Assumption \ref{assume.r} hold, and let $\delta\in \Cc^1_{\alpha}([0,T])$ be a strictly positive function satisfying $\delta(0)=1$.
	\begin{itemize}
		\item[(i)]  If $u \in \mathcal{C}^{1,2}_{ [0,T]}$ and $m\in\mathscr{M}_{\nu}$ satisfy
		\begin{align}
			\partial_s u (t,s, x)&+\frac{1}{2}\operatorname{tr}\left((\sigma\sigma^T)(s,x,m_s) D^2_x u (t,s, x)\right) +\tilde{b}(s,x,m_s, \pi^*(s,x))\cdot D_x u(t,s, x) \notag \\
			& +\tilde{r}(s-t,x, m_s,\pi^*(s,x))+\lambda \delta(s-t)\mathcal{H}(\pi^*(s,x)) =0 \quad (t,s,x)\in \Delta_{[0,T]}\times \R^d, \label{eq:cha.HJBentropy_prime}\\
			u(t,T,x)&=F(t,x,m_T),\nonumber\\
			m_t&=\operatorname{law}(X^{\pi^*,m}_t) \quad \forall t\in[0,T],\label{eq:cha.m}
		\end{align}
		where
		\begin{equation}\label{eq:cha.pi}
			\begin{aligned}
	&\pi^*(t,x,a)=\\
&\frac{\exp\left\{\frac{1}{\lambda}\left[b(t,x,m_t,a)\cdot D_{x}u(t,t,x) +r(0,x,m_t,a)\right]\right\}}{\int_{U}\exp\left\{\frac{1}{\lambda}\left[b(t,x,m_t,a')\cdot D_{x}u(t,t,x) +r(0,x,m_t,a')\right]\right\}da'},\quad\forall(t,x,a)\in[0,T]\times\R^d\times U
			\end{aligned}
		\end{equation}
		and
		\begin{equation*}
			dX^{\pi^*,m}_t=\tilde{b}(t,X^{\pi^*,m}_t,m_t, \pi^*(t,X^{\pi^*,m}_t))dt+\sigma(t,X^{\pi^*,m}_t,m_t)dW_t,\quad X^{\pi^*,m}_0\sim\nu.
		\end{equation*}
        Furthermore, suppose that $D_x u(t,s,x)$ is Lipschitz continuous in $t$ uniformly in the other arguments, i.e., there exists a constant $C>0$ such that
        \begin{equation}\label{eq1:lm.lipschitz_t}
            |D_x u(t_1,s,x) - D_x u(t_2,s,x)| \leq C|t_1 - t_2|,
        \end{equation}
        for all $t_1, t_2 \in [0,T]$ and $(s,x) \in[0,T]\times\R^d$ with $s \geq t_1 \vee t_2$.
		Then $u(t,s,x)=V^{\pi^*, m}_\lambda(t,s,x)$ and $(\pi^*, m)$ is a regularized equilibrium with entropy weight $\lambda$.

		\item[(ii)] Additionally, let Assumption \ref{assume.lipsa.U} hold and let  $\nu\in\mathscr{P}_{2+\kappa}(\R^d)$. There exists a finite constant $\lambda_0>0$ such that for all entropy parameters $0<\lambda\leq \lambda_0$, a regularized equilibrium $(\pi^*_\lambda, m^*_\lambda)$ with entropy parameter $\lambda$ exists. Moreover, the following estimates also hold:
		\begin{equation}\label{eq:thm.entropy}
			\begin{aligned}
				\|V^{\pi^*_\lambda, m^*_\lambda}_\lambda\|_{\widetilde{\Cc}^{0, 1}_{\alpha,[0, T]}} \vee \|V^{\pi^*_\lambda, m^*_\lambda}_\lambda \|_{\widehat{W}^{1,2}_{p,[0, T]}}  \leq A^*,\quad\|V^{\pi^*_\lambda,m^*_\lambda}_\lambda\|_{\widetilde{\Cc}^{1, 2}_{\alpha,[0, T]}}\leq A_\lambda,
			\end{aligned}
		\end{equation}
		where $A^*$ is a constant depending on the constants $K_1,  K_2, \eta$, $\Theta, \zeta, \gamma$ in Assumptions \ref{assume.r} and \ref{assume.lipsa.U}, but independent of $\lambda$, and $A_\lambda$ is another constant that further depends on $\lambda$.
	\end{itemize}
\end{Theorem}
\subsection{Proof of Theorem \ref{thm:existence_regu}}

\begin{proof}[{\bf Proof of Theorem \ref{thm:existence_regu} Part (i):}]
Suppose that $(u, m,\pi^*)$ with $u \in \mathcal{C}^{1,2}_{[0,T]}$ satisfies the system \eqref{eq:cha.HJBentropy_prime}-\eqref{eq:cha.pi}. We first show that $V^{\pi^*,m}_\lambda(t,s,x)=u(t,s,x)$ for any $(t,s,x) \in \Delta_{[0,T]} \times \mathbb{R}^d$.

Since $u \in \mathcal{C}^{1,2}_{[0,T]}$, for any fixed $t \in [0,T]$, the function $(s,x) \mapsto u(t,s,x)$ belongs to $\mathcal{C}^{1,2}([t,T] \times \mathbb{R}^d)$. {{Combining this with \eqref{eq:cha.pi}, $m\in \mathscr{M}_\nu$, and the conditions on $b$ and $\sigma$ in Assumption \ref{assume.r}}, we deduce that, for any initial point $(s,x)\in[t,T]\times\R^d$, the SDE
\begin{equation*}
	dX^{\pi^*}_l= \tilde{b}(l,X^{\pi^*}_l,m_l,\pi^*(l,X^{\pi^*}_l))dl+\sigma(l,X^{\pi^*}_l,m_l)dW_l,\quad X^{\pi^*}_s=x, \quad l \in [s,T],
\end{equation*}
admits a unique strong solution (see, e.g., \cite{veretennikov1981strong}).} Take $(t,s)\in \Delta_{[0,T]}$. Applying It\^o's formula to $u(t,l, X^{\pi^*}_l)$ yields
\begin{align*}
	&du(t,l, X^{\pi^*}_l)\\
	&= \left(\partial_l u(t, l, X^{\pi^*}_l)+ \tilde{b}(l,X^{\pi^*}_l,m_l,\pi^*(l,X^{\pi^*}_l))\cdot D_x u(t,l,X^{\pi^*}_l)+\frac{1}{2}\operatorname{tr}\bigg((\sigma\sigma^T)(l,X^{\pi^*}_l,m_l) D^2_x u(t,l,X^{\pi^*}_l) \bigg)\right) dl\\
	&\qquad+ \sigma(l,X^{\pi^*}_l,m_l)D_xu(t,l,X^{\pi^*}_l)dW_l.
\end{align*}
Integrating from $s$ to $T$, taking expectations on both sides, and using the fact that $u$ satisfies the PDE \eqref{eq:cha.HJBentropy_prime} along with the terminal condition $u(t,T,x)=F(t,x,m_T)$, we obtain
\begin{equation*}
    u(t,s,x)=\mathbb{E}_{s,x}\left[\int_s^T  \left(\tilde{r}(l-t, X^{\pi^*}_l, m_l, \pi^*(l, X^{\pi^*}_l))+\delta(l-t)\lambda\mathcal{H}(\pi^*(l, X^{\pi^*}_l))\right) dl+ F(t,X^{\pi^*}_T,m_T)\right],
\end{equation*}
which is exactly $V^{\pi^*,m}_\lambda(t,s,x)$.

To show that $(\pi^*, m)$ is a regularized equilibrium, it suffices to verify that \eqref{eq:def.equipi.reg} holds. We begin by defining the following functions for any $(t,s,x) \in \Delta_{[0,T]} \times \mathbb{R}^d$:
\begin{align*}
    H_{\lambda}(t,s,x) &:= \sup_{\varpi \in \mathscr{P}_c(U)} \bigg\{ \int_U \bigg( b(s,x,m_s, a) \cdot D_x V^{\pi^*,m}_\lambda(t,s,x) \\
    &\qquad\qquad \qquad \qquad + r(s-t,x,m_s,a) -\lambda\delta(s-t)\log(\varpi(a))\bigg) \varpi(a)da \bigg\}, \\
    L_{\lambda}(t,s,x) &:= \tilde{b}(s,x, m_s,\pi^*(s,x)) \cdot D_x V^{\pi^*,m}_\lambda(t,s,x) + \tilde{r}(s-t,x,m_s,\pi^*(s,x))+\lambda\delta(s-t)\Hc(\pi^*(s,x)).
\end{align*}
We first observe that there exists a finite constant $C > 0$, independent of $t$ and $s$, such that
\begin{equation}\label{eq:origin'.thm1}
    \sup_{x\in \mathbb{R}^d} |H_{\lambda}(t,s,x) - H_{\lambda}(s,s,x)| \leq C|t-s|^{\frac{\alpha}{2}}, \quad \forall |t-s| \leq 1.
\end{equation}
Indeed, for any fixed $(t,s,x)\in\Delta_{[0,T]}\times \mathbb{R}^d$, the supremum in the definition of $H_\lambda(t,s,x)$ is uniquely attained by the Gibbs measure $\tilde{\varpi}$ with the density:
\begin{equation*}
    \tilde{\varpi}(t,s,x,a) = \frac{\exp\left\{ \frac{1}{\lambda\delta(s-t)} \left[ b(s,x,m_s,a) \cdot D_x V^{\pi^*,m}_\lambda(t,s,x) + r(s-t,x,m_s,a) \right] \right\}}{\int_U \exp\left\{ \frac{1}{\lambda\delta(s-t)} \left[ b(s,x,m_s,a') \cdot D_x V^{\pi^*,m}_\lambda(t,s,x) + r(s-t,x,m_s,a') \right] \right\} da'}.
\end{equation*}
Because $V^{\pi^*,m}_\lambda = u \in \Cc^{1,2}_{[0,T]}$, Assumption \ref{assume.r}, and the strict positivity of $\delta$ on $[0,T]$ hold, the Shannon entropy $\mathcal{H}(\tilde{\varpi}(t,s,x))$ is bounded by a finite constant independent of $(t,s,x)$.
Then,
\begin{align*}
    H_{\lambda}(t,s,x) &= \tilde{b}(s,x, m_s,\tilde{\varpi}(t,s,x)) \cdot  D_x V^{\pi^*,m}_\lambda(t,s,x) + \tilde{r}(s-t,x,m_s,\tilde{\varpi}(t,s,x)) + \lambda\delta(s-t)\mathcal{H}(\tilde{\varpi}(t,s,x)) \\
    &\leq \tilde{b}(s,x, m_s,\tilde{\varpi}(t,s,x)) \cdot  D_x V^{\pi^*,m}_\lambda(s,s,x) + \tilde{r}(0,x,m_s,\tilde{\varpi}(t,s,x)) + \lambda\mathcal{H}(\tilde{\varpi}(t,s,x)) + C|t-s|^{\frac{\alpha}{2}} \\
    &\leq H_{\lambda}(s,s,x) + C|t-s|^{\frac{\alpha}{2}},
\end{align*}
where the first inequality follows from \eqref{eq1:lm.lipschitz_t}, 
the H\"older continuity of $r(\cdot,x,m,a)$ with respect to its first argument, and the regularity of $\delta$. By applying a symmetric argument to $H_{\lambda}(s,s,x) - H_{\lambda}(t,s,x)$, we obtain \eqref{eq:origin'.thm1}. Similarly, we have that
\begin{equation}\label{eq:origin'.thm3}
    |L_{\lambda}(t,s,x) - L_{\lambda}(s,s,x)| \leq C|t-s|^{\alpha/2}, \quad \forall |t-s| \leq 1.
\end{equation}
Combining \eqref{eq:origin'.thm1} and \eqref{eq:origin'.thm3} yields, for any $t \in [0,T)$, that
\begin{align}
    L_{\lambda}(t,s,x) - H_{\lambda}(t,s,x) &= (L_{\lambda}(t,s,x) - L_{\lambda}(s,s,x)) + (L_{\lambda}(s,s,x) - H_{\lambda}(s,s,x)) + (H_{\lambda}(s,s,x) - H_{\lambda}(t,s,x)) \notag\\
    &\geq L_{\lambda}(s,s,x) - H_{\lambda}(s,s,x) - 2C|t-s|^{\alpha/2} \label{eq:origin'.thm5}\\
    &\geq -2C|t-s|^{\alpha/2}, \quad\forall (s,x) \in [t,T] \times \mathbb{R}^d, \notag
\end{align}
where the last inequality follows from the fact that the Gibbs measure $\pi^*(s,x)$ defined in \eqref{eq:cha.pi} exactly achieves the maximum of $H_{\lambda}(s,s,x)$.

Now take an arbitrary $\pi' \in \mathcal{D}_{c}$ and $(t,x) \in [0,T) \times \mathbb{R}^d$. For any $\epsilon > 0$, it holds that
\begin{align*}
J^{\pi'\otimes_{t,\epsilon} \pi^*, m}_\lambda(t,x) &= \mathbb{E}_{t,x}\left[\int_t^{t+\epsilon} \left(\tilde{r}(l-t, X^{\pi'}_l, m_l, \pi'(l, X^{\pi'}_l)) + \delta(l-t)\lambda\mathcal{H}(\pi'(l, X^{\pi'}_l))\right) dl \right. \\
&\qquad \left. + V^{\pi^*,m}_\lambda\left(t, t+\epsilon, X^{\pi'}_{t+\epsilon} \right)\right].
\end{align*}
Subtracting $J^{\pi^*, m}_\lambda(t,x) = V^{\pi^*,m}_\lambda(t,t,x)$ and applying It\^o's formula to $V^{\pi^*,m}_\lambda$ on the interval $[t, t+\epsilon]$, we obtain
\begin{align*}
	&J^{\pi'\otimes_{t,\epsilon} \pi^*, m}_\lambda(t,x) - J^{\pi^*, m}_\lambda(t,x) \\
    &= \mathbb{E}_{t,x}\bigg[\int_t^{t+\epsilon} \bigg( \tilde{r}(l-t, X^{\pi'}_l, m_l, \pi'(l, X^{\pi'}_l)) + \lambda\delta(l-t)\mathcal{H}(\pi'(l, X^{\pi'}_l)) \\
    &\qquad\qquad\qquad\quad + \partial_s V^{\pi^*,m}_\lambda\left(t, l, X^{\pi'}_{l} \right) + \tilde{b}(l,X^{\pi'}_{l},m_l, \pi'(l,X^{\pi'}_{l})) \cdot D_x V^{\pi^*,m}_\lambda(t,l,X^{\pi'}_{l}) \\
    &\qquad\qquad\qquad\quad + \frac{1}{2}\operatorname{tr}\left( \sigma\sigma^T(l,X^{\pi'}_{l},m_l) D_x^2 V^{\pi^*,m}_\lambda(t,l,X^{\pi'}_{l}) \right) \bigg) dl\bigg].
\end{align*}
By \eqref{eq:cha.HJBentropy_prime}, the integrand can be dominated by the difference $H_{\lambda} - L_{\lambda}$ as follows:
\begin{align*}
    &J^{\pi'\otimes_{t,\epsilon} \pi^*, m}_\lambda(t,x) - J^{\pi^*, m}_\lambda(t,x) \\
    &= \mathbb{E}_{t,x}\bigg[\int_t^{t+\epsilon} \bigg( \Big[ \tilde{r}(l-t, X^{\pi'}_l, m_l, \pi'(l, X^{\pi'}_l)) + \lambda\delta(l-t)\mathcal{H}(\pi'(l, X^{\pi'}_l)) \\
    &\qquad\qquad\qquad\quad + \tilde{b}(l,X^{\pi'}_{l},m_l, \pi'(l,X^{\pi'}_{l})) \cdot D_x V^{\pi^*,m}_\lambda(t,l,X^{\pi'}_{l}) \Big] \\
    &\qquad\qquad\qquad \quad - \Big[ \tilde{r}(l-t, X^{\pi'}_l, m_l, \pi^*(l, X^{\pi'}_l)) + \lambda\delta(l-t)\mathcal{H}(\pi^*(l, X^{\pi'}_l)) \\
    &\qquad\qquad\qquad\quad + \tilde{b}(l,X^{\pi'}_{l},m_l, \pi^*(l,X^{\pi'}_{l})) \cdot D_x V^{\pi^*,m}_\lambda(t,l,X^{\pi'}_{l}) \Big] \bigg) dl\bigg] \\
    &\leq \mathbb{E}_{t,x}\bigg[\int_t^{t+\epsilon} \bigg( H_{\lambda}(t,l,X^{\pi'}_{l}) - L_{\lambda}(t,l,X^{\pi'}_{l}) \bigg) dl\bigg]\\
    &\leq \mathbb{E}_{t,x}\bigg[\int_t^{t+\epsilon} 2C|l-t|^{\alpha/2} dl\bigg] = \frac{2C}{1+\alpha/2}\epsilon^{1+\alpha/2},
\end{align*}
where the last inequality follows from \eqref{eq:origin'.thm5}.
Dividing both sides by $\epsilon$ and letting $\epsilon \to 0^+$, the right-hand side vanishes.
Consequently, by the arbitrariness of $(t,x)$ and $\pi' \in \mathcal{D}_c$, the equilibrium condition \eqref{eq:def.equipi.reg} is satisfied for $\pi^*$, which completes the proof of Theorem \ref{thm:existence_regu} Part (i).
\end{proof}

For $\lambda>0$, define $\Gamma_{\lambda}: [0,T]\times \R^d\times \R^d\times \mathscr{P}_2(\R^d)\times U \rightarrow \R_+$ by
\begin{equation*}
	\Gamma_{\lambda}(t, x, p, m, a) := \frac{\exp\left( \frac{1}{\lambda}\left[b(t,x,m,a)\cdot p +r(0,x,m, a) \right]\right)}{\int_U \exp\left( \frac{1}{\lambda}\left[b(t,x,m,a')\cdot p+r(0,x,m,a') \right]\right) da'}.
\end{equation*}
We define $\mu^{\pi,m}$ as the flow induced when all agents follow $\pi$ starting from $\nu$, i.e., $\mu^{\pi,m}_t=\operatorname{law}(X^{\pi,m}_t)$ for all $t\in [0,T]$, where
\begin{equation}\label{mupim}
	dX^{\pi,m}_{t}=\tilde{b}(t,X^{\pi,m}_{t},m_t,\pi(t,X^{\pi,m}_{t}))dt+\sigma(t,X^{\pi,m}_{t},m_t)dW_t,\quad X^{\pi,m}_{0}\sim\nu.
\end{equation}
We then define the operator
\begin{equation}\label{eq:Phi.fixedpoint}
	\Phi_\lambda(w,m):= (v^{\pi, m}, \mu^{\pi,m}),
\end{equation}
where $\pi(t,x,a)=\Gamma_\lambda(t,x,D_xw(t,x),m_t,a)$ and
\begin{equation*}
	v^{\pi,m}(t,x)=V^{\pi,m}_{\lambda}(t,t,x).
\end{equation*}
We prove the existence of an equilibrium by showing that a fixed point exists for the operator $\Phi_\lambda$.

The following lemma is adapted from \cite[Lemma 1]{bayraktar2025relaxed}, 
 and provides a sublinear growth estimate for the entropy term of the Gibbs-form operator $\Gamma_{\lambda}(t, x, y, m, a)$ in terms of $|y|$.
	\begin{Lemma}\label{lm:entropy.ln}
		Suppose that Assumption \ref{assume.lipsa.U} holds.
		There exists a finite constant $\lambda_0>0$ such that the following estimate holds for all $0<\lambda\leq \lambda_0$,
		\bee
		|\Hc(\Gamma_\lambda(t, x,y,m, a))|=\left| \int_U \ln(\Gamma_{\lambda}(t,x,y,m,a))\Gamma_{\lambda}(t,x,y,m,a)da\right|\leq K_3+K_4|\ln\lambda|+K_5 \ln(1+|y|).
		\eee
		where $K_3, K_4, K_5$ are positive constants that depend on $\ell$, $\Leb(U)$ and the constants $\Theta$, $\zeta$ and $\gamma$ in Assumption \ref{assume.lipsa.U}, but are independent of $\lambda$.
	\end{Lemma}
For $w\in \Cc^{0,1}_{\alpha}([0,T]\times\R^d)$ with $\|w\|_{\Cc^{0,1}_{\alpha}([0,T]\times\R^d)}\leq M$ and $m\in\mathscr{M}_{\nu}^{\kappa,C}$ for some $C>0$, we derive H\"older norm estimates for key quantities induced by the policy
\begin{align}\label{eq:def.Gammalambda}
    \pi(t,x,a)=\Gamma_\lambda(t,x,D_xw(t,x),m_t,a).
\end{align}

\begin{Lemma}\label{lm:pi.est}
	Let Assumption \ref{assume.r} hold. Given $t\in[0,T)$, suppose $\|w\|_{\Cc^{0,1}_{\alpha}([t,T]\times\R^d)}\leq  M_t$ and
	\begin{equation*}
	    \sup_{0\leq s_1\neq s_2\leq T}\mathcal{W}_2(m_{s_1},m_{s_2})\leq \sqrt{C}|s_1-s_2|^{1/2}
	\end{equation*}
	for finite constants $M_t\geq 1, C\geq 1$, then the following estimates hold for the policy $\pi$ defined in \eqref{eq:def.Gammalambda}:
\begin{align*}
		&\|\lambda \mathcal{H}(\pi(\cdot,\cdot))\|_{\Cc_{\alpha} ([t,T]\times\R^d)} \leq A_0\sqrt{C}\left(M_t+\frac{M^2_t}{\lambda}\right)\exp\left(\frac{A_0M_t\sqrt{C}}{\lambda}\right),\quad  \|\sigma(\cdot,\cdot,m_\cdot)\|_{\Cc_\alpha([0,T]\times\R^d)}\leq A_0\sqrt{C},\\
		&\|\tilde{b}(\cdot,\cdot,m_\cdot,\pi(\cdot,\cdot))\|_{\Cc_\alpha([t,T]\times\R^d)} \vee \|\tilde{r}(\cdot-t,\cdot,m_\cdot,\pi(\cdot,\cdot))\|_{\Cc_\alpha([t,T]\times\R^d)}\vee \|\partial_t \tilde{r}(\cdot-t,\cdot,m_\cdot,\pi(\cdot,\cdot))\|_{\Cc_\alpha([t,T]\times\R^d)} \\
		&\quad \leq\frac{A_0M_t\sqrt{C}}{\lambda}\exp\left(\frac{A_0 M_t}{\lambda}\right),
	\end{align*}
	where $A_0$ is a constant depending on $d$, $\Leb(U)$ and the constants $K_1, K_2$ in Assumption \ref{assume.r}, and is independent of $\lambda\leq \lambda_0$, $M_t$ and $C$. The functions $b, r, \sigma$ above are understood as functions of $(s,x)$ for a fixed $t$.
\end{Lemma}

\begin{proof}
Throughout this proof, let $(s_1,x), (s_2,y)\in[t,T]\times\R^d$ satisfy $|s_1-s_2|+|x-y|^2\leq 1$, and let $A_0$ be a generic positive constant depending on $d$, $\Leb(U)$ and the constants $K_1, K_2$ in Assumption \ref{assume.r}, but independent of $\lambda$, $M_t$ and $C$. This constant may vary from line to line.

Let $g(s,x, a):=\frac{1}{\lambda}[b(s,x,m_s,a)\cdot D_x w(s,x)+r(0,x, m_s, a)]$. A direct calculation gives that
\begin{align*}
    |g(s_1,x,a)-g(s_2,y,a)|&\leq \frac{1}{\lambda}\Big[|b(s_1,x,m_{s_1},a)-b(s_2,y,m_{s_2},a)||D_{x}w(s_1,x)|\\
    &\quad +|b(s_2,y,m_{s_2},a)||D_{x}w(s_1,x)-D_xw(s_2,y)|
     +|r(0,x,m_{s_1},a)-r(0,y,m_{s_2},a)|\Big]\\
    &\leq \frac{A_0 M_t}{\lambda}\left[(|s_1-s_2|+|x-y|^2)^{\frac{\alpha}{2}}+\mathcal{W}_2(m_{s_1},m_{s_2})\right]\\
    &\leq \frac{A_0M_t\sqrt{C}}{\lambda}(|s_1-s_2|+|x-y|^2)^{\frac{\alpha}{2}},
\end{align*}
and $ \sup_{a\in U} \|g(\cdot,\cdot,a)\|_{\Cc^0([t,T]\times \R^d)}\leq \frac{A_0 M_t}{\lambda}$,
which implies that
\begin{equation}\label{eq:lm.piest.0}
 \sup_{a\in U}\|g(\cdot,\cdot,a)\|_{\Cc_{\alpha}([t,T]\times\R^d)}\leq \frac{A_0 M_t\sqrt{C}}{\lambda},
\end{equation}
and
\begin{equation}\label{eq:lm.piest.1}
\exp\left(\frac{-A_0 M_t}{\lambda}\right) \leq \pi(s,x,a) \leq \exp\left(\frac{A_0 M_t}{\lambda}\right), \quad \forall(s,x,a)\in[t,T]\times \R^d\times U.
\end{equation}
This further yields that
\begin{equation}\label{eq:lm.piest.3}
\sup_{a\in U}[\exp(g(\cdot,\cdot, a))]_{\Cc_{\alpha}([t,T]\times\R^{d})}\leq \frac{A_0M_t\sqrt{C}}{\lambda}\exp\left(\frac{A_0 M_t}{\lambda}\right).
\end{equation}
Note that
\begin{align*}
&|\pi(s_1,x,a)-\pi(s_2,y,a)|\\
&=\left| \frac{  \exp[g(s_1,x,a)]\int_U \exp[g(s_2,y,a')]da' -   \exp[g(s_2,y,a)]\int_U \exp[g(s_1,x,a')]da' }{   \int_U \exp[g(s_1,x,a')]da'\int_U \exp[g(s_2,y,a')]da' } \right|\\
&\leq \left|  \frac{ \int_U \left( \exp[g(s_2,y,a')]-\exp[g(s_1,x, a')]\right)da'  }{\int_U \exp[g(s_2,y,a')]da'} \right| \pi(s_1,x,a) + \left|  \frac{ \exp[g(s_1,x,a)]-\exp[g(s_2,y, a)]  }{\int_U \exp[g(s_2,y,a')]da'} \right|.
\end{align*}
This, together with \eqref{eq:lm.piest.1} and \eqref{eq:lm.piest.3}, yields that
\begin{equation}\label{eq:lm.piest.5}
\sup_{a\in U}[\pi(\cdot,\cdot, a)]_{\Cc_{\alpha}([t,T]\times\R^d)}\leq \frac{A_0M_t\sqrt{C}}{\lambda}\exp\left(\frac{A_0 M_t}{\lambda}\right).
\end{equation}
By \eqref{eq:lm.piest.1} again and the fact that $\pi(s,x, a)$ is a probability density on $U$ for each $(s,x)$, we deduce that
\begin{equation*}
\left|\int_U \ln(\pi(s,x,a))\pi(s,x,a)da \right|\leq \frac{A_0 M_t}{\lambda}.
\end{equation*}
Also, it holds that
\begin{equation}\label{eq:lm.piest.5_prime}
\begin{aligned}
&\left| \int_U \ln(\pi(s_1,x,a))\pi(s_1,x, a)da  - \int_U \ln(\pi(s_2,y,a))\pi(s_2,y, a)da  \right|\\
&\leq \sup_{a\in U} \left|   \ln(\pi(s_1,x,a))  -  \ln(\pi(s_2,y,a)) \right| +\sup_{a\in U} \left| \frac{\pi(s_2,y,a)}{\pi(s_1,x,a)} - 1 \right| |\ln(\pi(s_2,y,a))|.
\end{aligned}
\end{equation}
In view of \eqref{eq:lm.piest.0}, the first term on the right-hand side above is estimated by
\begin{align*}
&\left|   \ln(\pi(s_1,x,a))  -  \ln(\pi(s_2,y,a)) \right|  \\
&\leq |g(s_1,x,a)-g(s_2,y,a)|+\left|\ln \left( \int_U \exp[g(s_1,x, a')]da'\right) -\ln \left( \int_U \exp[g(s_2,y,a')]da'\right) \right|\\
&=|g(s_1,x,a)-g(s_2,y,a)|+\left|\ln \left( \int_U \exp[g(s_1,x, a')-g(s_2,y,a')]\pi(s_2,y,a')da'\right)  \right|\\
&\leq \frac{A_0M_t\sqrt{C}}{\lambda} (|s_1-s_2|+|x-y|^{2})^{\frac{\alpha}{2}}.
\end{align*}
To estimate the second term on the right-hand side of \eqref{eq:lm.piest.5_prime}, assume without loss of generality that $\pi(s_2,y,a)\geq \pi(s_1,x,a)$. Then,
\begin{align*}
 \left| \frac{\pi(s_2,y,a)}{\pi(s_1,x,a)} - 1 \right|  &\leq \frac{\exp[g(s_2,y,a)]\int_U \exp\left[g(s_2,y,a')  + [g(\cdot,\cdot, a)]_{\Cc_{\alpha}([t,T]\times\R^d)} (|s_1-s_2|+|x-y|^{2})^{\frac{\alpha}{2}} \right] da' }{  \int_U \exp[g(s_2,y,a')]da' \exp[g(s_1,x,a)]}  - 1\\
&=  \exp[g(s_2,y,a)-g(s_1,x,a)]\exp\left[  \sup_{a\in U} [g(\cdot,\cdot,a)]_{\Cc_{\alpha}([t,T]\times\R^d)} (|s_1-s_2|+|x-y|^{2})^{\frac{\alpha}{2}}  \right]-1\\
&\leq \frac{A_0M_t\sqrt{C}}{\lambda}\exp\left(\frac{A_0M_t\sqrt{C}}{\lambda}\right) (|s_1-s_2|+|x-y|^{2})^{\frac{\alpha}{2}}.
\end{align*}
Combining the above with \eqref{eq:lm.piest.1} and \eqref{eq:lm.piest.5_prime}, we conclude that
\begin{equation*}
\|\lambda \Hc(\pi(\cdot,\cdot)\|_{\Cc_{\alpha} ([t,T]\times\R^d)}\leq A_0\sqrt{C}\left(M_t+\frac{M^2_t}{\lambda}\right)\exp\left(\frac{A_0M_t\sqrt{C}}{\lambda}\right).
\end{equation*}
Using Assumption \ref{assume.r}, \eqref{eq:lm.piest.1}, and \eqref{eq:lm.piest.5}, we obtain
\begin{align*}
    &|\tilde{r}(s_1-t,x,m_{s_1},\pi(s_1,x))-\tilde{r}(s_2-t,y,m_{s_2},\pi(s_2,y))|\\
    &\leq |\tilde{r}(s_1-t,x,m_{s_1},\pi(s_1,x))-\tilde{r}(s_2-t,y,m_{s_2},\pi(s_1,x))|\\
    &\quad +|\tilde{r}(s_2-t,y,m_{s_2},\pi(s_1,x))-\tilde{r}(s_2-t,y,m_{s_2},\pi(s_2,y))|\\
    &\leq A_0\sqrt{C}(|s_1-s_2|+|x-y|^2)^{\frac{\alpha}{2}}+\frac{A_0M_t\sqrt{C}}{\lambda}\exp\left(\frac{A_0 M_t}{\lambda}\right)(|s_1-s_2|+|x-y|^2)^{\frac{\alpha}{2}}\\
    &\leq \frac{A_0M_t\sqrt{C}}{\lambda}\exp\left(\frac{A_0 M_t}{\lambda}\right)(|s_1-s_2|+|x-y|^2)^{\frac{\alpha}{2}}.
\end{align*}
Therefore, we conclude that
\begin{equation*}
	\|\tilde{r}(\cdot-t,\cdot,m_\cdot,\pi(\cdot,\cdot))\|_{\Cc_\alpha([t,T]\times\R^d)} \leq\frac{A_0M_t\sqrt{C}}{\lambda}\exp\left(\frac{A_0 M_t}{\lambda}\right).
\end{equation*}
Similarly, we establish that
\begin{align*}
	&\|\tilde{b}(\cdot,\cdot,m_\cdot,\pi(\cdot,\cdot))\|_{\Cc_\alpha([t,T]\times\R^d)}\vee{\|\partial_t\tilde{r}(\cdot-t,\cdot,m_\cdot,\pi(\cdot,\cdot))\|_{\Cc_\alpha([t,T]\times\R^d)} } \leq\frac{A_0M_t\sqrt{C}}{\lambda}\exp\left(\frac{A_0 M_t}{\lambda}\right),\\
    &\|\sigma(\cdot,\cdot,m_\cdot)\|_{\Cc_\alpha([0,T]\times\R^d)}\leq A_0\sqrt{C}.
\end{align*}
\end{proof}

\begin{Lemma}\label{lm:iteration.estimateentropy}
	Let Assumptions \ref{assume.r} and \ref{assume.lipsa.U} hold and suppose $\nu\in\mathscr{P}_{2+\kappa}(\R^d)$ and $\delta\in \Cc^1_{\alpha}([0,T])$. Fix $0<\lambda\leq\lambda_0$. For a given $(w,m)\in\Cc^{0,1}_{\alpha}([0,T]\times\R^d)\times\mathscr{M}_{\nu}^{\kappa,C}$, consider the relaxed policy
	\begin{equation*}
		\pi(t,x,a)=\Gamma_{\lambda}(t,x,D_xw(t,x),m_t,a).
	\end{equation*}
	Then $V^{\pi, m}_\lambda(t,s,x)$ is the unique classical solution to the following PDE:
	\begin{equation}\label{eq:prop.iteraPDE}
		\begin{aligned}
			\partial_s V^{\pi, m}_\lambda (t,s, x) &+ \frac{1}{2}\operatorname{tr}\left((\sigma\sigma^T)(s,x, m_s) D^2_x V^{\pi, m}_\lambda (t,s, x)\right) + \tilde{b}(s,x, m_s, \pi(s,x))\cdot D_x V^{\pi, m}_\lambda(t,s, x) \\
			&+ \tilde{r}(s-t,x, m_s,\pi(s,x)) + \lambda \delta(s-t)\mathcal{H}(\pi(s,x))=0,\\
			V^{\pi, m}_\lambda (t,T, x) &= F(t,x,m_T).
		\end{aligned}
	\end{equation}
	Moreover, there exist constants $A^*>0$, $C^*>0$ independent of $\lambda\leq \lambda_0$, such that if $\|w\|_{\Cc^{0,1}_\alpha([0,T]\times\R^d)}\leq A^*$ and $m\in\mathscr{M}_{\nu}^{\kappa,C^*}$, then
	\begin{align*}
		\|V^{\pi, m}_\lambda\|_{\widetilde{\Cc}^{0, 1}_{\alpha,[0, T]}}\vee \|V^{\pi, m}_\lambda \|_{\widehat{W}^{1,2}_{p,[0, T]}}\vee \|J^{\pi,m}_\lambda\|_{\Cc^{0, 1}_{\alpha}([0,T]\times\R^d)}
        \leq A^*, \quad
	\mu^{\pi,m}\in\mathscr{M}_{\nu}^{\kappa,C^*}.
	\end{align*}
	And there exists another constant $A_\lambda$ depending on $\lambda$ such that $\|V^{\pi, m}_\lambda\|_{\widetilde{\Cc}^{1, 2}_{\alpha,[0, T]}}\vee\|J_{\lambda}^{\pi,m}\|_{\Cc^{0, 2}_{\alpha}([0,T]\times\R^d)} \leq A_\lambda$.
\end{Lemma}
	\begin{proof}
Throughout the proof, let $A_0$ be a generic positive constant depending on $d, \alpha$ and the constants in Assumptions \ref{assume.r} and \ref{assume.lipsa.U}, but independent of $\lambda$ and the norms of all orders of derivatives of $w$. This constant may vary from line to line.

\textbf{Step 1.} We first show that $V^{\pi, m}_\lambda(t,s,x)$ is the unique classical solution to \eqref{eq:prop.iteraPDE}. From $\|w\|_{\Cc^{0,1}_\alpha([0,T]\times\R^d)} < \infty$, Assumption \ref{assume.r}, $\delta\in \Cc^{1}_{\alpha}([0,T])$, and Lemma \ref{lm:pi.est}, we deduce that, for any fixed $t\in[0,T]$,
\begin{equation*}
\begin{aligned}
	&\|\tilde{b}(\cdot,\cdot,m_\cdot,\pi(\cdot,\cdot))\|_{\Cc_\alpha([t,T]\times\R^d)} < \infty, \quad
	\|(\sigma\sigma^T)(\cdot,\cdot,m_\cdot)\|_{\Cc_\alpha([t,T]\times\R^d)} < \infty,\\
	&\|\tilde{r}(\cdot-t,\cdot,m_\cdot,\pi(\cdot,\cdot))\|_{\Cc_{\alpha}([t,T]\times \R^d)} < \infty, \quad
	\|\delta(\cdot-t)\lambda\mathcal{H}(\pi(\cdot,\cdot))\|_{\Cc_{\alpha}([t,T]\times \R^d)} < \infty.
\end{aligned}
\end{equation*}
It then follows from standard parabolic PDE theory that \eqref{eq:prop.iteraPDE} admits a unique classical solution $u(t,\cdot,\cdot)\in \Cc^{1,2}_{\alpha}([t,T]\times\R^d)$ for any fixed $t\in[0,T]$. The same argument as in the proof of Theorem \ref{thm:existence_regu} Part (i) yields that $u(t,s,x)=V^{\pi,m}_\lambda(t,s,x)$.

\textbf{Step 2.}
For any $\pi\in\Ac$ and $m\in\mathscr{M}_{\nu}$, consider the process $X^{\pi,m}$ evolving as \eqref{mupim}. By the uniform boundedness of $b$ and $\sigma$, we have, for any $0 \leq s < t \leq T$,
\begin{align*}
    \mathbb{E}\left[|X^{\pi,m}_{t}-X^{\pi,m}_{s}|^{2}\right] &\leq 2\mathbb{E}\left[ \left|\int_s^t \tilde{b}(u, X^{\pi,m}_{u}, m_u, \pi(u, X^{\pi,m}_{u})) du \right|^2 \right] + 2\mathbb{E}\left[ \left|\int_s^t \sigma(u, X^{\pi,m}_{u}, m_u) dW_u \right|^2 \right] \\
    &\leq 2\|b\|_\infty^2 (t-s)^2 + 2\|\sigma\|_\infty^2 (t-s) \leq A_0 |t-s|.
\end{align*}
Similarly, for the $(2+\kappa)$-th moment, the Burkholder-Davis-Gundy (BDG) inequality and the boundedness of the coefficients directly yield
\begin{equation*}
    \mathbb{E}\left[|X^{\pi,m}_{t}|^{2+\kappa}\right] \leq A_0 \left( \mathbb{E}[|X^{\pi,m}_0|^{2+\kappa}] + 1 \right).
\end{equation*}
As the initial distribution satisfies $\nu \in \mathscr{P}_{2+\kappa}(\mathbb{R}^d)$, the right-hand side is uniformly bounded. Consequently, we conclude that
\begin{equation}\label{eq:est.mu}
\begin{cases}
	\sup_{t\in[0,T]}\int_{\mathbb{R}^d}|x|^{2+\kappa}\mu^{\pi,m}_t(dx) = \sup_{t\in[0,T]}\mathbb{E}\left[|X^{\pi,m}_{t}|^{2+\kappa}\right] \leq A_0,\\
	\mathcal{W}_2^2(\mu^{\pi,m}_t,\mu^{\pi,m}_s) \leq \mathbb{E}\left[|X^{\pi,m}_{t}-X^{\pi,m}_{s}|^{2}\right] \leq A_0|t-s|,
\end{cases}
\end{equation}
which implies the existence of the uniform constant $C^*$.

\textbf{Step 3.} We now provide estimates on the solution to \eqref{eq:prop.iteraPDE}. Fix an arbitrary $t\in [0,T]$.
By Lemma \ref{lm:entropy.ln},
\begin{equation*}
	|\lambda\mathcal{H}(\pi(s,x))|\leq A_0\left( 1+\ln\left(1+|D_x w(s,x)|\right) \right),
\end{equation*}
which implies
\begin{equation*}
	\|V^{\pi, m}_\lambda(t,\cdot,\cdot)\|_{\Cc^0([t,T]\times\R^d)} \leq A_0 \left(1+\ln\left(1+\|D_x w\|_{L^\infty([t,T]\times\R^d)} \right) \right).
\end{equation*}
Recall $p=\frac{d+2}{1-\alpha}$. It follows from the Sobolev embedding (see, e.g., \cite[Lemma 3.3 in Chapter II]{ladyzhenskaia1968linear}) and  $W^{1,2}_p$ estimates with truncation arguments (see, e.g., \cite[Theorem 5.2.10]{krylov2008lectures}) that
\begin{align*}
	\|V^{\pi,m}_\lambda(t, \cdot,\cdot)\|_{\Cc^{0,1}_{\alpha}(D_T(s_0, x_0))}
	&\leq A_0 \|V^{\pi, m}_\lambda(t,\cdot,\cdot)\|_{W^{1, 2}_p(D_T(s_0, x_0))}\\
	&\leq  A_0\Bigg( \|V^{\pi, m}_\lambda(t,\cdot,\cdot)\|_{L^p(D_{T+1}(s_0,x_0))} \\
	&\qquad\quad + \|\tilde{r}(\cdot-t, \cdot, m_\cdot,\pi(\cdot,\cdot)) + \lambda \delta(\cdot-t) \mathcal{H}(\pi(\cdot,\cdot))\|_{L^p(D_{T+1}(s_0, x_0))} \\
	&\qquad\quad + \|F(t,\cdot,m_T)\|_{W^{2,p}(B_{T+1}(x_0))} \Bigg).
\end{align*}
We then deduce that
\begin{equation}\label{eq:lm.V1}
	\|V^{\pi,m}_\lambda(t, \cdot,\cdot)\|_{\Cc^{0,1}_{\alpha}(D_T(s_0, x_0))} \vee \|V^{\pi, m}_\lambda(t,\cdot,\cdot)\|_{W^{1, 2}_p(D_T(s_0, x_0))} \leq A_0\left( 1+\ln \left( 1+\|D_xw\|_{L^\infty([t,T]\times \R^d)} \right)\right).
\end{equation}
Taking the supremum over all $(s_0,x_0)\in[t,T]\times \R^d$ yields that
\begin{equation}\label{eq:lm.V3}
	\|V^{\pi,m}_\lambda(t, \cdot,\cdot)\|_{\Cc^{0,1}_{\alpha}([t,T]\times\R^d)} \vee \|V^{\pi, m}_\lambda(t,\cdot,\cdot)\|_{W^{1, 2,\ul}_p([t,T]\times\R^d)}\leq A_0\left( 1+\ln \left( 1+\|D_xw\|_{L^\infty([t,T]\times\R^d)} \right)\right).
\end{equation}
From the regularity of $\delta$ and the conditions on $r$ and $F$ in Assumption \ref{assume.r}, it is clear that for any $(t,s,x)$, the derivative $W^{\pi, m}_\lambda := \partial_t V^{\pi, m}_\lambda$ admits the following representation
\begin{align}
	W^{\pi, m}_\lambda(t,s,x) = \mathbb{E}_{s,x}\bigg[&\int_s^T -\left(\partial_t \tilde{r}(l-t, X^{\pi,m}_l, m_l, \pi(l, X^{\pi,m}_l))+\partial_t \delta(l-t)\lambda\mathcal{H}(\pi(l, X^{\pi,m}_l))\right) dl \notag \\
    &+ \partial_t F(t,X^{\pi,m}_T,m_T)\bigg]. \label{eq:prop.Vt}
\end{align}
As in Step 1, we conclude that 
$W^{\pi, m}_\lambda$ is the unique classical solution to the equation
\begin{equation}\label{eq:prop.PDEW}
\begin{aligned}
	 \partial_s W^{\pi, m}_\lambda (t,s, x)&+\frac{1}{2}\operatorname{tr}\left((\sigma\sigma^T)(s,x, m_s) D^2_x W^{\pi, m}_\lambda (t,s, x)\right)
	+\tilde{b}(s,x, m_s, \pi(s,x))\cdot D_x W^{\pi, m}_\lambda(t,s, x) \\
    &-\partial_t \tilde{r}(s-t,x, m_s, \pi(s,x))  -\lambda\partial_t\delta(s-t)\mathcal{H}(\pi(s,x))=0,\\
    W^{\pi, m}_\lambda (t,T, x)&= \partial_{t}F(t,x,m_T),
\end{aligned}
\end{equation}
with all mixed derivatives being consistent, e.g., $\partial_{x}\partial_{t}V^{\pi,m}_{\lambda}=\partial_{t}\partial_{x}V^{\pi,m}_{\lambda}$. Then, by an argument similar to that for \eqref{eq:lm.V3}, we obtain
\begin{equation*}\label{eq:lm.1_prime}
\|W^{\pi,m}_\lambda(t, \cdot,\cdot)\|_{\Cc^{0,1}_{\alpha}([t,T]\times\R^d)} \vee \|W^{\pi, m}_\lambda(t,\cdot,\cdot)\|_{W^{1, 2,\ul}_p([t,T]\times\R^d)}\leq A_0\left( 1+\ln \left( 1+\|D_xw\|_{L^\infty([t,T]\times\R^d)} \right)\right).
\end{equation*}
Combining these estimates, we obtain
\begin{equation}\label{est:1}
\|V^{\pi, m}_\lambda\|_{\widetilde{\Cc}^{0, 1}_{\alpha,[0, T]}}\vee \|V^{\pi, m}_\lambda \|_{\widehat{W}^{1,2}_{p,[0, T]}}\leq  A_0\left( 1+\ln \left(1+ \|w\|_{\Cc^{0,1}_{\alpha}([0,T]\times\R^d)} \right)\right).
\end{equation}
Next, we analyze the H\"older norm on the diagonal function $J^{\pi,m}_\lambda(t,x)= V^{\pi, m}_\lambda(t,t,x)$. Take $(s,x), (t,y)\in[0,T]\times\R^d$ satisfying $|s-t|+|x-y|^2\leq 1$. Without loss of generality, assume $t\leq s$. Then
\begin{align*}
\frac{\left| V^{\pi, m}_\lambda(s,s,x)- V^{\pi, m}_\lambda(t,t,y) \right|}{ (|s-t|+|x-y|^2)^{\alpha/2} } &\leq \frac{\left| V^{\pi, m}_\lambda(s,s,x)- V^{\pi, m}_\lambda(t,s,x) \right|}{ (|s-t|+|x-y|^2)^{\alpha/2} } + \frac{\left| V^{\pi, m}_\lambda(t,s,x)- V^{\pi, m}_\lambda(t,t,y) \right|}{ (|s-t|+|x-y|^2)^{\alpha/2} }\\
&\leq \sup_{t\in [0,T]} \|W^{\pi,m}_{\lambda}(t,\cdot,\cdot)\|_{\Cc^{0}([t,T]\times\R^d)} + \sup_{t\in[0,T]} [V^{\pi,m}_\lambda(t,\cdot,\cdot)]_{\Cc_{\alpha}([t,T]\times\R^d)}.
\end{align*}
Therefore, we conclude that
\begin{equation*}
    \|J^{\pi,m}_\lambda\|_{\Cc_{\alpha}([0,T]\times\R^d)}\leq \sup_{t\in [0,T]} \|W^{\pi,m}_{\lambda}(t,\cdot,\cdot)\|_{\Cc^{0}([t,T]\times\R^d)}+ \sup_{t\in [0,T]} \|V^{\pi,m}_\lambda(t,\cdot,\cdot)\|_{\Cc_{\alpha}([t,T]\times\R^d)}.
\end{equation*}
Similarly, by noting the consistency of mixed derivatives, we obtain
\begin{align}
     \|J^{\pi,m}_\lambda\|_{\Cc^{0,1}_{\alpha}([0,T]\times\R^d)} &\leq \sup_{t\in [0,T]} \|W^{\pi,m}_{\lambda}(t,\cdot,\cdot)\|_{\Cc^{0,1}([t,T]\times\R^d)}+ \sup_{t\in[0,T]} \|V^{\pi,m}_\lambda(t,\cdot,\cdot)\|_{\Cc^{0,1}_{\alpha}([t,T]\times\R^d)},\label{est:dig1}\\
      \|J^{\pi,m}_\lambda\|_{\Cc^{0,2}_{\alpha}([0,T]\times\R^d)} &\leq \sup_{t\in [0,T]} \|W^{\pi,m}_{\lambda}(t,\cdot,\cdot)\|_{\Cc^{0,2}([t,T]\times\R^d)}+ \sup_{t\in[0,T]} \|V^{\pi,m}_\lambda(t,\cdot,\cdot)\|_{\Cc^{0,2}_{\alpha}([t,T]\times\R^d)}\label{est:dig2}.
\end{align}
Then, it follows from \eqref{est:1} and \eqref{est:dig1} that
\begin{equation}\label{est:2}
\|V^{\pi, m}_\lambda\|_{\widetilde{\Cc}^{0, 1}_{\alpha,[0, T]}}\vee \|V^{\pi, m}_\lambda \|_{\widehat{W}^{1,2}_{p,[0, T]}}\vee \|J^{\pi,m}_\lambda\|_{\Cc^{0,1}_{\alpha}([0,T]\times\R^d)}\leq  A_0\left( 1+\ln \left( 1+\|w\|_{\Cc^{0,1}_{\alpha}([0,T]\times\R^d)} \right)\right).
\end{equation}

Note that the function $y\mapsto A_0(1+\ln (1+y))$, which maps $[0, \infty)$ to $[0, \infty)$, has sublinear growth. There exists a constant $A^*$ such that
\begin{equation*}
0 \leq A_0(1+\ln (1+y)) \leq A^*, \quad \text{whenever } 0 \leq y \leq A^*.
\end{equation*}
We can thus conclude from \eqref{est:2} that
\begin{equation*}
\|V^{\pi, m}_\lambda\|_{\widetilde{\Cc}^{0, 1}_{\alpha,[0, T]}}\vee \|V^{\pi, m}_\lambda \|_{\widehat{W}^{1,2}_{p,[0, T]}}\vee \|J^{\pi,m}_\lambda\|_{\Cc^{0,1}_{\alpha}([0,T]\times\R^d)}\leq A^*,\quad \text{provided that }\|w\|_{\Cc^{0,1}_{\alpha}([0,T]\times\R^d)} \leq A^*.
\end{equation*}
Moreover, by the conditions in Assumption \ref{assume.r}, Schauder estimate to \eqref{eq:prop.iteraPDE} and \eqref{eq:prop.PDEW} (see, e.g., \cite[Theorem 9.2.2]{krylov1996lectures}) yields that
\begin{align}
\|V^{\pi, m}_\lambda\|_{\widetilde{\Cc}^{1, 2}_{\alpha,[0, T]}} \leq \tilde{C}\sup_{t\in[0,T]}\bigg(&\|\tilde{r}(\cdot-t,\cdot,m_{\cdot},\pi(\cdot,\cdot))\|_{\Cc_{\alpha}\left([t,T]\times\R^d\right)} + \|\partial_t \tilde{r}(\cdot-t,\cdot,m_\cdot,\pi(\cdot,\cdot))\|_{\Cc_{\alpha}\left([t,T]\times\R^d\right)}  \notag \\
&+ \|\lambda \delta(\cdot-t) \mathcal{H}(\pi(\cdot,\cdot))\|_{\Cc_{\alpha}\left([t,T]\times\R^d\right)}
+\|\lambda \partial_t\delta(\cdot-t) \mathcal{H}(\pi(\cdot,\cdot))\|_{\Cc_{\alpha}\left([t,T]\times\R^d\right)}  \label{eq:V.est}  \\
&+ \|F(t,\cdot,m_{T})\|_{\Cc^{2}_{\alpha}(\R^d)}+  \|\partial_t F(t,\cdot,m_{T})\|_{\Cc^{2}_{\alpha}(\R^d)}\bigg), \notag
\end{align}
where $\tilde{C}$ is a constant depending only on $\|\tilde{b}(\cdot,\cdot,m_\cdot,\pi(\cdot,\cdot))\|_{\Cc_\alpha\left([0,T]\times\R^d\right)}$ and $\|\sigma(\cdot,\cdot,m_\cdot)\|_{\Cc_\alpha\left([0,T]\times\R^d\right)}$, such that $\tilde{C}$ can be bounded by $\varphi\left(\|\tilde{b}\|_{\Cc_\alpha([0,T]\times\R^d)},\|\sigma\|_{\Cc_\alpha([0,T]\times\R^d)}\right)$ with $\varphi(\cdot,\cdot):[0, \infty)\times[0,\infty) \rightarrow[0, \infty)$ being a fixed increasing function on the two arguments.  Therefore, by Lemma \ref{lm:pi.est}, \eqref{eq:est.mu} and \eqref{eq:V.est}, given $\|w\|_{\Cc^{0,1}_{\alpha}([0,T]\times\R^d)} \leq A^*$ and $m\in\mathscr{M}_{\nu}^{\kappa,C^*}$, it holds that
\begin{equation*}
\|V^{\pi, m}_\lambda\|_{\widetilde{\Cc}^{1, 2}_{\alpha,[0, T]}} \leq \varphi\left(\frac{A_0 A^*\sqrt{C^*}}{\lambda} \exp \left(\frac{A_0 A^*}{\lambda}\right),A_0\sqrt{C^*}\right) \frac{A_0 (A^*)^2\sqrt{C^*}}{\lambda} \exp \left(\frac{A_0 A^*\sqrt{C^*}}{\lambda}\right)=: A_\lambda,
\end{equation*}
which gives $\|V^{\pi, m}_\lambda\|_{\widetilde{\Cc}^{1, 2}_{\alpha,[0, T]}} \leq A_\lambda$. Then, by \eqref{est:dig2}, we also have $ \|J^{\pi,m}_\lambda\|_{\Cc^{0,2}_{\alpha}([0,T]\times\R^d)}\leq A_{\lambda}$, which completes the proof.
\end{proof}



Fix an arbitrary $0<\lambda\leq \lambda_0$ and define
\begin{equation}\label{compact-set}
\mathcal{M}_\lambda := \left\{ w\in \mathcal{C}^{0,2}_{\alpha}([0,T]\times\mathbb{R}^d) : \|w\|_{\mathcal{C}^{0,1}_{\alpha}([0,T]\times\mathbb{R}^d)}\leq A^*, \, \|w\|_{\mathcal{C}^{0,2}_{\alpha}([0,T]\times\mathbb{R}^d)}\leq A_\lambda \right\}.
\end{equation}
Note that $\mathcal{M}_\lambda$ is a convex subset of $\mathcal{C}^{0,2}_{\alpha}([0,T]\times\mathbb{R}^d)$. 

\begin{Lemma}\label{lm:Phi.continuous0}
Fix an arbitrary $\beta\in[0,\alpha)$. Take a sequence $(w^n, m^n)_{n\in \mathbb{N}\cup\{\infty\}} \subset \mathcal{M}_{\lambda} \times \mathscr{M}_{\nu}^{\kappa,C^*}$ such that
\begin{equation*}
    \lim_{n\rightarrow\infty} (\|w^n-w^\infty\|_{\Cc^{0,2,\wg}_{\beta}([0,T]\times\R^d)} + d(m^n, m^\infty)) = 0.
\end{equation*}
Set $\pi^n(t,x,a)=\Gamma_\lambda(t,x,D_x w^n(t,x), m^n_t, a)$, then for each $N\in \mathbb{N}$:
\begin{equation}\label{eq:lemma3.4}
\begin{aligned}
    &\lim_{n\rightarrow\infty} \Big(
    \|\tilde{b}^n - \tilde{b}^\infty\|_{\Cc_{\beta}(D_N)} + \|\sigma^n(\sigma^n)^\top - \sigma^\infty(\sigma^\infty)^\top\|_{\Cc_{\beta}(D_N)} \\
    &\qquad\quad + \sup_{t\in[0,T]} \|(\tilde{r}^n - \tilde{r}^\infty)(\cdot-t, \cdot)\|_{\Cc_{\beta}(D_N(t,0))} + \sup_{t\in[0,T]} \| (\partial_{t}\tilde{r}^n - \partial_{t}\tilde{r}^\infty) (\cdot-t, \cdot) \|_{\Cc_{\beta}(D_N(t,0))} \\
    &\qquad\quad  + \sup_{t\in[0,T]}\|\lambda\delta(\cdot-t)\mathcal{H}(\pi^n) - \lambda\delta(\cdot-t)\mathcal{H}(\pi^\infty)\|_{\Cc_{\beta}(D_N(t,0))} \\
    &\qquad\quad + \sup_{t\in[0,T]}\|\lambda\delta'(\cdot-t)\mathcal{H}(\pi^n) - \lambda\delta'(\cdot-t)\mathcal{H}(\pi^\infty)\|_{\Cc_{\beta}(D_N(t,0))} \\
    &\qquad\quad + \sup_{t\in[0,T]}\|F(t,\cdot,m^n_T)-F(t,\cdot,m^\infty_T)\|_{\Cc^{2}_{\beta}(B_N(0))} \\
    &\qquad\quad + \sup_{t\in[0,T]}\|\partial_tF(t,\cdot,m^n_T)-\partial_tF(t,\cdot,m^\infty_T)\|_{\Cc^{2}_{\beta}(B_N(0))} \Big) = 0,
\end{aligned}
\end{equation}
where, for $n\in \N\cap\{\infty\}$, we use the notation
\begin{align*}
    \tilde{b}^n(t,x) &:= \tilde{b}(t, x, m^n_t, \pi^n(t,x)), \quad \sigma^n(t,x) := \sigma (t,x, m^n_t), \\
    \tilde{r}^n(s-t,x) &:= \tilde{r}(s-t, x, m^n_s, \pi^n(s,x)).
\end{align*}
\end{Lemma}

\begin{proof}
    In the following proof,  let $A_0$ be a generic finite constant depending on $A^*, A_\lambda$, $C^*$, $N,\lambda$, $\Leb(U)$ but independent of $n$, which may change from line to line.

   Write $g^n(t,x,a) := \frac{1}{\lambda}[b(t,x,m^n_t,a) \cdot D_x w^n(t,x) + r(0,x,m^n_t,a)]$ for $n\in \mathbb{N}\cup\{\infty\}$. It follows from \eqref{eq:lm.piest.0} and \eqref{eq:lm.piest.3} that
\begin{align}\label{eq:lm.C12est0'}
    &\sup_{a\in U}\|g^n(\cdot,\cdot, a)\|_{\Cc_\beta(D_{N})} \leq \sup_{a\in U}\|g^n(\cdot,\cdot, a)\|_{\Cc_\alpha(D_{N})} \leq A_0, \quad \forall n\in \mathbb{N}\cup\{\infty\}, \\
    &\sup_{a\in U}\|\exp(g^n(\cdot,\cdot,a))\|_{\Cc_\beta(D_{N})} \leq \sup_{a\in U}\|\exp(g^n(\cdot,\cdot,a))\|_{\Cc_\alpha(D_{N})} \leq A_0, \quad \forall n\in \mathbb{N}\cup\{\infty\}.\label{eq:lm.C12est11}
\end{align}
Given generic functions $f(t,x), g(t,x), h(t,x,a)$ with $D$ being a domain in $[0,T]\times\R^d$, the following results hold
 \be\label{eq:lm.C12est0}
 [fg]_{\Cc_\beta(D)}\leq \|f\|_{L^\infty(D)}[g]_{\Cc_\beta(D)}+ \|g\|_{L^\infty(D)}[f]_{\Cc_\beta(D)}, \quad  \left[ \int_U h(t,x,a)da \right]_{\Cc_\beta(D)}\leq A_0\sup_{a\in U} [h(\cdot,\cdot,a)]_{\Cc_\beta(D)}.
 \ee
First, \eqref{eq:lm.C12est0'} and \eqref{eq:lm.C12est0} together imply that
\begin{align*}
    \sup_{a\in U}\| (g^n-g^\infty)(\cdot,\cdot, a)\|_{\Cc_{\beta}(D_{N})} \leq  A_0\Big[ & \sup_{a\in U}\|b^n(\cdot,\cdot,a)-b^{\infty}(\cdot,\cdot,a)\|_{\Cc_{\beta}(D_{N})} + \|D_x(w^n-w^\infty)\|_{\Cc_{\beta}(D_{N})} \\
    &+ \sup_{a\in U}\|r^n(\cdot,\cdot,a)-r^{\infty}(\cdot,\cdot,a)\|_{\Cc_{\beta}(D_{N})} \Big],
\end{align*}
where $b^{n}(t,x,a) := b(t,x,m^n_t,a)$ and $r^n(t,x,a) := r(0,x,m^n_t,a)$ for $n\in \mathbb{N}\cup\{\infty\}$.

It follows from the interpolation inequality (see, e.g. \cite[Exercise 3.26]{krylov1996lectures}) and Lipschitz property of $b$ in Assumption \ref{assume.r} that
\begin{align}
    \|b^n(\cdot,\cdot,a)-b^{\infty}(\cdot,\cdot,a)\|_{\Cc_{\beta}(D_{N})}&\leq A_0  \|b^n(\cdot,\cdot,a)-b^{\infty}(\cdot,\cdot,a)\|_{\Cc_{\alpha}(D_{N})}^{\frac{\beta}{\alpha}}  \|b^n(\cdot,\cdot,a)-b^{\infty}(\cdot,\cdot,a)\|_{L^{\infty}(D_{N})}^{1-\frac{\beta}{\alpha}}\nonumber\\
    &\leq A_0 \|b^n(\cdot,\cdot,a)-b^{\infty}(\cdot,\cdot,a)\|_{L^{\infty}(D_{N})}^{1-\frac{\beta}{\alpha}}\leq A_0d^{1-\frac{\beta}{\alpha}}(m^n,m^{\infty}).\label{es:bn}
    \end{align}
Similarly, we have that
\begin{align}
      & \|r^n(\cdot,\cdot,a)-r^{\infty}(\cdot,\cdot,a)\|_{\Cc_{\beta}(D_{N})}\leq A_0d^{1-\frac{\beta}{\alpha}}(m^n,m^{\infty}),\nonumber\\
        &\|r_t^n(\cdot,\cdot,a)-r_t^{\infty}(\cdot,\cdot,a)\|_{\Cc_{\beta}(D_{N}(t,0))}\vee \|\partial_tr_t^n(\cdot,\cdot,a)-\partial_tr_t^{\infty}(\cdot,\cdot,a)\|_{\Cc_{\beta}(D_{N}(t,0))}\leq A_0d^{1-\frac{\beta}{\alpha}}(m^n,m^{\infty}),\label{es:rn}\\
        &\| \sigma^n(\sigma^n)^\top - \sigma^\infty(\sigma^\infty)^\top \|_{\Cc_{\beta}(D_N)}\leq A_0d^{1-\frac{\beta}{\alpha}}(m^n,m^{\infty}),\label{es:lm3.4_1}\\
         &\|F(t,\cdot,m^n_T)-F(t,\cdot,m^\infty_T)\|_{\Cc^{2}_{\beta}(B_N(0))}\leq A_0\|F(t,\cdot,m^n_T)-F(t,\cdot,m^\infty_T)\|_{L^{\infty}(B_N(0))}^{1-\frac{2+\beta}{2+\alpha}},\label{es:lm3.4_2}\\
          &\|\partial_tF(t,\cdot,m^n_T)-\partial_tF(t,\cdot,m^\infty_T)\|_{\Cc^{2}_{\beta}(B_N(0))}\leq A_0\|\partial_tF(t,\cdot,m^n_T)-\partial_tF(t,\cdot,m^\infty_T)\|_{L^{\infty}(B_N(0))}^{1-\frac{2+\beta}{2+\alpha}},\label{es:lm3.4_22}
\end{align}
where $r_t^n(s,x,a) := r(s-t,x,m^n_s,a)$ for $n\in \mathbb{N}\cup\{\infty\}$.
Therefore, we obtain
\begin{align}\label{eq:lm.C12est0''}
    \sup_{a\in U}\| (g^n-g^\infty)(\cdot,\cdot, a)\|_{\Cc_{\beta}(D_{N})}\leq &A_0[d^{1-\frac{\beta}{\alpha}}(m^n,m^{\infty}) +\| D_x(w^n-w^\infty)\|_{\Cc_{\beta}(D_{N})}].
\end{align}
A direct calculation gives that
\begin{align*}
    &\ln\left( \int_U \exp[g^n(t,x,a)] da \right) - \ln\left( \int_U \exp[g^\infty(t,x,a)]da \right) \\
    &= \ln\left( \int_U \exp[g^n(t,x,a)-g^\infty(t,x,a)] \pi^\infty(t,x,a)da \right)\\
	&= \ln\left( 1 + \int_U \left( \exp[g^n(t,x,a)-g^\infty(t,x,a)] - 1 \right) \pi^\infty(t,x,a) da \right).
\end{align*}
Applying \eqref{eq:lm.C12est0} to the term  inside the logarithm, and combining it with \eqref{eq:lm.piest.5}, \eqref{eq:lm.C12est0'}, \eqref{eq:lm.C12est11} and \eqref{eq:lm.C12est0''}, we obtain that
\begin{align}\label{eq:lm.C12est7'}
	&\left\| \ln\left( \frac{\int_U \exp[g^n(\cdot,\cdot,a)] da} {\int_U \exp[g^\infty(\cdot,\cdot,a)]da} \right)  \right\|_{\Cc_\beta(D_{N})} \notag\\
	&= \left\| \ln\left( \int_U \exp[g^n(\cdot,\cdot,a)-g^\infty(\cdot,\cdot,a)] \pi^\infty(\cdot,\cdot,a)da \right) \right\|_{L^\infty(D_N)} \notag \\
	&\quad + \left[ \ln\left( 1 + \int_U \left( \exp[g^n(\cdot,\cdot,a)-g^\infty(\cdot,\cdot,a)] - 1 \right) \pi^\infty(\cdot,\cdot,a) da \right) \right]_{\Cc_\beta(D_N)}  \nonumber\\
	&\leq \sup_{a\in U} \| g^n(\cdot,\cdot,a) - g^\infty(\cdot,\cdot,a) \|_{L^{\infty}(D_{N})} \notag \\
	&\quad + A_0 \sup_{a\in U}\Big( \| \exp[g^n(\cdot,\cdot,a)-g^\infty(\cdot,\cdot,a)] - 1 \|_{L^\infty(D_N)} [\pi^\infty(\cdot,\cdot,a)]_{\Cc_\beta(D_N)} \notag \\
	&\quad \left. + \| \pi^\infty(\cdot,\cdot,a) \|_{L^\infty(D_N)} [\exp[g^n(\cdot,\cdot,a)-g^\infty(\cdot,\cdot,a)] - 1]_{\Cc_\beta(D_N)} \Big)  \right. \nonumber\\
	&\leq A_0 \sup_{a\in U} \| g^n(\cdot,\cdot,a) - g^\infty(\cdot,\cdot,a) \|_{\Cc_{\beta}(D_{N})} \nonumber\\
	&\leq A_0 \left[ d^{1-\frac{\beta}{\alpha}}(m^n, m^{\infty}) + \| D_x(w^n - w^\infty) \|_{\Cc_{\beta}(D_{N})} \right].
\end{align}
Note that
\begin{equation*}
    \exp[g^n(t,x,a)] - \exp [g^\infty(t,x,a)] = \exp[g^\infty(t,x,a)] \left( \exp[g^n(t,x,a)-g^\infty(t,x,a)] - 1 \right).
\end{equation*}
Applying \eqref{eq:lm.C12est0} to the above equality, together with \eqref{eq:lm.C12est0'} and \eqref{eq:lm.C12est11}, we deduce that
\bee
\begin{aligned}
 \left[e^{g^n(\cdot,\cdot,a)}- e^{g^\infty(\cdot,\cdot,a)} \right]_{\Cc_{\beta}(D_{N})}&\leq A_0  \sup_{a\in U} \bigg([e^{g^\infty(\cdot,\cdot,a)}]_{\Cc_\beta (D_{N})} \|(g^n-g^\infty) (\cdot,\cdot, a)\|_{L^\infty(D_{N})} \\
&\qquad\qquad \quad  +  \|e^{g^\infty(\cdot,\cdot,a)}\|_{L^\infty(D_{N})} [(g^n-g^\infty)(\cdot,\cdot,a)]_{\Cc_\beta(D_{N})} \bigg),
\end{aligned}
\eee
which further implies that
\be\label{eq:lm.C12est1}
\|\exp[g^n(\cdot,\cdot,a)]- \exp [g^\infty(\cdot,\cdot,a)] \|_{\Cc_\beta(D_{N})}\leq A_0 \left( d^{1-\frac{\beta}{\alpha}}(m^n,m^{\infty}) +\| D_x(w^n-w^\infty)\|_{\Cc_{\beta}(D_{N})} \right).
\ee
Note that
$$
\begin{aligned}
	&\pi^n(t,x,a)-\pi^\infty(t,x,a)\\
	&= \frac{  \exp[g^n(t,x,a)]\int_U \exp[g^\infty(t,x,a')]da' -   \exp[g^\infty(t,x,a)]\int_U \exp[g^n(t,x,a')]da' }{   \int_U \exp[g^n(t,x,a')]da'\int_U \exp[g^\infty(t,x,a')]da' } \\
	&=\int_U \left(e^{g^\infty(t,x,a')}-e^{g^n(t,x, a')}\right)da' \frac{ \pi^n(t,x,a) }{\int_U \exp[g^\infty(t,x,a')]da'}  +   \left(e^{g^n(t,x,a)}-e^{g^\infty(t,x, a)}\right) \frac{ 1 }{\int_U e^{g^\infty(t,x,a')}da'}.
\end{aligned}
$$
Applying \eqref{eq:lm.C12est0} to the two terms in the last line above, together with  \eqref{eq:lm.C12est0'}, \eqref{eq:lm.C12est11} and \eqref{eq:lm.C12est1}, we obtain
\be\label{eq:lm.C12est3}
\sup_{a\in U} \|\pi^n(\cdot,\cdot,a)-\pi^\infty(\cdot,\cdot,a)\|_{\Cc_{\beta}(D_{N})}\leq A_0 \left( d^{1-\frac{\beta}{\alpha}}(m^n,m^{\infty}) +\| D_x(w^n-w^\infty)\|_{\Cc_{\beta}(D_{N})} \right).
\ee
Then combining \eqref{eq:lm.C12est3}  with Assumption \ref{assume.r} and \eqref{eq:lm.C12est0}--\eqref{es:rn} shows that
\begin{align}\label{eq:lm.C12est5}
    &\sup_{t \in [0,T]}\|\tilde{r}^{n}-\tilde{r}^{\infty}\|_{\Cc_{\beta}(D_{N_0}(t,0))} + \sup_{t \in [0,T]}\|\partial_t\tilde{r}^{n}-\partial_t\tilde{r}^{\infty}\|_{\Cc_{\beta}(D_{N_0}(t,0))} + \|\tilde{b}^{n}-\tilde{b}^{\infty}\|_{\Cc_{\beta}(D_{N})} \notag \\
    &\leq A_0 \bigg[ \sup_{a\in U,t\in[0,T]}\|r^n_t(\cdot,\cdot,a)-r_t^{\infty}(\cdot,\cdot,a)\|_{\Cc_{\beta}(D_{N_0}(t,0))} + \sup_{a\in U,t\in[0,T]}\|\partial_t r^n_t(\cdot,\cdot,a)-\partial_t r_t^{\infty}(\cdot,\cdot,a)\|_{\Cc_{\beta}(D_{N_0}(t,0))} \notag \\
    &\qquad \quad + \sup_{a\in U}\|b^n(\cdot,\cdot,a)-b^{\infty}(\cdot,\cdot,a)\|_{\Cc_{\beta}(D_{N})} + \sup_{a\in U} \|\pi^n(\cdot,\cdot,a)-\pi^\infty(\cdot,\cdot,a)\|_{\Cc_{\beta}(D_{N})} \bigg] \notag \\
    &\leq A_0 \left[ d^{1-\frac{\beta}{\alpha}}(m^n,m^{\infty}) + \|D_x(w^n-w^\infty)\|_{\Cc_{\beta}(D_{N})} \right].
\end{align}
Meanwhile, it holds that
$$
\begin{aligned}
\Hc(\pi^n(t,x))-\Hc(\pi^\infty(t,x))=&\int_U [g^\infty(t,x,a)-g^n(t,x,a)] \pi^n(t,x,a)da \\&+\int_U g^\infty(t,x,a) [\pi^\infty(t,x,a)-\pi^n(t,x,a)]da
+\ln\left( \frac{\int_U \exp[g^n(t,x,a)] da} {\int_U \exp[g^\infty(t,x,a)]da} \right).
\end{aligned}
$$
Applying \eqref{eq:lm.C12est0'} and \eqref{eq:lm.C12est0} to the first two integrals on the right-hand side of the equality above, together with \eqref{eq:lm.C12est0''}, \eqref{eq:lm.C12est7'} and \eqref{eq:lm.C12est3}, we obtain
$$
\|\Hc(\pi^n)-\Hc(\pi^\infty)\|_{\Cc_\beta(D_{N})}\leq A_0\left[ d^{1-\frac{\beta}{\alpha}}(m^n,m^{\infty}) +\| D_x(w^n-w^\infty)\|_{\Cc_{\beta}(D_{N})} \right],
$$
 which yields
 \begin{align}\label{eq:lm.C12est7}
    &\sup_{t\in[0,T]}\|\lambda\delta(\cdot-t) (\mathcal{H}(\pi^n)-\mathcal{H}(\pi^\infty))\|_{\Cc_{\beta}(D_{N}(t,0))} + \sup_{t\in[0,T]}\|\lambda\delta'(\cdot-t) (\mathcal{H}(\pi^n)-\mathcal{H}(\pi^\infty))\|_{\Cc_{\beta}(D_{N}(t,0))} \notag\\
    &\leq A_0 \left[ d^{1-\frac{\beta}{\alpha}}(m^n,m^{\infty}) + \|D_x(w^n-w^\infty)\|_{\Cc_{\beta}(D_{N})} \right].
\end{align}
Finally, \eqref{eq:lemma3.4} follows immediately from the combination of \eqref{es:lm3.4_1}, \eqref{es:lm3.4_2}, \eqref{es:lm3.4_22}, \eqref{eq:lm.C12est5}, and \eqref{eq:lm.C12est7}.
\end{proof}

{We next step away from the context of $\pi^n$ and present a standard stability result for hitting times of diffusions. Since we were unable to locate a convenient reference for the precise form needed here, we include a self-contained proof for completeness.

\begin{Lemma}\label{lm:hitting.uniform}
Take $R, T\in(0,\infty)$. For $n\in\mathbb{N}\cup\{\infty\}$ and $(s,x)\in [0,\infty)\times B_{R}(0)$, consider
\begin{align*}
	dX^n_l = b^n(l, X^n_l )dl + \sigma^n(l,X^n_l)dW_l, \;\; l\in[s,\infty), \quad X^n_s = x,
\end{align*}
where $b^n:[0,\infty)\times\mathbb{R}^d\to\mathbb{R}^d$ and $\sigma^n:[0, \infty)\times\mathbb{R}^d\to\mathbb{R}^{d\times m}$ are uniformly bounded in $n$ and satisfy the following conditions:
\bi
\item $|b^n(t,x)-b^n(t,y)| + |\sigma^n(t,x)-\sigma^n(t,y)| \leq L |x-y|$ for a constant $L>0$ independent of $n$;
\item $\sigma^n$ is uniformly elliptic, i.e.,
$
\xi^\top (\sigma^n (\sigma^n)^\top)(t,x)\,\xi \geq \tilde\eta |\xi|^2,
$
for all $\xi\in\mathbb{R}^d$, for a constant $\tilde \eta>0$ independent of $n$;
\item
$
\lim_{n\to\infty} \left( \|b^n-b^\infty\|_{C^0([0,T]\times B_N(0))}+\|\sigma^n-\sigma^\infty\|_{C^0([0,T]\times B_N(0))}
\right) = 0,
$ for all $N\in \N$.
\ei
Let $\rho^n:= \inf\{l \geq s : X^{n}_l \notin B_{R}(0)\}$ be the first exit time of the process $X^n$ from the ball $B_R(0)$. Then, for any $\epsilon>0$,
\begin{equation}\label{eq:lm:hitting.unif}
	\lim_{n\to\infty} \sup_{(s,x) \in [0,T]\times B_R(0)} \mathbb{P}_{s,x} \left( |\rho^n \wedge T - \rho^\infty \wedge T | > \epsilon \right) = 0.
\end{equation}
\end{Lemma}

\begin{proof}
We first show that
\begin{equation}\label{eq:lm.hittinguniform0}
	\sup_{(s,x) \in [0,T]\times B_R(0)}\mathbb{E}_{s,x}\left[ \sup_{l \in [s,T]} \left|X^{n}_l - X^{\infty}_l\right|^2\right] \to 0, \quad \text{as } n \to \infty.
\end{equation}
For $M>R$, define
$$
\tau_M^n:=\inf\{l\geq s: |X_l^n|\geq M\},\quad
\tau_M^\infty:=\inf\{l\geq s: |X_l^\infty|\geq M\},\quad
 T_M^n:=\tau_M^n\wedge \tau_M^\infty\wedge T.
$$
Let
$
\Delta_l^n:=X_l^n-X_l^\infty$ and $\Delta_n := \sup_{l\in[s,T]} |X^n_l - X^\infty_l|.$
For $l\in[s,T]$, we have that
$$
\begin{aligned}
\Delta_{l\wedge T_M^n}^n&=\int_s^{l\wedge T_M^n}\big(b^n(u,X_u^n)-b^\infty(u,X_u^\infty)\big)du +
	\int_s^{l\wedge T_M^n}\big(\sigma^n(u,X_u^n)-\sigma^\infty(u,X_u^\infty)\big)dW_u .
\end{aligned}
$$
Since $|X_u^n|,|X_u^\infty|\leq M$ for $u\leq T_M^n$, we obtain
$$
\begin{aligned}
	|b^n(u,X_u^n)-b^\infty(u,X_u^\infty)|&\leq L|\Delta_u^n|+\|b^n-b^\infty\|_{C^0([0,T]\times B_M(0))},\\
	|\sigma^n(u,X_u^n)-\sigma^\infty(u,X_u^\infty)|&\leq L|\Delta_u^n|+\|\sigma^n-\sigma^\infty\|_{C^0([0,T]\times B_M(0))}.
\end{aligned}
$$
Therefore, by Cauchy--Schwarz and the BDG inequality,
$$
\begin{aligned}
	\E_{s,x}  \left[  \sup_{s\leq l\leq t}	|\Delta_{l\wedge T_M^n}^n|^2	\right]
	&\leq  C\int_s^t  \E_{s,x}\left[  \sup_{s\leq q\leq u} |\Delta_{q\wedge T_M^n}^n|^2  \right]du  \\
	&\quad+C\left(  \|b^n-b^\infty\|_{C^0([0,T]\times B_M(0))}^2  +  \|\sigma^n-\sigma^\infty\|_{C^0([0,T]\times B_M(0))}^2  \right),
\end{aligned}
$$
where the constant $C$ is independent of
$n,s,x$ and $M$. Then Gronwall's inequality yields that
$$
\sup_{(s,x)\in[0,T]\times B_R(0)} \E_{s,x}
\left[ \sup_{s\leq l\leq T} |\Delta_{l\wedge T_M^n}^n|^2 \right]
\leq  C' \left( \|b^n-b^\infty\|_{C^0([0,T]\times B_M(0))}^2 + \|\sigma^n-\sigma^\infty\|_{C^0([0,T]\times B_M(0))}^2 \right).
$$
We then arrive at
\be\label{eq:lm.hittinguniform1}
\lim_{n\to\infty} \sup_{(s,x)\in[0,T]\times B_R(0)} \E_{s,x} \left[ \sup_{s\leq l\leq T} |\Delta_{l\wedge T_M^n}^n|^2 \right] =0,\quad \forall M\in(R,\infty).
\ee
Since the coefficients are uniformly bounded, there exists a constant $A>0$, independent of $n$ and $(s,x)\in[0,T]\times B_R(0)$, such that
$
\E_{s,x}\left[ \sup_{s\leq l\leq T}|X_l^n-x|^4\right] \leq A .
$
Thus, it holds that
$$
\sup_{(s,x)\in[0,T]\times B_R(0)} \P_{s,x}(\tau_M^n<T)  \leq  \frac{A}{(M-R)^4},\quad \forall n\in \N\cup\{\infty\}.
$$
Again, the uniform boundedness of the coefficients implies that
$$
\sup_{n\in \N}\sup_{(s,x)\in[0,T]\times B_R(0)} \E_{s,x} \left[ \sup_{s\leq l\leq T}|X_l^n-X_l^\infty|^4  \right]<\infty.
$$
Then, Cauchy--Schwarz inequality gives that
$$
\begin{aligned}
\E_{s,x}	\left[ \sup_{s\leq l\leq T}|\Delta_l^n|^2	\right] &\leq2\E_{s,x}\left[\sup_{s\leq l\leq T}|\Delta_{l\wedge T_M^n}^n|^2\right] + 2\E_{s,x} \left[ \sup_{s\leq l\leq T}|\Delta_l^n|^2	\mathbf 1_{\{\tau_M^n\wedge \tau_M^\infty<T\}} \right] \\
&\leq 2\E_{s,x} \left[ \sup_{s\leq l\leq T}	|\Delta_{l\wedge T_M^n}^n|^2	\right]	+ \frac{A}{(M-R)^2}.
\end{aligned}
$$
This, together with \eqref{eq:lm.hittinguniform1}, gives
$$
\begin{aligned}
\limsup_{n\to\infty} \sup_{(s,x)\in[0,T]\times B_R(0)} \E_{s,x}	\left[ \sup_{s\leq l\leq T}|\Delta_l^n|^2	\right] \leq \frac{A}{(M-R)^2},
\end{aligned}
$$
and \eqref{eq:lm.hittinguniform0} is verified by sending $M\to\infty$.

Now we show the desired result. It follows from \eqref{eq:lm.hittinguniform0} that
\be\label{eq:lm.hittinguniform0'}
\lim_{n\to\infty}\sup_{(s,x) \in [0,T]\times B_R(0)}\mathbb{P}_{s,x}(\Delta_n > \delta) = 0,\; \text{ for any $\delta > 0$.}
\ee
Let $\tau^n:=\rho^n\wedge T$ for $n\in \N\cup\{\infty\}$. We decompose the target probability as
\begin{align*}
	\mathbb{P}_{s,x} ( |\tau^n - \tau^\infty| > \epsilon ) \leq \mathbb{P}_{s,x} (\Delta_n > \delta) + \mathbb{P}_{s,x}(\tau^n > \tau^\infty + \epsilon, \, \Delta_n \leq \delta) + \mathbb{P}_{s,x}(\tau^n < \tau^\infty - \epsilon, \, \Delta_n \leq \delta).
\end{align*}
The goal is to show that
\be\label{eq:lm.hittinguniform7}
	\lim_{\delta\to 0+}\sup_{n \in \mathbb{N}}\sup_{(s,x) \in [0,T]\times B_R(0)}	\mathbb{P}_{s,x}(\tau^n > \tau^\infty + \epsilon, \, \Delta_n \leq \delta)=0.
\ee
On $\{\tau^n > \tau^\infty + \epsilon, \Delta_n \leq \delta\}$, the path of $X^\infty$ hits the boundary at $\tau^\infty$ and remains within distance $\delta$ of the exterior for duration $\epsilon$. The strong Markov property gives that
\begin{align*}
\sup_{(s,x) \in [0,T]\times B_R(0)}	\mathbb{P}_{s,x}(\tau^n > \tau^\infty + \epsilon, \, \Delta_n \leq \delta)&\leq \sup_{(r,y)\in [0, T]\times \partial B_R (0)} \mathbb{P}_{r,y}\left(\sup_{r\leq l \leq r+ \epsilon} |X^\infty_l| \leq R + \delta\right).
\end{align*}
To verify \eqref{eq:lm.hittinguniform7}, it suffices to show that
\be  \label{eq:lm.hittinguniform3}
\lim_{\delta\to 0+}  \sup_{(r,y)\in [0, T] \times \partial B_R (0)} \mathbb{P}_{r,y}\left(\sup_{r\leq l \leq r+ \epsilon} |X^\infty_l| \leq R + \delta\right)=0.
\ee
Fix an arbitrary $(r,y)\in [0,\infty) \times \partial B_{R}(0)$. Set $K:=\|b\|_{L^\infty(\T\times \R^d)}<\infty$. Define $Y_l := X^\infty_{r+l} \text{ for }0\leq l\leq \eps$.
Let us introduce a one-dimensional process
	$$
	Z_l := \frac{y}{|y|} \cdot (Y_l - y) \;\text{ for }0\leq l\leq \eps.
	$$
	We next relate $Z_l$ to $|Y_l|$. Note that
	$
	Z_l = \frac{y}{|y|} \cdot Y_l - \frac{y}{|y|} \cdot y = \frac{y}{|y|} \cdot Y_l - |y|.
	$
	Since $|\frac{y}{|y|}|=1$, $\frac{y}{|y|} \cdot Y_l \leq |Y_l|$, and hence $Z_l \leq |Y_l| - R$, it holds that
	$$
	\left\{ \sup_{0\leq l \leq \eps} |Y_l| \leq R+\delta \right\}\subseteq\left\{ \sup_{0\leq l \leq \eps} Z_l \leq \delta \right\}.
	$$
	Next, we compute the dynamics of $Z_l$. From the dynamics of $Y_l$, we obtain that
	$$
	Z_l=\int_0^l \frac{y}{|y|} \cdot b^\infty(r+u,Y_u)du+\int_0^l \frac{y}{|y|} \cdot \sigma^\infty(r+u,Y_u)dW_{r+u}.
	$$
	Define the martingale term
	$M_l := \int_0^l \frac{y}{|y|} \cdot \sigma^\infty(r+u,Y_u)dW_{r+u}.$
	Since
	$
	\int_0^l \frac{y}{|y|} \cdot b^\infty(r+u,Y_u)du \geq -Kl,
	$
	we deduce that $Z_l \geq M_l - K l$, which implies
	$
	\left\{ \sup_{0\leq l\leq \eps} Z_l \leq \delta \right\}\subseteq\left\{ \sup_{0\leq l\leq \eps} (M_l - Kl) \leq \delta \right\}.
	$
	That is,
	\be\label{eq:lm.hittinguniform3'}
	\left\{ \sup_{0\leq l\leq \eps} |Y_l| \leq R+\delta \right\}\subseteq\left\{ \sup_{0\leq l \leq \eps} (M_l - Kl) \leq \delta \right\}.
	\ee
The quadratic variation of $M$ is given by $\langle M \rangle_l=\int_0^l \frac{y^\top}{|y|} \sigma^\infty(\sigma^\infty)^\top(r+u,Y_u)\frac{y}{|y|}du.$
	We then conclude from the uniform ellipticity assumption that $\langle M \rangle_l \geq\tilde \eta l$. Then by the Dambis--Dubins--Schwarz theorem (see, e.g., \cite[Theorem 3.4.6]{karatzas1991brownian}), there exists a one-dimensional standard Brownian motion $B$ such that
	$
	M_l = B_{\langle M \rangle_l}.
	$
	Define the right-continuous inverse of $\langle M \rangle$ as $\theta_u := \inf\{ l \geq 0 : \langle M \rangle_l > u\}$. It follows from $\langle M \rangle_l \geq\tilde \eta l$ that
\be\label{eq:lm.hittinguniform5}
	\theta_u \leq \frac{u}{\tilde \eta} \quad \forall u \in [0,\tilde \eta\eps].
\ee
	Now, if $\sup_{0\leq l\leq \eps} (M_l - Kl) \leq \delta$,
	then for every $u \in [0,\tilde \eta\eps]$, $B_u - K \theta_u \leq \delta$. We then conclude from \eqref{eq:lm.hittinguniform5} that $B_u - \frac{K}{\tilde \eta}u \leq B_u - K \theta_u \leq \delta$, and hence
	$$
	\left\{   \sup_{0\leq l\leq \eps} (M_l - Kl) \leq \delta  \right\}
	\subseteq  \left\{ \sup_{0\leq u\leq \tilde \eta\eps}
	\left( B_u - \frac{K}{\tilde \eta}u \right) \leq \delta   \right\}.
	$$
Together with \eqref{eq:lm.hittinguniform3'}, this gives
	\be\label{eq:lm.hittinguniform5'}
	 \sup_{(r,y)\in [0, T] \times \partial B_R (0)}  \P_{r,y}  \left( \sup_{0\leq l\leq \eps} |X^\infty_{r+l}| \leq R+\delta \right)
	\leq \P \left( \sup_{0\leq u\leq \tilde \eta\eps} \left( B_u - \frac{K}{\tilde \eta}u \right) \leq \delta \right).
	\ee
Note that, for the Brownian motion with drift $-\frac{K}{\tilde \eta}$, we have for any $a\in (0,\infty)$ that
	$$
	\begin{aligned}
	\P  \left(\sup_{0\leq u\leq a} (B_u - \frac{K}{\tilde \eta}u) \leq \delta\right)
&=  \Phi\!\left( \frac{\delta + \frac{K}{\tilde \eta}a}{\sqrt{a}} \right)  -  e^{-2\frac{K}{\tilde \eta}\delta}  \Phi\!\left( \frac{\frac{K}{\tilde \eta}a - \delta}{\sqrt{a}} \right)
 \\&\to \Phi(\frac{K}{\tilde \eta}\sqrt{a}) - \Phi(\frac{K}{\tilde \eta}\sqrt{a})=0, \;  \text{ as $\delta\to 0+$},
	\end{aligned}
	$$
	where $\Phi$ denotes the standard normal distribution function.
	By taking $a = \tilde \eta \eps$ and combining with \eqref{eq:lm.hittinguniform5'}, we achieve \eqref{eq:lm.hittinguniform3}, and consequently \eqref{eq:lm.hittinguniform7} holds.

In a similar fashion, we also have that
$$
	\lim_{\delta\to 0+}\sup_{n \in \mathbb{N}}\sup_{(s,x) \in [0,T]\times B_R(0)}	\mathbb{P}_{s,x}(\tau^n < \tau^\infty - \epsilon, \, \Delta_n \leq \delta)=0.
$$
The desired result follows from combining the above result with \eqref{eq:lm.hittinguniform0'} and \eqref{eq:lm.hittinguniform7}.
\end{proof}
}

With the previous preparations, we are ready to prove the next key lemma.

\begin{Lemma}\label{lm:Phi.continuous}
	Fix $0<\lambda\leq \lambda_0$. For any $0\leq\beta<\alpha$, the following holds:
	\begin{itemize}
		\item[(i)] $\mathcal{E}_{\lambda} := \mathcal{M}_\lambda \times \mathscr{M}_{\nu}^{\kappa,C^*}$ is a compact subset of the product space $\Cc^{0,2, \wg}_{\beta}([0,T]\times\R^d)\times \mathscr{M}$;
		\item[(ii)] $\Phi_\lambda$ is a continuous mapping from $\mathcal{E}_\lambda$ into $\mathcal{E}_\lambda$, equipped with the product topology induced by the norm $\|\cdot\|_{\Cc^{0,2,\wg}_{\beta}([0,T]\times\R^d)}$ and the metric $d(\cdot, \cdot)$.
	\end{itemize}
\end{Lemma}
	\begin{proof}
{\bf Part (i).} It follows from a similar diagonal argument as in \cite[Lemma 3.5 (i)]{WYZZ2026} that $\mathcal{M}_\lambda$ is a compact subset of the space $\Cc^{0,2, \wg}_{\beta}([0,T]\times \R^d)$. Note that $\mathscr{M}_{\nu}^{\kappa,C^*}$ is a compact subset of $(\mathscr{M}, d)$ via the Arzel\`a-Ascoli lemma. As the product of two compact sets is compact in the product topology, the conclusion holds.

{\bf Part (ii).} Take $(w,m)\in \Ec_\lambda$, then Lemma \ref{lm:iteration.estimateentropy} gives that $\Phi_\lambda(w,m)\in \Ec_\lambda$. Thus, $\Phi_\lambda$ maps from $\Ec_\lambda$ to $\Ec_\lambda$.
We next show that $\Phi_\lambda$ is a continuous mapping from $\Ec_\lambda$ to $\Ec_\lambda$.

Take a sequence $(w^n,m^n)_{n\in \N\cup\{\infty\}}\subset \Ec_\lambda$ with $\lim\limits_{n\to\infty}\|w^n-w^\infty\|_{\Cc^{0,2,\wg}_{\beta}([0,T]\times \R^d)}=0$ and $\lim\limits_{n\to\infty}d(m^n,m^{\infty})=0$. Set $ \pi^n(t,x,a)=\Gamma_\lambda(t,x,D_xw^n(t,x),m^n_t,a), V^n:=V^{\pi^n,m^n}_{\lambda}$ for each $n\in \N\cup\{\infty\}$. By Lemma \ref{lm:iteration.estimateentropy}, it holds that
\begin{equation*}
		\begin{aligned}
			\partial_s V^{n}_\lambda (t,s, x) &+ \frac{1}{2}\operatorname{tr}\left((\sigma\sigma^T)(s,x, m^n_s) D^2_x V^{n}_\lambda (t,s, x)\right) + \tilde{b}(s,x, m^n_s, \pi^n(s,x))\cdot D_x V^{n}_\lambda(t,s, x) \\
			&+ \tilde{r}(s-t,x, m^n_s,\pi^n(s,x)) + \lambda \delta(s-t)\mathcal{H}(\pi^n(s,x))=0,\\
			V^{n}_\lambda (t,T, x) &= F(t,x,m^n_T).
		\end{aligned}
	\end{equation*}
Define $\bar V^n = V_\lambda^\infty - V_\lambda^n$ for $n\in \mathbb{N}$. Then $\bar V^n$ satisfies
\be\label{eq:PDF.Vbar}
\begin{aligned}
\partial_s \bar V^n(t,s,x) &+ \frac{1}{2}\operatorname{tr}\left((\sigma\sigma^T)(s,x, m^\infty_s) D^2_x \bar V^n(t,s,x)\right) \\
&+\tilde{b}(s,x, m^{\infty}_s, \pi^{\infty}(s,x))\cdot D_x \bar V^{n}(t,s,x) + f^n(t,s,x)=0,\\
\bar V^n(t,T,x) &= F(t,x,m^\infty_T) - F(t,x,m^n_T),
\end{aligned}
\ee
where
\begin{align}\label{eq:lmC12.fn}
f^n(t,s,x) := & \tilde{r}(s-t,x,m^{\infty}_s,\pi^{\infty}(s,x)) - \tilde{r}(s-t,x,m^{n}_s,\pi^{n}(s,x)) \notag\\
&+ \lambda \delta(s-t)(\mathcal{H}(\pi^\infty(s,x)) - \mathcal{H}(\pi^n(s,x))) \notag\\
&+ (\tilde{b}(s,x, m^\infty_s, \pi^\infty(s,x)) - \tilde{b}(s,x, m^n_s, \pi^n(s,x)))\cdot D_x V^{n}_\lambda(t,s, x) \notag\\
&+ \frac{1}{2}\operatorname{tr}\left(((\sigma\sigma^T)(s,x, m^\infty_s) - (\sigma\sigma^T)(s,x, m^n_s))D^2_x V^{n}_\lambda (t,s, x)\right).
\end{align}

Let $\bar{J}^n(t,x) := \bar{V}^n(t,t,x)$. We claim that $\|\bar J^n\|_{\Cc^{0,2,\wg}_{\beta}([0,T]\times \R^d)} \to 0$ as $n \to \infty$. First, note that for any $\eps > 0$, there exists $N_0$ such that $\sum_{N=N_0}^\infty \frac{A_\lambda}{2^N} \leq \frac{\eps}{4}$. Since $(w^n) \subset \Mc_\lambda$, the uniform estimate in Lemma \ref{lm:iteration.estimateentropy} implies
\begin{align*}
\sum_{N=N_0}^\infty \frac{1}{2^N}\|\bar J^n\|_{\Cc^{0,2}_{\beta}(D_N)}&\leq \sum_{N=N_0}^\infty \frac{1}{2^N} \left( \|J^n\|_{\Cc^{0,2}_{\beta}(D_N)}+  \|J^\infty\|_{\Cc^{0,2}_{\beta}(D_N)} \right)\\
&\leq \sum_{N=N_0}^\infty \frac{1}{2^N} \left( \|J^n\|_{\Cc^{0,2}_{\alpha}(D_N)}+  \|J^\infty\|_{\Cc^{0,2}_{\alpha}(D_N)} \right)
\leq 2\sum_{N=N_0}^\infty\frac{A_\lambda}{2^N}\leq \frac{\eps}{2}.
\end{align*}
That is, to prove the claim, it suffices to show
\be\label{eq:lm.continuous1}
\lim_{n\to \infty} \|\bar J^n\|_{\Cc^{0,2}_{\beta}(D_{N_0})} =0.
\ee
We first verify that
\begin{equation}\label{eq:lmC12.vbarbound}
\sup_{t\in[0,T]} \|V^n(t,\cdot,\cdot) - V^\infty(t,\cdot,\cdot)\|_{\Cc^0(D_{N_0+1}(t,0))} + \sup_{t\in[0,T]} \|\partial_t V^n(t,\cdot,\cdot) - \partial_t V^\infty(t,\cdot,\cdot)\|_{\Cc^0(D_{N_0+1}(t,0))} \to 0,
\end{equation}
as $n \to \infty$. For $N \in \mathbb{N} \cup \{\infty\}$, let $\rho^n_{s,N} := \inf\{l \geq s : X^{n}_l \notin B_{N}(0)\}$ be the first exit time of the process $X^n$ from the ball $B_N(0)$, where $X^n$ is governed by
\begin{equation*}
dX^n_l = \tilde{b}(l,X^n_l,m^n_l,\pi^n(l,X^n_l))dl + \sigma(l,X^n_l,m^n_l)dW_l, \quad X^n_s = x.
\end{equation*}
Let us define the truncated value functions $V^n_{N}$ for $N \in \mathbb{N} \cup \{\infty\}$ as
\begin{equation*}
V^n_{N}(t,s,x) := \mathbb{E}_{s,x} \left[ \int_s^{T \wedge \rho^{n}_{s,N}} \left( \tilde{r}(l-t, X^{n}_l,m^n_l,\pi^n(l,X^n_l)) + \delta(l-t)\lambda \mathcal{H}(\pi^n(l,X^{n}_l)) \right) dl + F(t,X^{n}_{T \wedge \rho^{n}_{s,N}},m^n_T) \right].
\end{equation*}
From the boundedness of the coefficients in Assumption \ref{assume.r} and the estimates in Lemma \ref{lm:pi.est}, it follows that
\begin{align*}
| V^n (t,s,x) - V^n_{N} (t,s,x) | &= \bigg| \mathbb{E}_{s,x} \left[ \int_{\rho^{n}_{s,N}}^{T} \left( \tilde{r} + \delta(l-t)\lambda \mathcal{H}(\pi^n) \right) dl \cdot \mathbf{1}_{\{\rho^{n}_{s,N} < T\}} \right] \\
&\quad + \mathbb{E}_{s,x} \left[ \left(F(t,X^{n}_{T},m^n_T) - F(t,X^{n}_{\rho^{n}_{s,N}},m^n_T)\right) \cdot \mathbf{1}_{\{\rho^{n}_{s,N} < T\}} \right] \bigg| \\
&\leq A_0 \mathbb{P}(\rho^{n}_{s,N} < T) \leq A_0 \mathbb{P} \left( \sup_{s \leq l \leq T} |X^{n}_l| > N \right) \leq \frac{A_0}{N^2} \mathbb{E}_{s,x}\left[\sup_{s \leq l \leq T} |X^{n}_l|^2\right].
\end{align*}
As $b$ and $\sigma$ are bounded, the second moment of the supremum of $X^n$ is uniformly bounded. Consequently, for any $\varepsilon > 0$, there exists a sufficiently large $N_1 \in \mathbb{N}$ such that
\begin{equation}\label{eq:lmvbarN.eps}
\sup_{t\in[0,T]}\|V^n(t,\cdot,\cdot) - V^{n}_{N_1}(t,\cdot,\cdot)\|_{\Cc^{0}(D_{N_0+1}(t,0))} \leq \varepsilon, \quad \forall n \in \mathbb{N} \cup \{\infty\}.
\end{equation}
It then holds that
\begin{align*}
&V^n_{N_1} (t,s,x) - V^\infty_{N_1} (t,s,x) \\
&= \mathbb{E}_{s,x}\bigg[\int_{s}^{\rho^{n}_{s,N_1}\wedge T} \left( [\tilde{r}^n-\tilde{r}^{\infty}](l-t, X^{n}_l) + \delta(l-t)\lambda [\mathcal{H}(\pi^n)-\mathcal{H}(\pi^\infty)](l,X^{n}_l) \right) dl \\
&\qquad \qquad + F(t,X^{n}_{T \wedge \rho^{n}_{s,N_1}},m^n_T) - F(t,X^{n}_{T \wedge \rho^{n}_{s,N_1}},m^\infty_T) \bigg] \tag{I} \\
&\quad + \mathbb{E}_{s,x}\bigg[\int_s^{\rho^{n}_{s,N_1}\wedge T} \left( \tilde{r}^{\infty}(l-t, X^{n}_l) + \delta(l-t)\lambda \mathcal{H}(\pi^\infty(l,X^{n}_l))\right) dl \\
&\qquad \qquad \; - \int_s^{\rho^{\infty}_{s,N_1}\wedge T} \left( \tilde{r}^{\infty}(l-t, X^{\infty}_l) + \delta(l-t)\lambda \mathcal{H}(\pi^\infty(l,X^{\infty}_l))\right) dl \\
&\qquad \qquad \; + F(t,X^{n}_{T \wedge \rho^{n}_{s,N_1}},m^\infty_T) - F(t,X^{\infty}_{T \wedge \rho^{\infty}_{s,N_1}},m^\infty_T) \bigg]. \tag{II}
\end{align*}
By Lemma \ref{lm:Phi.continuous0}, for term (I), we directly obtain that
\begin{equation*}
    \sup_{t\in[0,T]}\|\text{(I)}\|_{\Cc^{0}(D_{N_0+1}(t,0))} \to 0, \quad \text{as } n \to \infty.
\end{equation*}
For term (II), leveraging the estimates in Lemmas \ref{lm:pi.est} and \ref{lm:iteration.estimateentropy}, we conclude that
\begin{align}\label{ineq:II}
	|\text{(II)}|&\leq  A_0 \mathbb{E}_{s,x} \left[
	\sup_{l\in[s,T]}|X^{n}_{l} - X^{\infty}_{l}|^{\alpha} +|T \wedge\rho^{\infty}_{s,N_1} -T \wedge\rho^{n}_{s,N_1}|+ |X^{\infty}_{T \wedge \rho^{n}_{s,N_1}} - X^{\infty}_{T\wedge \rho^{\infty}_{s,N_1}}|^{\alpha} \right].
\end{align}
As $w^n \in \mathcal{M}_\lambda$, Lemma \ref{lm:derivatives_bounds} implies that both  $\tilde{b}^n(t,x) = \tilde{b}(t,x,m^n_t,\pi^n(t,x))$ and  $\sigma^n(t,x) = \sigma(t,x,m^n_t)$ are uniformly Lipschitz continuous in $x$. Applying the BDG inequality yields
\begin{align*}
\mathbb{E}_{s,x}\left[\sup_{l\in[s,t]}\left|X^{n}_l - X^{\infty}_l \right|^2\right] &\leq A_0 \mathbb{E}_{s,x} \bigg[ \int_s^t \left( \left| \tilde{b}^n(l,X^{n}_l) - \tilde{b}^\infty(l, X^{\infty}_l) \right|^2  + \left| \sigma^n(l, X^{n}_l) - \sigma^\infty(l,X^{\infty}_l) \right|^2 \right) dl \bigg].
\end{align*}
By decomposing the differences via the uniform spatial Lipschitz property of $\tilde{b}^\infty$ and $\sigma^\infty$, we obtain that
\begin{equation*}
	\mathbb{E}_{s,x}\left[ \sup_{l\in[s,t]} \left| X^{n}_l - X^{\infty}_l \right|^2\right] \leq A_0\left( \int_s^t \mathbb{E}_{s,x}\left[ \sup_{r\in[s,l]} \left| \tilde{X}^{n}_r - \tilde{X}^{\infty}_r \right|^2\right] dl + \epsilon_n(s,x) \right),
\end{equation*}
where
\begin{equation*}
	\epsilon_n(s,x) := \int_s^{T} \mathbb{E}_{s,x} \left[ |\tilde{b}^n(l, X^{n}_l) - \tilde{b}^\infty(l, X^{n}_l)|^2 + |\sigma^n(l, X^{n}_l) - \sigma^\infty(l, X^{n}_l)|^2 \right] dl.
\end{equation*}
For $N \in \mathbb{N}$ large enough, truncating the expectation over the domain $D_N$ yields
\begin{align*}
	\epsilon_n(s,x) \leq T \left( \|\tilde{b}^n - \tilde{b}^\infty\|_{\mathcal{C}^0(D_N)}^2 + \|\sigma^n - \sigma^\infty\|_{\mathcal{C}^0(D_N)}^2 \right) + \frac{A_0 \mathbb{E}_{s,x}[\sup_{l \in [s,T]}|X^{n}_l|^2]}{N^2}.
\end{align*}
By Lemma \ref{lm:Phi.continuous0} and the uniform bound on the second moments of $X^n$, taking the supremum over $(s,x)\in D_{N_0+1}(t,0)$ and $t\in[0,T]$, and first sending $n\to\infty$ and then sending $N\to\infty$, we obtain $$\sup_{t\in[0,T]}\sup_{(s,x) \in D_{N_0+1}(t,0)} \epsilon_n(s,x) \to 0.$$ Gronwall's inequality yields
\begin{equation}\label{ineq:gronwall1}
\begin{aligned}
&\sup_{t\in[0,T]}\sup_{(s,x) \in D_{N_0+1}(t,0)} \mathbb{E}_{s,x}\left[ \sup_{l \in [s,T]} \left|X^{n}_l - X^{\infty}_l\right|^2\right]
\leq A_0 \left(\sup_{t\in[0,T]}\sup_{(s,x) \in D_{N_0+1}(t,0)} \epsilon_n(s,x)\right) e^{A_0 T} \to 0,
\end{aligned}
\end{equation}
as $n\rightarrow\infty$. This uniform $L^2$-convergence ensures that the first term on the right-hand side of \eqref{ineq:II} vanishes uniformly for $(s,x) \in D_{N_0+1}(t,0)$ and $t\in[0,T]$. 
{Moreover, by Lemma \ref{lm:hitting.uniform},
\begin{equation}\label{eq:prob_uniform_fixed1}
	\lim_{n\to\infty}\sup_{t\in[0,T]} \sup_{(s,x) \in D_{N_0+1}(t,0)} \mathbb{P}_{s,x} \left( |\rho^n_{s,N_1}\wedge T - \rho^\infty_{s,N_1}\wedge T| > \epsilon \right) = 0.
\end{equation}}
It follows directly from \eqref{eq:prob_uniform_fixed1} that
\begin{equation*}
	\sup_{t\in[0,T]}\sup_{(s,x) \in D_{N_0+1}(t,0)} \mathbb{E}_{s,x} [|\rho^n_{s,N_1}\wedge T - \rho^\infty_{s,N_1}\wedge T|] \leq \epsilon  + T\sup_{t\in[0,T]} \sup_{(s,x) \in D_{N_0+1}(t,0)} \mathbb{P}_{s,x}(|\rho^n_{s,N_1}\wedge T - \rho^\infty_{s,N_1}\wedge T| > \epsilon).
\end{equation*}
Letting $n \to \infty$ first and then $\epsilon \to 0$, the expectation vanishes uniformly. By setting $A^n_{\epsilon} := \{|\rho^n_{s,N_1}\wedge T -\rho^\infty_{s,N_1}\wedge T|  > \epsilon\}$, the third term on the right-hand side of \eqref{ineq:II} is decomposed as
\begin{align*}
	\mathbb{E}_{s,x} \left[ |X^\infty_{T \wedge \rho^n_{s,N_1}} - X^\infty_{T \wedge \rho^\infty_{s,N_1}}|^\alpha \right] &= \mathbb{E}_{s,x} \left[ |X^\infty_{T \wedge \rho^n_{s,N_1}} - X^\infty_{T \wedge \rho^\infty_{s,N_1}}|^\alpha \mathbf{1}_{(A^n_\epsilon)^c} \right]  + \mathbb{E}_{s,x} \left[ |X^\infty_{T \wedge \rho^n_{s,N_1}} - X^\infty_{T \wedge \rho^\infty_{s,N_1}}|^\alpha \mathbf{1}_{A^n_\epsilon} \right].
\end{align*}
On the event $(A^n_\epsilon)^c$, it holds that
\begin{equation*}
	\mathbb{E}_{s,x} \left[ |X^\infty_{T \wedge \rho^n_{s,N_1}} - X^\infty_{T \wedge \rho^\infty_{s,N_1}}|^\alpha \mathbf{1}_{(A^n_\epsilon)^c} \right]  \leq C \epsilon^{\alpha/2},
\end{equation*}
where the last inequality follows from the BDG inequality and the uniform boundedness of coefficients. For the second term, H\"older's inequality with $p = 2/\alpha$ yields
\begin{align*}
	\mathbb{E}_{s,x} \left[ |\tilde{X}^\infty_{T \wedge \rho^n_{s,N_1}} - X^\infty_{T \wedge \rho^\infty_{s,N_1}}|^\alpha \mathbf{1}_{A^n_\epsilon} \right] &\leq \left( \mathbb{E}_{s,x} [|X^\infty_{T \wedge \rho^n_{s,N_1}} - X^\infty_{T \wedge \rho^\infty_{s,N_1}}|^2] \right)^{\alpha/2} \left( \mathbb{P}_{s,x}(A^n_\epsilon) \right)^{1-\alpha/2} \\
	&\leq \left( 2 \mathbb{E}_{s,x} [\sup_{l \in [s,T]} |X^\infty_l|^2] \right)^{\alpha/2} \left( \mathbb{P}_{s,x}(A^n_\epsilon) \right)^{1-\alpha/2}.
\end{align*}
Under Assumption \ref{assume.r}, $\mathbb{E}_{s,x}[\sup_{l \in [s,T]} |X^\infty_l|^2]$ is uniformly bounded. Combining these with \eqref{eq:prob_uniform_fixed1}, we obtain
\begin{equation*}
	\limsup_{n \to \infty}\sup_{t\in[0,T]} \sup_{(s,x) \in D_{N_0+1}(t,0)} \mathbb{E}_{s,x} \left[ |X^\infty_{T \wedge \rho^n_{s,N_1}} - X^\infty_{T \wedge \rho^\infty_{s,N_1}}|^\alpha \right] \leq C \epsilon^{\alpha/2}.
\end{equation*}
As $\epsilon$ is arbitrary, letting $\epsilon \to 0$ yields the uniform convergence of the third term. Consequently, it holds that
\begin{equation*}
	\sup_{t\in[0,T]}\| \text{(II)}\|_{\mathcal{C}^{0}(D_{N_0+1}(t,0))} \to 0, \quad \text{as } n \to \infty.
\end{equation*}
As a result, we conclude that
\begin{equation*}
    \sup_{t\in[0,T]}\|V^n_{N_1} (t,\cdot,\cdot) - V^\infty_{N_1} (t,\cdot,\cdot)\|_{\Cc^{0}(D_{N_0+1}(t,0))} \to 0, \quad \text{as } n \to \infty.
\end{equation*}
This, together with \eqref{eq:lmvbarN.eps}, yields
\begin{equation*}
    \sup_{t\in[0,T]} \|V^n(t,\cdot,\cdot) - V^\infty(t,\cdot,\cdot)\|_{\Cc^0(D_{N_0+1}(t,0))} \to 0, \quad \text{as } n \to \infty.
\end{equation*}
Similarly, the convergence also holds for  $\partial_t V^n(t,\cdot,\cdot) - \partial_tV^\infty(t,\cdot,\cdot)$. The desired result \eqref{eq:lmC12.vbarbound} holds.

By Lemma \ref{lm:Phi.continuous0} and the uniform bound $\sup_{n}\|V^n\|_{\widetilde{\Cc}^{1, 2}_{\alpha,[0, T]}}\leq A_{\lambda}$, the term $f^n$ in \eqref{eq:lmC12.fn} satisfies
 \be\label{eq:lmC12.fnvanish}
 \sup_{t\in[0,T]}\|f^n(t,\cdot,\cdot)\|_{\Cc_{\beta}(D_{N_0+1}(t,0))}+\sup_{t\in[0,T]}\|\partial_{t}f^n(t,\cdot,\cdot)\|_{\Cc_{\beta}(D_{N_0+1}(t,0))}\to 0 \text{ as $n\to \infty$}.
 \ee
Applying the Schauder estimate to \eqref{eq:PDF.Vbar} and its $t$-derivative yields that
{\small\begin{align}
 &\|\bar V^n(t,\cdot,\cdot)\|_{\Cc^{1,2}_{\beta}(D_{N_0}(t,0))}+ \|\partial_t\bar V^n(t,\cdot,\cdot)\|_{\Cc^{1,2}_{\beta}(D_{N_0}(t,0))}\nonumber\\&\leq C\bigg(  \|\bar V^n(t,\cdot,\cdot)\|_{\Cc^{0}(D_{N_0+1}(t,0))} + \|f^n(t,\cdot,\cdot)\|_{\Cc_{\beta}(D_{N_0+1}(t,0))} +\|F(t,\cdot,m^\infty_T)-F(t,\cdot,m^n_T)\|_{\Cc^2_{\beta}(B_{N_0+1}(0))} \label{eq:lm.barvest} \\
 &\quad+\|\partial_t\bar V^n(t,\cdot,\cdot)\|_{\Cc^{0}(D_{N_0+1}(t,0))} + \|\partial_tf^n(t,\cdot,\cdot)\|_{\Cc_{\beta}(D_{N_0+1}(t,0))} +\|\partial_tF(t,\cdot,m^\infty_T)-\partial_tF(t,\cdot,m^n_T)\|_{\Cc^2_{\beta}(B_{N_0+1}(0))}\bigg).\nonumber
\end{align}}
Following the same argument as in \eqref{est:dig2}, we  conclude that
\begin{align*}
     \|\bar{J}^{n}\|_{\Cc^{0,2}_{\beta}(D_{N_0})} &\leq\sup_{t\in [0,T]} \|\partial_t\bar{V}^{n}(t,\cdot,\cdot)\|_{\Cc^{0,2}([t,T]\times B_{N_0}(0))}+ \sup_{t\in[0,T]} \|\bar{V}^{n}(t,\cdot,\cdot)\|_{\Cc^{0,2}_{\beta}([t,T]\times B_{N_0}(0))}.
\end{align*}
\eqref{eq:lemma3.4}, \eqref{eq:lmC12.vbarbound}, \eqref{eq:lmC12.fnvanish}, and \eqref{eq:lm.barvest} together imply that the right-hand side of the above equation tends to zero as $n\to\infty$, which gives \eqref{eq:lm.continuous1}.

Finally, we show $d(\mu^{\pi^n, m^n}, \mu^{\pi^\infty, m^\infty}) \to 0$. By \eqref{ineq:w2}, we have
\begin{align*}
    d^2(\mu^{\pi^n, m^n}, \mu^{\pi^\infty, m^\infty})\leq \sup_{t\in[0,T]}\mathbb{E}\left[|X^{\pi^n, m^n}_t-X^{\pi^\infty, m^\infty}_t|^2\right],
\end{align*}
where for $n\in\mathbb{N}\cup\{\infty\}$,
\begin{align*}
    dX^{\pi^n, m^n}_t=\tilde{b}(t,X^{\pi^n, m^n}_t,m^n_t,\pi^n(t,X^{\pi^n, m^n}_{t}))dt+\sigma(t,X^{\pi^n, m^n}_{t},m^n_t)dW_t,\quad X^{\pi^n, m^n}_{0}=\xi\sim \nu.
\end{align*}
Then, it follows from the same spatial truncation and Gronwall inequality arguments used to derive \eqref{ineq:gronwall1} that
\begin{equation*}
    \sup_{t \in [0,T]} \mathbb{E}\left[ \left|X^{\pi^n, m^n}_t - X^{\pi^\infty, m^\infty}_t\right|^2\right] \leq A_0 \epsilon_n e^{A_0 T},
\end{equation*}
where
\begin{equation*}
    \epsilon_n := \int_0^T \mathbb{E} \left[ |\tilde{b}^n(s, X^{\pi^n, m^n}_s) - \tilde{b}^\infty(s, X^{\pi^n, m^n}_s)|^2 + |\sigma^n(s, X^{\pi^n, m^n}_s) - \sigma^\infty(s, X^{\pi^n, m^n}_s)|^2 \right] ds
\end{equation*}
tends to $0$ as $n \to \infty$. Thus, $d(\mu^{\pi^n, m^n}, \mu^{\pi^\infty, m^\infty}) \to 0$, which completes the proof.
\end{proof}

	We are now ready to prove Theorem \ref{thm:existence_regu} Part (ii).
	\begin{proof}[{\bf Proof of Theorem \ref{thm:existence_regu} Part (ii):}]

		Fix $0<\lambda<\lambda_0$ and an arbitrary $0\leq\beta<\alpha$. Lemmas \ref{lm:iteration.estimateentropy} and \ref{lm:Phi.continuous} show that the mapping $\Phi_\lambda$ defined in \eqref{eq:Phi.fixedpoint} is a continuous map from the compact convex set $\Ec_\lambda$ to itself, equipped with the product topology induced by the norm $\|\cdot\|_{\Cc^{0,2,\wg}_{\beta}([0,T]\times\R^d)}$ and the metric $d(\cdot, \cdot)$. By Schauder's fixed-point theorem (see, e.g., \cite[Theorem 5.28]{Rudin1991functional}), $\Phi_\lambda$ has a fixed point $(w_{\lambda}^*,m_{\lambda}^*)\in \Ec_\lambda$ such that $w_{\lambda}^*(t,x)=V^{\pi_{\lambda}^*,m_{\lambda}^*}_\lambda(t,t,x)$ with $\pi^*_{\lambda}(t,x,a)= \Gamma_{\lambda}(t,x, D_xw_{\lambda}^*(t,x),m^*_t, a)$. As a result, Lemma \ref{lm:iteration.estimateentropy} yields that $V^{\pi_{\lambda}^*,m_{\lambda}^*}_\lambda$ satisfies \eqref{eq:cha.HJBentropy_prime} with $m$ replaced by $m^*$. Moreover, the estimates of $V^{\pi_{\lambda}^*,m_{\lambda}^*}_\lambda$ in \eqref{eq:thm.entropy} are direct consequences of $(w_{\lambda}^*,m_{\lambda}^*)\in \Ec_\lambda$ and Lemma \ref{lm:iteration.estimateentropy}. By Theorem \ref{thm:existence_regu} (i), $(\pi^*,m^*)$ is a regularized equilibrium with entropy parameter $\lambda$.
	\end{proof}

	\section{Existence of Equilibrium by Vanishing Entropy Regularization}\label{sec:convergence}

	In this section, we return to the original time-inconsistent MFGs without entropy regularization. Our goal is to establish equilibrium existence by proving the appropriate convergence from the entropy-regularized problem to the original one as $\lambda\rightarrow 0+$.

Take a sequence $(\lambda_n)_{n\in \mathbb{N}}$ with $\lambda_n \to 0+$. For each $n\in \mathbb{N}$, Theorem \ref{thm:existence_regu} guarantees the existence of a regularized equilibrium $(\pi^n, m^n)$ with entropy parameter $\lambda_n$, satisfying the uniform upper bounds in \eqref{eq:thm.entropy}. Denote $V^n(t,s, x) := V^{\pi^n, m^n}_{\lambda_n}(t,s, x)$ for each $n$.
It then follows from the construction of the regularized equilibria that $(V^n, \pi^n, m^n)$ satisfies
\begin{align}
    \partial_s V^n (t,s, x) &+ \frac{1}{2}\operatorname{tr}\left((\sigma\sigma^T)(s,x,m^n_s) D^2_x V^n (t,s, x)\right) + \tilde{b}(s,x,m^n_s,\pi^n(s,x)) \cdot D_x V^n(t,s, x) \notag \\
    &+ \tilde{r}(s-t,x, m^n_s,\pi^n(s,x)) + \lambda_n \delta(s-t)\mathcal{H}(\pi^n(s,x)) = 0, \label{eq:HJBn1} \\
    V^n(t,T,x) &= F(t,x,m^n_T), \label{eq:HJBn2} \\
    \pi^n(t,x,a)&=  \Gamma_{\lambda_n}(t,x,D_x V^n(t,t,x),m^n_t,a),\label{eq:pin}   \\
    m^n_t &= \operatorname{law}(X^{\pi^n,m^n}_t), \label{eq:consisn}
\end{align}
where
\begin{equation*}
    dX^{\pi^n,m^n}_t = \tilde{b}(t,X^{\pi^n,m^n}_t,m^n_t,\pi^n(t,X^{\pi^n,m^n}_t))dt + \sigma(t,X^{\pi^n,m^n}_t,m^n_t)dW_t, \quad X^{\pi^n,m^n}_0 = \xi \sim \nu.
\end{equation*}
 We show that a subsequential limit of $(\pi^n, m^n)_{n\in \N}$ satisfies Definition \ref{def:equi.relaxed}. The following theorem is the main result of this section.
	\begin{Theorem}\label{thm:equi.existence}
	Let Assumptions \ref{assume.r} and \ref{assume.lipsa.U} hold. Then there exists an equilibrium $(\pi^*, m^*)$ satisfying Definition \ref{def:equi.relaxed}.
	\end{Theorem}

	\subsection{Proof of Theorem \ref{thm:equi.existence}}


\begin{Lemma}\label{lm:thm.C12andyoung}
	There exist a subsequence $(V^{n_k}, \pi^{n_k}, m^{n_k})_{k\in \mathbb{N}}$ and a triple $(V^\infty, \pi^\infty, m^\infty)$ satisfying the following:
	\begin{itemize}
		\item[(i)] $V^\infty\in \Cc^{0,1}_{\alpha, [0,T]}\cap\widetilde{W}^{1,2}_{p, [0,T]}$ and for each $N\in \mathbb{N}$, $t\in[0,T]$ and any test function $\phi\in L^q(D_N(t,0))$ with $1/p+1/q=1$, it holds that
		\begin{equation}\label{eq:lm.converge}
		\begin{aligned}
			&\lim_{k\to\infty}\|V^{n_k}-V^\infty\|_{\Cc^{0,1}_{\beta,[0,T]\times B_{N}(0)}}=0, \quad \forall 0\leq\beta<\alpha,\\
			&\lim_{k\to\infty}\int_{D_N(t,0)} \phi(s,x) \left(\partial^l_sD^a_{x} V^{n_k}(t,s, x)-\partial^l_sD_{x}^a V^\infty(t,s, x) \right) dsdx = 0\quad \forall 0\leq 2l+|a|_{l^1}\leq 2,\\
			&   \|V^\infty\|_{\Cc^{0,1}_{\alpha,[0,T]}}\vee\|V^\infty\|_{\widetilde{W}^{1,2}_{p,[0,T]}}\leq A^*.
		\end{aligned}
		\end{equation}
		Moreover, $D_x V^\infty(t,s,x)$ is uniformly Lipschitz continuous with respect to $t$, i.e., there exists a constant $C>0$ such that
  \begin{equation}\label{eq:lm.lipschitz_t}
    |D_x V^\infty(t_1,s,x) - D_x V^\infty(t_2,s,x)| \leq C|t_1 - t_2|,
\end{equation}
for all $t_1, t_2 \in [0,T]$ and $(s,x) \in [0,T]\times\mathbb{R}^d$ with $s \geq t_1 \vee t_2$.
		\item[(ii)] $m^\infty\in \mathscr{M}_{\nu}^{\kappa,C^*}$ such that $d(m^{n_k}, m^\infty)\to 0$ as $k\to\infty$. The policy $\pi^\infty:[0,T]\times\R^d\rightarrow\mathscr{P}(U)$ is Borel measurable, and $\pi^{n_k}$ converges to $\pi^\infty$ in the sense that
		\begin{equation}\label{eq:convergence.young}
		\lim_{k\to\infty}\int_{[0,T]\times\R^d} \left(\int_U \phi(t, x,a)\pi^{n_k}(t, x,a)da\right)dtdx = \int_{[0,T]\times\R^d} \left(\int_U \phi(t,x,a)\pi^\infty(t, x,da)\right)dtdx
		\end{equation}
		for any test function $\phi(t, x,a):[0,T]\times\R^d\times U\rightarrow \R$ that is continuous in the control variable $a$ for almost every $(t,x)$ and
		\begin{equation*}
		\int_{[0,T]\times\R^d} \max_{a} |\phi(t, x,a)| dtdx<\infty.
		\end{equation*}
		Moreover, for any fixed $t \in [0,T]$,
      {\small \begin{equation}\label{eq:weakstar}
       \begin{aligned}
    \tilde{b}(s,x,m^{n_k}_s,\pi^{n_k}(s,x)) &\to \tilde{b}(s,x,m^{\infty}_s,\pi^{\infty}(s,x))\quad \text{in the weak-$\ast$ topology of } L^\infty([0,T]\times\R^d), \\
    \tilde{r}(s-t,x,m^{n_k}_s,\pi^{n_k}(s,x)) &\to \tilde{r}(s-t,x,m^{\infty}_s,\pi^{\infty}(s,x))\quad \text{in the weak-$\ast$ topology of } L^\infty([t,T]\times\R^d).
       \end{aligned}
      \end{equation}}


		\item[(iii)] $V^{\pi^\infty, m^\infty}=V^\infty$.
	\end{itemize}
\end{Lemma}

	\begin{proof}
{\bf Part (i).} We first show that the limit of a subsequence exists in the space $\Cc^{0,1}_{\alpha, [0,T]}\cap \widetilde{W}^{1,2}_{p, [0,T]}$ using a diagonal argument.

In view of Theorem \ref{thm:existence_regu}, for each $N \in \mathbb{N}$, it holds that
\begin{align}
&\sup_{n\in \mathbb{N}}\left(\sup_{t\in[0,T]}\|V^{n}(t,\cdot,\cdot)\|_{\Cc^{0,1}_{\alpha}(D_N(t,0))}+\sup_{t\in[0,T]}\|\partial_tV^{n}(t,\cdot,\cdot)\|_{\Cc^{0,1}_{\alpha}(D_N(t,0))}\right)\leq A^*, \label{ineq:lm41}  \\
&\sup_{n\in \mathbb{N}}\left(\sup_{t\in[0,T]}\|V^{n}(t,\cdot,\cdot)\|_{ W^{1,2}_{p}(D_N(t,0))}+\sup_{t\in[0,T]}\|\partial_tV^{n}(t,\cdot,\cdot)\|_{ W^{1,2}_{p}(D_N(t,0))}\right)\leq A_N A^*,\label{ineq:lm42}
\end{align}
where $A^*$ is the finite constant stated in Theorem \ref{thm:existence_regu}, and $A_N$ is a finite constant depending solely on $d, p$, and $N$.


Take $(V^{0,n})_{n\in \mathbb{N}}=(V^n)_{n\in \mathbb{N}}$ and fix an arbitrary $0\leq \beta<\alpha$. We repeat the following discussion for each $N=T,T+1,\cdots$. Note that $\Cc^{0,1}_{\alpha}(D_N(t,0))$ is compactly embedded in $\Cc^{0,1}_{\beta}(D_N(t,0))$. Furthermore, \eqref{ineq:lm41} implies that  the mapping $t \mapsto V^n(t,\cdot,\cdot) \in \Cc^{0,1}_{\beta}(D_N(t,0))$ is uniformly equicontinuous. Hence, by the Arzel\`a-Ascoli theorem, we can find a subsequence $(V^{N-1, n_k})_{k\in \mathbb{N}}$ of $(V^{N-1,n})_{n\in \mathbb{N}}$ and a limit function $u^N$ defined on $\Delta_{[0,T]}\times B_N(0)$ such that
\begin{equation*}
    \lim_{k\to\infty}\sup_{t\in[0,T]}\|V^{N-1, n_k}(t,\cdot,\cdot)-u^N(t,\cdot,\cdot)\|_{\Cc^{0,1}_{\beta}(D_N(t,0))}= 0.
\end{equation*}
By \eqref{ineq:lm42}, the Banach–Alaoglu theorem and the Arzelà-Ascoli theorem yield a further subsequence of $(V^{N-1, n_k})_{k\in \mathbb{N}}$, denoted by $(V^{N, n})_{n\in \mathbb{N}}$, such that for any $t\in[0,T]$ and any test function $\phi \in L^q(D_N(t,0))$,
\begin{align*}
	\lim _{n \to \infty} \int_{D_N(t,0)} \phi(s, x) \partial_s V^{N, n}(t,s, x) dsdx &= \int_{D_N(t,0)} \phi(s, x) w(t,s, x) dsdx, \\
	\lim _{n \to \infty} \int_{D_N(t,0)} \phi(s, x) \partial_{x_i x_j}^2 V^{N, n}(t,s, x) dsdx &= \int_{D_N(t,0)} \phi(s, x) w^{i j}(t,s, x) dsdx, \quad \forall 1 \leq i, j \leq d,
\end{align*}
where the limit functions satisfy $w(t,\cdot,\cdot),\,w^{i j}(t,\cdot,\cdot)\in L^p(D_N(t,0))$. Then, for any test function $\phi \in C_c^{\infty}(D_N(t,0))$ with $\partial_s \phi, \partial_{x_i} \phi \in L^q(D_N(t,0))$, we have for each $1 \leq i, j \leq d$ that
\begin{align*}
	-\int_{D_N(t,0)} \phi(s, x) w^{i j}(t,s, x) dsdx &= -\lim _{n \to \infty} \int_{D_N(t,0)} \phi(s, x) \partial_{x_i x_j}^2 V^{N, n}(t,s, x) dsdx \\
	&= \lim _{n \to \infty} \int_{D_N(t,0)} \partial_{x_i} \phi(s, x) \partial_{x_j} V^{N, n}(t,s, x) dsdx \\
    &= \int_{D_N(t,0)} \partial_{x_i} \phi(s, x) \partial_{x_j} u^N(t,s, x) dsdx.
\end{align*}
Thus, $w^{i j}$ is exactly the weak derivative $\partial_{x_i x_j}^2 u^N$ for each $1 \leq i, j \leq d$. A similar  argument shows that $w$ is the weak derivative $\partial_s u^N$. Therefore, $u^N(t,\cdot,\cdot) \in W^{1,2}_p(D_N(t,0))$ for any $t\in[0,T]$, and
\begin{equation}\label{eq:lm.C12}
    \lim_{n\to\infty}\int_{D_N(t,0)} \phi(s,x) \left(\partial^l_s D^a_{x} V^{N, n}(t,s,x)-\partial^l_s D_{x}^a u^N(t,s,x) \right) dsdx = 0,
\end{equation}
for all $2l+|a| \leq 2$ and all $\phi\in L^q(D_N(t,0))$.

As a result, \eqref{eq:lm.C12} and the following convergence hold for the subsequence $(V^{N, n})_{n\in \mathbb{N}}$:
\begin{equation*}
    \lim_{n\to\infty}\sup_{t\in[0,T]}\|V^{N, n}(t,\cdot,\cdot)-u^N(t,\cdot,\cdot)\|_{\Cc^{0,1}_{\beta}(D_N(t,0))}= 0.
\end{equation*}
Proceeding this way for each $N$, we obtain a sequence of limits $u^N$. Note that $u^{N+1}(t,s,x)=u^N(t,s,x)$ for all $(t,s,x) \in \Delta_{[0,T]} \times B_N(0)$. Hence, there exists a well-defined global function $V^\infty\in \Cc^{0,1}_{\alpha,[0,T]}$ such that $V^\infty(t,s,x)=u^N(t,s,x) \in W^{1,2}_p(D_N(t,0))$ for each $N\in \mathbb{N}$, and it satisfies
\begin{equation*}
    \|V^\infty\|_{\Cc^{0,1}_{\alpha,[0,T]}}\leq A^*.
\end{equation*}
By taking the diagonal subsequence $(V^{N,N})_{N\in \mathbb{N}}$, we have that for each integer $K$ and each multi-index $(l,a)$ with $2l+|a| \leq 2$,
\begin{align*}
	&\lim_{N\to\infty}\sup_{t\in[0,T]}\|V^{N,N}(t,\cdot,\cdot)-V^\infty(t,\cdot,\cdot)\|_{\Cc^{0,1}_{\beta}(D_K(t,0))}=0,\\
	&\lim_{N\to\infty}\int_{D_K(t,0)} \phi(s,x) \left(\partial^l_s D^a_{x} V^{N,N}(t,s,x)-\partial^l_s D_{x}^a V^\infty(t,s,x) \right) dsdx = 0, \quad \forall \phi\in L^q(D_K(t,0)).
\end{align*}
Moreover, by fixing $t\in[0,T]$ and an arbitrary point $(t_0,x_0)\in [t,T]\times \R^d$, the weak convergence implies that for any $\phi\in L^q(D_1(t_0,x_0))$,
\begin{equation*}
    \lim_{N\to\infty}\int_{D_1(t_0,x_0)} \phi(s,x) \left(\partial^l_s D^a_{x} V^{N, N}(t,s,x)-\partial^l_s D_{x}^a V^\infty(t,s,x) \right) dsdx = 0.
\end{equation*}
Note that the dual representation of the $L^p$ norm for a generic function $g\in L^p(D_1(t_0,x_0))$ is given by
\begin{equation*}
    \|g\|_{L^p(D_1(t_0,x_0))} = \sup_{\|\phi\|_{L^q(D_1(t_0,x_0))}\leq 1} \int_{D_1(t_0,x_0)} \phi(s,x) g(s,x) dsdx.
\end{equation*}
Then, for any test function $\phi$ with $\|\phi\|_{L^q(D_1(t_0,x_0))}\leq 1$, it holds that
\begin{align*}
	\int_{D_1(t_0,x_0)} \phi(s,x) \partial^l_s D_{x}^a V^\infty(t,s,x) dsdx &= \lim_{N\to\infty} \int_{D_1(t_0,x_0)} \phi(s,x) \partial^l_s D^a_{x} V^{N, N}(t,s,x) dsdx \\
	&\leq \liminf_{N\to\infty} \|\partial^l_s D^a_x V^{N,N}(t,\cdot,\cdot)\|_{L^p(D_1(t_0,x_0))},
\end{align*}
which  implies that
\begin{equation*}
    \|\partial^l_s D^a_x V^\infty(t,\cdot,\cdot)\|_{L^p(D_1(t_0,x_0))} \leq \liminf_{N\to\infty} \|\partial^l_s D^a_x V^{N,N}(t,\cdot,\cdot)\|_{L^p(D_1(t_0,x_0))} \leq A^*, \quad \forall 2l+|a| \leq 2.
\end{equation*}
Taking the supremum over all $(t_0, x_0)$ yields that $\|V^\infty\|_{\widetilde{W}^{1,2}_{p,[0,T]}}\leq A^*$.

Finally, we verify the uniform Lipschitz continuity of $D_x V^\infty$ with respect to the parameter $t$. Recall from Lemma \ref{lm:iteration.estimateentropy} (specifically the uniform estimates on $\partial_t V^n$) that
$$\sup_{n\in \mathbb{N}} \sup_{t\in[0,T]} \|\partial_t V^n(t,\cdot,\cdot)\|_{\Cc^{0,1}_\alpha([t,T]\times\R^d)} \leq A^*.$$ Because the mixed derivatives are consistent (i.e., $\partial_t D_x V^n = D_x \partial_t V^n$), we have $$\sup_{n\in \mathbb{N}} \sup_{t\in[0,T]} \|\partial_t D_x V^n(t,s,x)\|_{\Cc^0([t,T]\times\R^d)} \leq A^*,$$ which implies that $D_x V^n$ is uniformly Lipschitz continuous in $t$, independently of $n, s$, and $x$. That is, for any $t_1, t_2 \in[0,T]$ and $s \geq t_1 \vee t_2$,
\begin{equation*}
   |D_x V^{n_k}(t_1,s,x) - D_x V^{n_k}(t_2,s,x)| \leq A^* |t_1 - t_2|.
\end{equation*}
As $D_x V^{n_k} \to D_x V^\infty$ pointwise as $k \to \infty$, taking the limit $k \to \infty$ in the above inequality gives \eqref{eq:lm.lipschitz_t}.

{\bf Part (ii).}
Next, we consider the associated sequence $(m^{N,N}, \pi^{N,N})_{N\in \mathbb{N}}$ corresponding to the sequence $(V^{N,N})_{N\in \mathbb{N}}$. Note that the control space $U$ is compact. By the compactness of $\mathscr{M}_{\nu}^{\kappa,C^*}$ and Young measure theory (see, e.g., \cite[Chapter IV]{warga2014optimal}), there exists a subsequence $(m^{n_k}, \pi^{n_k})_{k\in \mathbb{N}}$ of $(m^{N,N}, \pi^{N,N})_{N\in \mathbb{N}}$ and a limit $(m^\infty, \pi^\infty)$ such that
\begin{align*}
&m^\infty\in \mathscr{M}_{\nu}^{\kappa,C^*}, \text{ and } \lim_{k\to\infty}d(m^{n_k}, m^\infty)=0,\\
&\pi^\infty: [0,T]\times\R^d\to\mathscr{P}(U) \text{ is Borel measurable, and the convergence in \eqref{eq:convergence.young} holds}.
\end{align*}
Moreover, for any test function $\phi \in L^1([0,T]\times\mathbb{R}^d)$, by Assumptions \ref{assume.r} and \ref{assume.lipsa.U}, we have
\begin{equation*}
\begin{cases}
\text{The mapping $a \rightarrow b(t,x, m^{\infty}_t, a)\phi(t,x)$ is continuous for every $(t,x)$},\\
\int_{[0,T]\times\mathbb{R}^d} \sup_{a \in U} |b(t,x, m_t^{\infty},a) \phi(t,x)| dtdx \leq K_1 \int_{[0,T]\times\mathbb{R}^d} |\phi(t,x)| dtdx < \infty.
\end{cases}
\end{equation*}
Note that
\begin{align*}
    &\left| \int_{[0,T]\times\mathbb{R}^d} \tilde{b}(t,x,m^{n_k}_t,\pi^{n_k}(t,x)) \phi(t,x) dtdx - \int_{[0,T]\times\mathbb{R}^d} \tilde{b}(t,x,m^{\infty}_t,\pi^{\infty}(t,x)) \phi(t,x) dtdx \right| \\
    &\leq \int_{[0,T]\times\mathbb{R}^d} \left| \int_U[b(t,x, m^{n_k}_t,a) - b(t,x, m^{\infty}_t,a)] \pi^{n_k}(t,x, a) da \right| |\phi(t,x)| dtdx \\
    &\quad + \bigg| \int_{[0,T]\times\mathbb{R}^d} \int_U b(t,x, m^{\infty}_t,a) \phi(t,x) \pi^{n_k}(t,x, a) da dtdx \\
    &\;\qquad - \int_{[0,T]\times\mathbb{R}^d} \int_U b(t,x, m^{\infty}_t,a) \phi(t,x) \pi^{\infty}(t,x, a) da dtdx \bigg|.
\end{align*}
For the first term, the Lipschitz continuity of $b$ with respect to $m$ in Assumption \ref{assume.r} implies that
\begin{align*}
   & \int_{[0,T]\times\mathbb{R}^d} \left| \int_U[b(t,x, m^{n_k}_t,a) - b(t,x, m^{\infty}_t,a)] \pi^{n_k}(t,x, a) da \right| |\phi(t,x)| dtdx \\
    &\leq K_2 \sup_{t\in[0,T]} \mathcal{W}_2(m^{n_k}_t, m^\infty_t) \int_{[0,T]\times\mathbb{R}^d} |\phi(t,x)| dtdx,
\end{align*}
which tends to $0$ as $k \to \infty$ since $d(m^{n_k}, m^\infty) \to 0$.
For the second term, by \eqref{eq:convergence.young}, we have
\begin{align*}
&\lim _{k \rightarrow \infty} \int_{[0,T]\times\mathbb{R}^d} \tilde{b}(t,x,m^{\infty}_t,\pi^{n_k}(t,x)) \phi(t,x) dtdx \\
&= \lim _{k \rightarrow \infty} \int_{[0,T]\times\mathbb{R}^d} \left(\int_U[b(t,x, m^{\infty}_t,a) \phi(t,x)] \pi^{n_k}(t,x, a) da\right) dtdx \\
&= \int_{[0,T]\times\mathbb{R}^d} \left(\int_U[b(t,x,m^{\infty}_t, a) \phi(t,x)] \pi^{\infty}(t,x, a) da\right) dtdx \\
&= \int_{[0,T]\times\mathbb{R}^d} \tilde{b}(t,x,m^{\infty}_t,\pi^{\infty}(t,x)) \phi(t,x) dtdx.
\end{align*}
Therefore, we obtain
\begin{equation*}
    \lim_{k \rightarrow \infty} \int_{[0,T]\times\mathbb{R}^d} \tilde{b}(t,x,m^{n_k}_t,\pi^{n_k}(t,x)) \phi(t,x) dtdx = \int_{[0,T]\times\mathbb{R}^d} \tilde{b}(t,x,m^{\infty}_t,\pi^{\infty}(t,x)) \phi(t,x) dtdx.
\end{equation*}
An analogous argument holds for the reward function $r(s-t,x,m_s,a)$. Thus, the weak-$\ast$ convergence result in \eqref{eq:weakstar}  is established.

{\bf Part (iii).} We next show that $V^{\pi^\infty,m^{\infty}}(t,s,x) = V^\infty(t,s,x)$ for any $(t,s,x) \in \Delta_{[0,T]}\times \R^d$. Fix an arbitrary $(s,x)\in [t,T]\times \R^d$. By the boundedness of $b$ and the conditions on $\sigma$ stated in Assumption \ref{assume.r}, the SDE
\begin{equation*}
    dX^{\infty}_l = \tilde{b}(l,X_l^{\infty},m_l^{\infty},\pi^{\infty}(l,X^{\infty}_l))dl + \sigma(l,X^{\infty}_l,m^{\infty}_l)dW_l \quad \text{for $l\geq s$ with } X^{\infty}_s = x,
\end{equation*}
admits a unique strong solution (see \cite[Theorem 1]{veretennikov1981strong}). By Krylov's inequality,  for any $T>0$,  $X^{\infty}_l$ admits a density function belonging to $L^q([s,T]\times \R^d)$ with $q=\frac{d+2}{d+1+\alpha}=\frac{p}{p-1}$ (see, e.g., \cite{krylov1980controlled}).
Define  $\rho_{s,N} := \inf\{l\geq s: X^{\infty}_l \notin B_N(0)\}$. For an arbitrary $\epsilon > 0$, since the coefficients are  bounded, there exists $N_0 > 0$ such that $\mathbb{P}_{s,x}(\rho_{s,N} \leq T) \leq \frac{\epsilon}{3A_0}$ for all $N \geq N_0$. Then,
\begin{align}\label{eq:lm.verify1}
    &\mathbb{E}_{s,x} \bigg[ \left|V^{\infty}(t, T\wedge\rho_{s,N}, X^{\infty}_{T\wedge\rho_{s,N}}) - F(t,X^{\infty}_T,m^{\infty}_T)\right|\mathbf{1}_{\{\rho_{s,N}<T\}} \bigg] \notag \\
    &+ \mathbb{E}_{s,x} \left[ \int_{\rho_{s,N}}^T \left|\tilde{r}(l-t,X^{\infty}_l,m^{\infty}_l,\pi^{\infty}(l,X^{\infty}_l))\right| dl \, \mathbf{1}_{\{\rho_{s,N}<T\}} \right] \leq \epsilon\quad \forall N\geq N_0.
\end{align}
Take any $N\in \mathbb{N}$ with $N \geq \max\{T, N_0\}$. In view of $V^\infty \in \Cc^{0,1}_{\alpha, [0,T]} \cap \widetilde{W}^{1,2}_{p,[0,T]}$, we can apply the It\^o-Krylov formula (see, e.g., \cite[Section 2.10, Theorem 1]{krylov1980controlled}) to $V^\infty(t, l, X^{\infty}_l)$ on the interval $l \in[s, T\wedge\rho_{s,N}]$ to get
\begin{align*}
	&V^{\infty}(t, T\wedge\rho_{s,N}, X^{\infty}_{T\wedge\rho_{s,N}}) - V^{\infty}(t,s,x) \\
	&= \int_{s}^{T\wedge\rho_{s,N}} \bigg( \partial_s V^\infty(t,l, X^{\infty}_l)+ \tilde{b}(l,X^{\infty}_l, m^{\infty}_l,\pi^{\infty}(l,X^{\infty}_l)) \cdot D_x V^\infty(t, l, X^{\infty}_l) \\
    &\qquad \qquad \quad \quad + \frac{1}{2}\operatorname{tr}\bigg((\sigma\sigma^T)(l,X^{\infty}_l,m^{\infty}_l) D^2_x V^\infty(t,l, X^{\infty}_l) \bigg) \bigg) dl \\
	&\quad+ \int_{s}^{T\wedge\rho_{s,N}} \left(D_x V^\infty(t,l, X^{\infty}_l)\right)^T \sigma(l,X^{\infty}_l,m^{\infty}_l) dW_l.
\end{align*}
Denote $p^{s,x}(l,y):= \P(X^{\infty}_l\in dy, l\leq \rho_{s,N})$. It holds that $p^{s,x}\in L^q(D_N(s,0))$. Indeed,
\begin{align*}
\int_{D_N(s,0)}\left|p^{s,x}(l, y)\right|^q d l d y&=\int_{D_N(s,0)}\left|{\mathbb{P}_{s,x}\left(X_l^{\infty} \in d y, l \leq \rho_{s,N}\right)}\right|^q dl d y
\leq  \int_{D_N(s,0)}
\left|\mathbb{P}_{s,x}\left(X_l^{\infty} \in d y\right)\right|^q d l d y<\infty.
\end{align*}
Taking expectations on both sides gives
\begin{equation*}
\begin{aligned}
&\mathbb{E}_{s,x}[V^{\infty}(t,T\wedge\rho_{s,N},X^{\infty}_{T\wedge\rho_{s,N}})]-V^{\infty}(t,s,x)\\
	&=\mathbb{E}_{s,x}\bigg[\int_{s}^{T\wedge\rho_{s,N}}\bigg(\partial_s V^\infty(t,l, X^{\infty}_l)+ \tilde{b}(l,X^{\infty}_l, m^{\infty}_l,\pi^{\infty}(l,X^{\infty}_l)) \cdot D_x V^\infty(t, l, X^{\infty}_l)\\
    &\qquad\qquad\qquad\qquad +\frac{1}{2}\operatorname{tr}\bigg((\sigma\sigma^T)(l,X^{\infty}_l,m^{\infty}_l) D^2_x V^\infty(t,l, X^{\infty}_l) \bigg)\bigg) dl\bigg]\\
	&= \int_{D_N(s,0)} \bigg[\partial_s V^\infty(t,l, y)+ \tilde{b}(l,y,m^{\infty}_l,\pi^{\infty}(l,y)) \cdot D_x V^\infty(t,l, y)\\
    &\qquad\qquad\quad+\frac{1}{2}\operatorname{tr}\left((\sigma\sigma^T)(l,y,m^{\infty}_l) D^2_x V^\infty(t,l,y) \right)\bigg] p^{s,x}(l,y) dy dl\\
	&=\lim_{k\to\infty} \int_{D_N(s,0)} \bigg[\partial_s V^{n_k}(t,l, y)+ \tilde{b}(l,y,m^{n_k}_l,\pi^{n_k}(l,y)) \cdot D_x V^{n_k}(t,l, y)\\
    &\qquad\qquad\qquad\qquad+\frac{1}{2}\operatorname{tr}\left((\sigma\sigma^T)(l,y,m^{n_k}_l) D^2_x V^{n_k}(t,l,y) \right)\bigg] p^{s,x}(l,y) dy dl\\
	&=\lim_{k\to\infty} \int_{D_N(s,0)}-\tilde{r}(l-t,y,m^{n_k}_l,\pi^{n_k}(l,y))
	p^{s,x}(l,y) dy dl\\
	&=\int_{D_N(s,0)}-\tilde{r}(l-t,y,m^{\infty}_l,\pi^{\infty}(l,y)) p^{s,x}(l,y) dy dl\\
	&= -\mathbb{E}_{s,x}\left[  \int_s^{T\wedge\rho_{s,N}} \tilde{r}(l-t,X^{\infty}_l,m^{\infty}_l,\pi^{\infty}(l,X^{\infty}_l)) dl \right],
\end{aligned}
\end{equation*}
where the third equality follows from \eqref{eq:lm.converge}, \eqref{eq:weakstar} and the continuity of $\sigma$ in $m$ (Assumption \ref{assume.r}), the fourth equality follows from \eqref{eq:HJBn1} and $\|\lambda_{n_k}\mathcal{H}(\pi^{n_k})\|_{L^\infty([0,T]\times\R^d)}\to 0$, and the fifth equality follows from \eqref{eq:weakstar} again. Together with \eqref{eq:lm.verify1}, the above result implies that
\begin{equation*}
\begin{aligned}
&|V^{\infty}(t,s,x)- V^{\pi^\infty,m^{\infty}}(t,s,x)|\\
&=\left| \mathbb{E}_{s,x}[V^{\infty}(t,T\wedge\rho_{s,N},X^{\infty}_{T\wedge\rho_{s,N}})]+\mathbb{E}_{s,x}\left[  \int_s^{T\wedge\rho_{s,N}} \tilde{r}(l-t,X^{\infty}_l,m^{\infty}_l,\pi^{\infty}(l,X^{\infty}_l)) dl \right]-V^{\pi^\infty,m^{\infty}}(t,s,x)\right|\\
&\leq \mathbb{E}_{s,x}\bigg[\left|V^{\infty}(t,T\wedge\rho_{s,N},X^{\infty}_{T\wedge\rho_{s,N}})-F(t,X^{\infty}_T,m^{\infty}_T)\right|\mathbf{1}_{\{\rho_{s,N}<T\}}\bigg]  \\
&\quad+ \mathbb{E}_{s,x}\left[ \int_{\rho_{s,N}}^T \left|\tilde{r}(l-t,X^{\infty}_l,m^{\infty}_l,\pi^{\infty}(l,X^{\infty}_l))\right| dl\mathbf{1}_{\{\rho_{s,N}<T\}} \right] \\
&\leq \epsilon.
\end{aligned}
\end{equation*}
Consequently, $V^{\infty}(t,s,x)= V^{\pi^\infty,m^{\infty}}(t,s,x)$ holds as desired.
	\end{proof}

\begin{proof}[{\bf Proof of Theorem \ref{thm:equi.existence}:}]
We show that the limit $(\pi^{\infty},m^{\infty})$ in Lemma \ref{lm:thm.C12andyoung} is indeed an equilibrium satisfying Definition \ref{def:equi.relaxed}.

{\bf Verification of Condition (a).} We first prove that \eqref{eq:def.equipi} holds for $(\pi^{\infty},m^{\infty})$. Let us begin by showing that
\begin{align}
    &\tilde{b}(t,x,m^{\infty}_t,\pi^{\infty}(t,x)) \cdot D_x V^{\infty}(t,t, x) + \tilde{r}(0, x,m^{\infty}_t,\pi^{\infty}(t,x)) \notag \\
    &\qquad = \sup _{\varpi \in \mathscr{P}(U)}\left\{\int_U\left[b(t,x,m^{\infty}_t, a) \cdot D_x V^{\infty}(t,t, x) + r(0, x, m^{\infty}_t, a)\right] \varpi(d a)\right\}, \quad \text{a.e. on } [0,T]\times\mathbb{R}^d,\label{lim:equb} \\
    &\partial_s V^{\infty}(t,s, x) + \frac{1}{2} \operatorname{tr}\left(\left(\sigma \sigma^T\right)(s,x,m^{\infty}_s) D_x^2 V^{\infty}(t,s, x)\right) + \tilde{b}(s,x,m^{\infty}_s,\pi^{\infty}(s,x)) \cdot D_x V^{\infty}(t,s, x) \notag \\
    &\qquad + \tilde{r}(s-t, x,m^{\infty}_s,\pi^{\infty}(s,x)) = 0, \quad \text{a.e. for } (s,x)\in[t,T)\times \mathbb{R}^d,\label{lim:eqhjb}\\
    &V^{\infty}(t,T, x) = F(t,x,m^{\infty}_T).\label{lim:eqhjb2}
\end{align}
We first verify that $V^{\infty}$ is a strong solution to \eqref{lim:eqhjb}--\eqref{lim:eqhjb2} by taking the limit in \eqref{eq:HJBn1}--\eqref{eq:HJBn2}. Note that \eqref{lim:eqhjb2} is a direct result of the continuity of $F$ in $m$ (Assumption \ref{assume.r}).

Recall that $V^\infty \in \Cc^{0,1}_{\alpha, [0,T]} \cap \widetilde{W}^{1,2}_{p, [0,T]}$. Fix an arbitrary $t\in[0,T)$, and let $(t_0, x_0)$ be any point belonging to $[t,T) \times \mathbb{R}^d$. A direct calculation gives, for any test function $\phi \in L^q\left(D_1\left(t_0, x_0\right)\right)$, that
\begin{align*}
    0 &= \lim _{k \rightarrow \infty} \int_{D_1\left(t_0, x_0\right)} \bigg( \partial_s V^{n_k}(t,s, x) + \frac{1}{2} \operatorname{tr}\left(\left(\sigma \sigma^T\right)(s,x,m^{n_k}_s) D_x^2 V^{n_k}(t,s, x)\right) \\
    &\qquad\qquad\qquad\quad + \tilde{b}(s,x,m^{n_k}_s,\pi^{n_k}(s,x)) \cdot D_x V^{n_k}(t,s, x) + \tilde{r}(s-t, x,m^{n_k}_s,\pi^{n_k}(s,x)) \\
    &\qquad\qquad\qquad\quad + \lambda_{n_k} \delta(s-t) \mathcal{H}\left(\pi^{n_k}(s,x)\right) \bigg) \phi(s, x) dsdx \\
    &= \int_{D_1\left(t_0, x_0\right)} \bigg( \partial_s V^{\infty}(t,s, x) + \frac{1}{2} \operatorname{tr}\left(\left(\sigma \sigma^T\right)(s,x,m^{\infty}_s) D_x^2 V^{\infty}(t,s, x)\right) \\
    &\qquad\qquad\qquad + \tilde{b}(s,x,m^{\infty}_s,\pi^{\infty}(s,x)) \cdot D_x V^{\infty}(t,s,x) + \tilde{r}(s-t, x,m^{\infty}_s,\pi^{\infty}(s,x)) \bigg) \phi(s, x) dsdx,
\end{align*}
where the second equality follows from $\lim _{k \rightarrow \infty}\left\|\lambda_{n_k} \mathcal{H}\left(\pi^{n_k}\right)\right\|_{L^{\infty}\left([0,T]\times\mathbb{R}^d\right)} = 0$, Lemma \ref{lm:thm.C12andyoung} Part (i), \eqref{eq:weakstar} in Lemma \ref{lm:thm.C12andyoung} Part (ii), and the continuity of $\sigma$ in $m$. By the arbitrariness of $\left(t_0, x_0\right)$, we conclude that $V^{\infty}$ is a strong solution to \eqref{lim:eqhjb}.

Next, we prove \eqref{lim:equb}. 
For each $n \in \mathbb{N} \cup\{\infty\}$, define
\begin{equation*}
    f^n(t,x, a) := b(t,x, m^{n}_t, a) \cdot D_x V^n(t,t, x) + r(0, x, m^n_t, a), \quad A^n(t,x) := \max _{a \in U} f^n(t,x, a).
\end{equation*}
First, for any uniformly equicontinuous sequence of functions $g^n: U \rightarrow \mathbb{R}$ converging to a limit $g^{\infty}(a)$, it holds that
\begin{equation*}
    \lim _{n \rightarrow \infty} \lambda_n \ln \left\{\int_U \exp \left[\frac{1}{\lambda_n} g^n(a)\right] d a\right\} = \max _{a \in U} g^{\infty}(a).
\end{equation*}
Therefore, for any test function $\phi \in L^1\left([0,T]\times\mathbb{R}^d\right)$,
\begin{align*}
    & \int_{[0,T]\times\mathbb{R}^d}\left[\tilde{b}(t,x,m^{\infty}_t,\pi^{\infty}(t,x)) \cdot D_x V^{\infty}(t, t,x) + \tilde{r}(0, x,m^{\infty}_t,\pi^{\infty}(t,x))\right] \phi(t,x) dtdx \\
    & = \lim _{k \rightarrow \infty} \int_{[0,T]\times\mathbb{R}^d}\big[\tilde{b}(t,x,m^{n_k}_t,\pi^{n_k}(t,x)) \cdot D_x V^{n_k}(t, t,x) \\
    &\qquad\qquad \qquad \qquad+ \tilde{r}(0, x,m^{n_k}_t,\pi^{n_k}(t,x)) + \lambda_{n_k} \mathcal{H}\left(\pi^{n_k}(t,x)\right)\big] \phi(t,x) dtdx \\
    & = \lim _{k \rightarrow \infty} \int_{[0,T]\times\mathbb{R}^d} \lambda_{n_k} \ln \left\{\int_U \exp \left[\frac{1}{\lambda_{n_k}} f^{n_k}(t,x, a)\right] da\right\} \phi(t,x) dtdx \\
    &= \int_{[0,T]\times\mathbb{R}^d} A^{\infty}(t,x) \phi(t,x) dtdx,
\end{align*}
where the first equality follows from Lemma \ref{lm:thm.C12andyoung} Parts (i) and (ii) and the fact that $$\lim _{k \rightarrow \infty}\left\|\lambda_{n_k} \mathcal{H}\left(\pi^{n_k}\right)\right\|_{L^{\infty}\left([0,T]\times\mathbb{R}^d\right)}=0.$$ Since this result holds for any $\phi \in L^1([0,T]\times\mathbb{R}^d)$, we conclude that the integrands are equal almost everywhere, and hence \eqref{lim:equb} holds.


 For any $(t,s,x) \in \Delta_{[0,T]} \times \mathbb{R}^d$, define the following functions:
\begin{align*}
    H(t,s,x) &:= \sup_{\varpi \in \mathscr{P}(U)} \left\{ \int_U \left( b(s,x,m^\infty_s, a) \cdot D_x V^\infty(t,s,x) + r(s-t,x,m^{\infty}_s,a) \right) \varpi(da) \right\}, \\
    L(t,s,x) &:= \tilde{b}(s,x, m^\infty_s,\pi^\infty(s,x)) \cdot D_x V^\infty(t,s,x) + \tilde{r}(s-t,x,m^\infty_s,\pi^\infty(s,x)).
\end{align*}
We first observe that there exists a finite constant $C > 0$, independent of $t$ and $s$, such that
\begin{equation}\label{eq:origin.thm1}
    \sup_{x\in \mathbb{R}^d} |H(t,s,x) - H(s,s,x)| \leq C|t-s|^{\frac{\alpha}{2}}, \quad \forall s\in [t,T].
\end{equation}
Indeed, for any fixed $(t,s,x)\in\Delta_{[0,T]}\times \mathbb{R}^d$, by the compactness of $U$ and the continuity of the coefficients, there exists a maximizer $\tilde{a} \in U$ such that
\begin{align*}
    H(t,s,x) &= b(s,x, m^{\infty}_s,\tilde{a}) \cdot D_x V^\infty(t,s,x) + r(s-t,x,m^{\infty}_s,\tilde{a}) \\
    &\leq b(s,x,m^{\infty}_s, \tilde{a}) \cdot D_x V^\infty(s,s,x) + r(0,x,m^{\infty}_s,\tilde{a}) + C|t-s|^{\frac{\alpha}{2}} \\
    &\leq H(s,s,x) + C|t-s|^{\frac{\alpha}{2}},\quad \forall (s,x)\in [t,T]\times \R^d.
\end{align*}
where the first inequality follows from \eqref{eq:lm.lipschitz_t} and the H\"older continuity of $r(\cdot,x,m,a)$ with respect to its first argument. By applying a symmetric argument to $H(s,s,x) - H(t,s,x)$, we obtain \eqref{eq:origin.thm1}. Similarly, we have
\begin{equation}\label{eq:origin.thm3}
       \sup_{x\in \mathbb{R}^d} |L(t,s,x) - L(s,s,x)| \leq C|t-s|^{\alpha/2}, \quad \forall s\in [t,T].
\end{equation}
Combining \eqref{eq:origin.thm1} and \eqref{eq:origin.thm3} yields that, for any fixed $t \in [0,T)$,
\begin{equation}\label{eq:origin.thm5}
\begin{aligned}
    L(t,s,x) - H(t,s,x) &= (L(t,s,x) - L(s,s,x)) + (L(s,s,x) - H(s,s,x)) + (H(s,s,x) - H(t,s,x)) \\
    &\geq L(s,s,x) - H(s,s,x) - 2C|t-s|^{\alpha/2} \\
    &\geq -2C|t-s|^{\alpha/2}, \quad \text{for a.e. } (s,x) \in [t,T] \times \mathbb{R}^d,
\end{aligned}
\end{equation}
where the last inequality follows from \eqref{lim:equb}.


Fix an arbitrary $(t,x) \in[0,T]\times \mathbb{R}^d$ and take an arbitrary $\pi' \in \mathcal{A}$. By the boundedness of $b$ and the conditions on $\sigma$ stated in Assumption \eqref{assume.r}, the SDE
$$
d X_l^{\pi'}= \tilde{b}(l,X_l^{\pi'},m^{\infty}_l,\pi'(l,X_l^{\pi'})) dl+\sigma\left(l,X_l^{\pi'},m^{\infty}_l\right) d W_l \quad \text { with } X_t^{\pi'}=x
$$
admits a unique strong solution. 
Fix an arbitrary $0<\varepsilon_0<1$ and let $\rho_{t,N}:=\inf\left\{l\geq t,X_l^{\pi'} \notin B_N(0)\right\}$. By Assumption \ref{assume.r}, for any $\epsilon$, there exists $N_0>0$ such that
$$
\mathbb{P}\left(\rho_{t,N} \leq t+\varepsilon_0\right) \leq \frac{\epsilon}{3 A_0} \quad \forall N \geq N_0 .
$$
Then, we obtain
\begin{equation}\label{pert:rho}
\begin{aligned}
& \left|\mathbb{E}_{t,x}\left[V^{\infty}\left(t,t+\varepsilon_0, X_{t+\varepsilon_0}^{\pi'}\right)\right]-\mathbb{E}_{t,x}\left[V^{\infty}\left(t,(t+\varepsilon_0 )\wedge \rho_{t,N}, X_{(t+\varepsilon_0 )\wedge \rho_{t,N}}^{\pi'}\right)\right]\right| \\
& +\left|\mathbb{E}_{t,x}\left[\int_t^{t+\varepsilon_0} \tilde{r}\left(s-t, X_s^{\pi'},m^{\infty}_s,\pi'(s,X_s^{\pi'})\right) d s\right]-\mathbb{E}_{t,x}\left[\int_t^{(t+\varepsilon_0) \wedge \rho_{t,N}} \tilde{r}\left(s-t, X_s^{\pi'},m^{\infty}_s,\pi'(s,X_s^{\pi'})\right) d s \right]\right| \\
& \leq3A_0\mathbb{P}\left(\rho_{t,N}\leq t+\varepsilon_0\right) \leq \epsilon, \quad \forall N \geq N_0.
\end{aligned}
\end{equation}
Take any $N \in \mathbb{N}$ with $N \geq N_0$. Since $V^{\infty}(t,\cdot,\cdot) \in W_p^{1,2}\left(D_N(t,0)\right) \cap \mathcal{C}_{\alpha}^{0,1}\left(D_N(t,0)\right)$, applying the generalized Itô-Krylov formula again to $V^{\infty}\left(t,l, X_l^{\pi'}\right)$ on the interval $l \in[t, (t+\epsilon_0)\wedge\rho_{t,N}]$ gives
\begin{align*}
	&V^{\infty}(t, (t+\epsilon_0)\wedge\rho_{t,N}, X^{\pi'}_{(t+\epsilon_0)\wedge\rho_{t,N}}) - V^{\infty}(t,t,x) \\
	&= \int_{t}^{(t+\epsilon_0)\wedge\rho_{t,N}} \bigg( \partial_s V^\infty(t,l, X^{\pi'}_l)+ \tilde{b}(l,X^{\pi'}_l, m^{\infty}_l,\pi'(l,X^{\pi'}_l)) \cdot D_x V^\infty(t, l, X^{\pi'}_l) \\
    &\qquad \qquad \quad + \frac{1}{2}\operatorname{tr}\bigg((\sigma\sigma^T)(l,X^{\pi'}_l,m^{\infty}_l) D^2_x V^\infty(t,l, X^{\pi'}_l) \bigg) \bigg) dl \\
	&\quad+ \int_{t}^{(t+\epsilon_0)\wedge\rho_{t,N}} \left(D_x V^\infty(t,l, X^{\pi'}_l)\right)^T \sigma(l,X^{\pi'}_l,m^{\infty}_l) dW_l.
\end{align*}
Taking expectations on both sides of the above equality yields
\begin{align}
& \mathbb{E}_{t,x}\left[V^{\infty}(t, (t+\epsilon_0)\wedge\rho_{t,N}, X^{\pi'}_{(t+\epsilon_0)\wedge\rho_{t,N}})\right]-J^{\pi^{\infty},m^{\infty}}(t,x) \label{pert:ito}\\
&\quad+\mathbb{E}_{t,x}\left[\int_t^{(t+\varepsilon_0) \wedge \rho_{t,N}} \tilde{r}\left(s-t, X_s^{\pi'},m^{\infty}_s,\pi'(s,X_s^{\pi'})\right) d s\right] \notag\\
& =\mathbb{E}_{t,x}\bigg[\int_{t}^{(t+\epsilon_0)\wedge\rho_{t,N}} \bigg\{ \partial_s V^\infty(t,l, X^{\pi'}_l)+ \tilde{b}(l,X^{\pi'}_l, m^{\infty}_l,\pi'(l,X^{\pi'}_l)) \cdot D_x V^\infty(t, l, X^{\pi'}_l)\notag\\
    &\qquad \qquad \qquad  \qquad  \quad + \frac{1}{2}\operatorname{tr}\bigg((\sigma\sigma^T)(l,X^{\pi'}_l,m^{\infty}_l) D^2_x V^\infty(t,l, X^{\pi'}_l) \bigg) +\tilde{r}\left(l-t, X_l^{\pi'},m^{\infty}_l,\pi'(l,X_l^{\pi'})\right)\bigg\} dl\bigg]. \notag
\end{align}
Using \eqref{lim:eqhjb}, \eqref{eq:origin.thm5}, \eqref{pert:rho} and \eqref{pert:ito}, we obtain
\begin{align*}
&J^{\pi' \otimes_{t,\varepsilon_0} \pi^\infty}(t,x)-J^{\pi^{\infty},m^{\infty}}(t,x)\\&=\mathbb{E}_{t,x}\left[V^{\infty}\left(t,t+\varepsilon_0, X_{t+\varepsilon_0}^{\pi'}\right)+\int_t^{t+\epsilon_0}  \tilde{r}\left(s-t, X_s^{\pi'},m^{\infty}_s,\pi'(s,X_s^{\pi'})\right) d s\right]-J^{\pi^{\infty},m^{\infty}}(t,x)\\
&\leq \mathbb{E}_{t,x}\bigg[\int_{t}^{(t+\epsilon_0)\wedge\rho_{t,N}} \bigg\{ \partial_s V^\infty(t,l, X^{\pi'}_l)+ \tilde{b}(l,X^{\pi'}_l, m^{\infty}_l,\pi'(l,X^{\pi'}_l)) \cdot D_x V^\infty(t, l, X^{\pi'}_l) \\
    &\qquad \qquad \quad \qquad\qquad + \frac{1}{2}\operatorname{tr}\bigg((\sigma\sigma^T)(l,X^{\pi'}_l,m^{\infty}_l) D^2_x V^\infty(t,l, X^{\pi'}_l) \bigg) +\tilde{r}\left(l-t, X_l^{\pi'},m^{\infty}_l,\pi'(l,X_l^{\pi'})\right)\bigg\} dl\bigg]+\epsilon\\
&=\mathbb{E}_{t,x}\bigg[\int_{t}^{(t+\epsilon_0)\wedge\rho_{t,N}} \bigg\{-\tilde{b}(l,X^{\pi'}_l, m^{\infty}_l,\pi^{\infty}(l,X^{\pi'}_l)) \cdot D_x V^\infty(t, l, X^{\pi'}_l)-\tilde{r}\left(l-t, X_l^{\pi'},m^{\infty}_l,\pi^{\infty}(l,X_l^{\pi'})\right)  \\
    &\qquad \qquad \qquad  \qquad \qquad+ \tilde{b}(l,X^{\pi'}_l, m^{\infty}_l,\pi'(l,X^{\pi'}_l)) \cdot D_x V^\infty(t, l, X^{\pi'}_l)+\tilde{r}\left(l-t, X_l^{\pi'},m^{\infty}_l,\pi'(l,X_l^{\pi'})\right)\bigg\} dl\bigg]+\epsilon\\
    &\leq \mathbb{E}_{t,x}\bigg[\int_{t}^{(t+\epsilon_0)\wedge\rho_{t,N}}\left\{  H(t,l,X^{\pi'}_l) - L(t,l,X^{\pi'}_l)  \right\}  dl\bigg]+\epsilon\\
    &\leq \int_t^{t+\eps_0}  2C|l-t|^{\frac{\alpha}{2}} dl +\epsilon= \frac{2C}{1+\alpha/2}\eps_0^{1+\frac{\alpha}{2}}+\epsilon.
\end{align*}
By the arbitrariness of $\epsilon$, we have
$$
J^{\pi' \otimes_{t,\varepsilon_0} \pi^\infty}(t,x)-J^{\pi^{\infty},m^{\infty}}(t,x) \leq \frac{2C}{1+\alpha/2}\eps_0^{1+\frac{\alpha}{2}},
$$
and hence \eqref{eq:def.equipi} follows.

{\bf Verification of Condition (b).} It remains to show that the identity $m^\infty_s = \operatorname{law}(X^{\pi^\infty, m^\infty}_s)$ holds for any $s \in (0,T]$.
For $n\in\mathbb{N}\cup\{\infty\}$, consider the processes
\begin{equation*}
    dX^{n}_t = \tilde{b}^n(t,X^{n}_t)dt + \sigma^n(t,X^n_t)dW_t, \quad X^{n}_0 = \xi \sim \nu,
\end{equation*}
where $\tilde{b}^n(t,x) := \tilde{b}(t, x, m^n_t, \pi^n(t,x))$ and $\sigma^n(t,x) := \sigma(t, x, m^n_t)$. 

Fix an arbitrary test function $\phi \in C_c^2(\mathbb{R}^d)$ and any $s \in (0,T]$. Consider the backward PDE on $[0,s]\times\mathbb{R}^d$:
\begin{equation}\label{eq:dual_pde}
\begin{cases}
\partial_t u(t,x) + \frac{1}{2}\operatorname{tr}\left( (\sigma^\infty (\sigma^{\infty})^T)(t,x) D_x^2 u(t,x) \right) + \tilde{b}^\infty(t,x) \cdot D_x u(t,x) = 0, \\
u(s,x) = \phi(x).
\end{cases}
\end{equation}
Under Assumption \ref{assume.r}, \eqref{eq:dual_pde} admits a unique strong solution $u \in W^{1,2}_p([0,s] \times \mathbb{R}^d)$ (see, e.g., \cite[Theorem 5.2.10]{krylov2008lectures}).  By Sobolev embedding (see, e.g., \cite[Lemma 3.3]{ladyzhenskaia1968linear}), $D_x u$ is bounded and continuous.

Applying the It\^o-Krylov formula to $u(t, X^n_t)$ and taking expectations, we obtain that
\begin{align*}
    \mathbb{E}[\phi(X^n_s)] - \mathbb{E}[u(0, \xi)] &= \mathbb{E} \int_0^s \left( \partial_t u + \frac{1}{2}\operatorname{tr}(\sigma^n (\sigma^{n})^T D_x^2 u) + \tilde{b}^n \cdot D_x u \right)(t, X^n_t) dt.
\end{align*}
Substituting $\partial_t u = - \frac{1}{2}\operatorname{tr}(\sigma^\infty (\sigma^{\infty})^T D_x^2 u) - \tilde{b}^\infty \cdot D_x u$ from \eqref{eq:dual_pde} into the above equality, we have
\begin{align}
    \mathbb{E}[\phi(X^n_s)] - \mathbb{E}[u(0, \xi)] &= \mathbb{E} \int_0^s (\tilde{b}^n - \tilde{b}^\infty)(t, X^n_t) \cdot D_x u(t, X^n_t) dt \nonumber \\
    &\quad + \frac{1}{2} \mathbb{E} \int_0^s\operatorname{tr}\left( (\sigma^n (\sigma^{n})^T - \sigma^\infty (\sigma^{\infty})^T)(t, X^n_t) D_x^2 u(t, X^n_t) \right) dt. \label{eq:law_diff}
\end{align}
Let $p^n(t,y)$ be the density of $X^n_t$. Then $p^n$ is a generalized solution to the corresponding Fokker-Planck equation, which can be expressed in the following divergence form:
\begin{align}\label{eq:fp_divergence}
    &\partial_t p^n - \sum_{i=1}^d \frac{\partial}{\partial x_i} \left( \sum_{j=1}^d A^n_{ij}(t,x) \frac{\partial p^n}{\partial x_j} + a^n_i(t,x) p^n \right) = 0,\\
    &p^n(0,x) = \nu(x),\nonumber
\end{align}
where $A^n(t,x) := \frac{1}{2}(\sigma^n (\sigma^{n})^T)(t,x)$ and $a^n_i(t,x) := \sum_{j=1}^d \partial_{x_j} A^n_{ij}(t,x) - \tilde{b}^n_i(t,x)$. Under the assumption that $\sigma^n$ is uniformly Lipschitz in $x$ and $\tilde{b}^n$ is uniformly bounded, the coefficients $A^n$ and $a^n$ are uniformly bounded with respect to $n$.

First, given the initial distribution density $p_0 \in L_2(\mathbb{R}^d)$, the existence of a unique generalized solution $p^n \in V_2^{1, 1/2}([0,s] \times \mathbb{R}^d)$ for each $n$ is guaranteed by \cite[Theorem 3.5.2]{ladyzhenskaia1968linear}. Next, by applying the classical energy estimate to the divergence-form equation \eqref{eq:fp_divergence}, the uniform ellipticity of $A^n$ and the uniform boundedness of the coefficients $a^n$ (both independent of $n$) ensure that the sequence $\{p^n\}$ is uniformly bounded in $L^2([0,s] \times \mathbb{R}^d)$.
Furthermore, \cite[Theorem 3.8.1]{ladyzhenskaia1968linear} implies that $\{p^n\}$ is locally uniformly bounded in $L^\infty((0,s] \times \mathbb{R}^d)$. Combined with the interior H\"older estimates provided by \cite[Theorem 3.10.1]{ladyzhenskaia1968linear}, there exists a common exponent $\alpha > 0$ such that $\{p^n\}$ is locally uniformly H\"older continuous on $(0,s] \times \mathbb{R}^d$ with H\"older norms bounded independently of $n$.
Consequently, the sequence $\{p^n\}$ is locally uniformly equicontinuous and bounded on $(0, s] \times \mathbb{R}^d$. By the Arzel\`a-Ascoli theorem and a diagonal argument, we conclude that there exists a subsequence (still denoted by $p^n$ for simplicity) and a continuous function $p^*$ such that $p^{n} \to p^*$ locally uniformly on $(0,s] \times \mathbb{R}^d$.

Now we claim that the strong convergence $\|p^n - p^*\|_{L^1([0,s] \times \mathbb{R}^d)} \to 0$ holds. We perform a truncation argument:
\begin{equation}\label{eq:l1_split}
    \iint_{[0,s] \times \mathbb{R}^d} |p^n - p^*| dxdt = \int_0^s \int_{B_R(0)} |p^n - p^*| dxdt + \int_0^s \int_{B_R^c(0)} |p^n - p^*| dxdt.
\end{equation}
Note that
\begin{equation*}
   \int_0^s \int_{B_R^c(0)} p^n(t,x) dxdt = \int_0^s \mathbb{P}(|X^n_t| > R) dt \leq \int_0^s \frac{\mathbb{E}[|X^n_t|^2]}{R^2} dt \leq \frac{A_0 s}{R^2} \leq \frac{A_0 T}{R^2},
\end{equation*}
where we have used $\sup_{n,t\in[0,T]}\mathbb{E}[|X^n_t|^2]\leq A_0$, which follows from the uniform boundedness of the coefficients. Since $p^n \to p^*$ almost everywhere, Fatou's Lemma implies that the limit $p^*$ also satisfies
\begin{equation*}
     \int_0^s\int_{B_R^c(0)} p^*(t,x) dxdt \leq \liminf_{n\to\infty} \int_0^s \int_{B_R^c(0)} p^n(t,x) dxdt \leq \frac{A_0 T}{R^2}.
\end{equation*}
Therefore, using $|p^n - p^*| \leq p^n + p^*$, the second term in \eqref{eq:l1_split} is bounded by
\begin{equation}\label{eq:tail_error}
    \int_0^s \int_{B_R^c(0)} |p^n - p^*| dxdt \leq \frac{2A_0 T}{R^2}.
\end{equation}
By Krylov's inequality, $\{p^n\}$ is uniformly bounded in $L^q([0,s] \times \mathbb{R}^d)$ for $q = p/(p-1) > 1$. This uniform integrability, combined with the fact that $\Leb([0,s]\times B_R(0)) < \infty$, allows us to apply Vitali's Convergence Theorem (see, e.g. \cite[Theorem 4.5.4]{bogachev2007measure}) to obtain
\begin{equation}\label{eq:compact_conv}
    \lim_{n\to\infty} \int_0^s \int_{B_R(0)} |p^n - p^*| dxdt = 0.
\end{equation}
For any given $\epsilon > 0$, we first choose $R$ sufficiently large such that $\frac{2A_0 T}{R^2} < \frac{\epsilon}{2}$.
Next, for this fixed $R$, by \eqref{eq:compact_conv}, we can find an $N_\epsilon$ such that
\begin{equation*}
    \int_0^s \int_{B_R(0)} |p^n - p^*| dxdt < \frac{\epsilon}{2},\quad \forall n > N_\epsilon.
\end{equation*}
Combining these two estimates into \eqref{eq:l1_split}, we conclude that
\begin{equation*}
    \iint_{[0,s]\times\R^d} |p^n - p^*| dxdt < \frac{\epsilon}{2} + \frac{\epsilon}{2} = \epsilon,\quad \forall n > N_\epsilon.
\end{equation*}
Since $\epsilon$ is arbitrary, the claim $\lim_{n\to\infty} \|p^n - p^*\|_{L^1([0,s]\times\R^d)} = 0$ holds.

Note that
\begin{align*}
    \int_0^s\int_{\mathbb{R}^d} (\tilde{b}^n - \tilde{b}^\infty) \cdot D_x u \, p^n dy dt &= \int_0^s\int_{\mathbb{R}^d} (\tilde{b}^n - \tilde{b}^\infty) \cdot D_x u \, p^* dy dt \\
    &\quad + \int_0^s\int_{\mathbb{R}^d} (\tilde{b}^n - \tilde{b}^\infty) \cdot D_x u \, (p^n - p^*) dy dt.
\end{align*}
For the first term, since $D_x u$ is bounded and $p^* \in L^1$, the product $D_x u \cdot p^* \in L^1$. From \eqref{eq:weakstar}, the weak-$*$ convergence $(\tilde{b}^n - \tilde{b}^\infty) \rightharpoonup^* 0$ in $L^\infty$ implies this term converges to $0$.
For the second term, since $\tilde{b}^n - \tilde{b}^\infty$ and $D_x u$ are uniformly bounded, it is controlled by $C \|p^n - p^*\|_{L^1}$, which converges to $0$ by the strong convergence of $p^n$.

For the diffusion difference term, the Lipschitz continuity of $\sigma$ in $m$ and $d(m^n, m^\infty) \to 0$ imply $\|\sigma^n \sigma^{n,T} - \sigma^\infty \sigma^{\infty,T}\|_{\Cc^0([0,s]\times\R^d)} \to 0$. By H\"older's inequality, since $D_x^2 u \in L^p$ and $p^n$ is bounded in $L^q$,
\begin{align*}
    &\left| \frac{1}{2} \int_0^s \mathbb{E}\left[ \operatorname{tr}\big( (\sigma^n \sigma^{n,T} - \sigma^\infty \sigma^{\infty,T}) D_x^2 u \big) \right] dt \right| \\&\qquad\leq C \|\sigma^n \sigma^{n,T} - \sigma^\infty \sigma^{\infty,T}\|_{\Cc^0([0,s]\times\R^d)}\|D_x^2 u\|_{L^p([0,s]\times\R^d))} \sup_n \|p^n\|_{L^q([0,s]\times\R^d))} \to 0.
\end{align*}
Consequently, along the subsequence, we obtain
\begin{equation*}
    \lim_{n\to\infty} \mathbb{E}[\phi(X^n_s)] = \mathbb{E}[u(0, \xi)].
\end{equation*}
Note that applying the It\^o-Krylov formula to $u(t, X^\infty_t)$  directly yields $\mathbb{E}[\phi(X^\infty_s)] = \mathbb{E}[u(0, \xi)]$. 
Consequently, we obtain $\lim_{n\to\infty} \mathbb{E}[\phi(X^n_s)] = \mathbb{E}[\phi(X^\infty_s)]$ for any $\phi \in C_c^2(\mathbb{R}^d)$. Moreover, as $m^n$ converges to $m^\infty$ in $(\mathscr{M}, d)$, it follows that $\mathbb{E}[\phi(X^n_s)] = \int_{\mathbb{R}^d} \phi(x) m^n_s(dx) \to \int_{\mathbb{R}^d} \phi(x) m^\infty_s(dx)$. By the uniqueness of the limit, we have
\begin{equation*}
    \int_{\mathbb{R}^d} \phi(x) m^\infty_s(dx) = \int_{\mathbb{R}^d} \phi(x) \operatorname{law}(X^\infty_s)(dx), \quad \forall \phi \in C_c^2(\mathbb{R}^d),
\end{equation*}
which implies that $m^\infty_s = \operatorname{law}(X^\infty_s)$ for all $s \in [0,T]$.
\end{proof}




\section{Convergence Analysis of the PIA for Regularized Equilibria}\label{sect:pia}

This section introduces a Policy Iteration Algorithm (PIA) for the entropy-regularized formulation and rigorously establishes its convergence to a regularized equilibrium. We first present the PIA as follows:

\begin{itemize}
	\item \textbf{Initialization:} Start with an initial guess $(J^0, m^0) \in \mathcal{M}_\lambda \times \mathscr{M}_{\nu}^{\kappa,C^*}$.

	\item \textbf{Iteration Step $k \geq 0$:}
	\begin{enumerate}
		\item \textit{Policy Update:} Given the current iterates $(J^k, m^k)$, the updated policy $\pi^{k+1}$ is computed as:
		\begin{equation*}
			\pi^{k+1}(t,x,a) = \Gamma_\lambda(t, x, D_x J^{k}(t,x), m^{k}_t, a), \quad \forall (t, x, a) \in [0, T] \times \mathbb{R}^d \times U.
		\end{equation*}

        \item {{\it Policy Evaluation:} Compute the updated auxiliary value function $V^{k+1}(t,s,x):=V^{\pi^{k+1}, m^k}(t,s,x)$ with $J^{k+1}(t,x):= V^{k+1}(t,t,x)$, and the updated measure flow $m^{k+1}:= \mu^{\pi^{k+1}, m^k}$.}
	\end{enumerate}
\end{itemize}
Note that the PIA can be compactly described by the recursive iteration using the operator $\Phi_\lambda$ in \eqref{eq:Phi.fixedpoint}
\begin{equation*}
	(J^{k+1}, m^{k+1}) = \Phi_\lambda(J^k, m^k). 
\end{equation*}

\begin{Assumption}\label{assume.PIA}
    $\sigma$ is continuously differentiable with respect to the spatial variable $x$. Moreover, the terminal cost $F$ and its spatial gradient $D_x F$ are Lipschitz continuous in $m$ uniformly with respect to the other arguments, i.e., there exists a constant $K_6 > 0$ such that
    \begin{equation*}
        |F(t,x,m) - F(t,x,\rho)| + |D_x F(t,x,m) - D_x F(t,x,\rho)| \leq K_6 \mathcal{W}_2(m,\rho),
    \end{equation*}
    for all $(t,x) \in[0,T] \times \mathbb{R}^d$ and $m, \rho \in \mathscr{P}_2(\mathbb{R}^d)$.
\end{Assumption}

\begin{Theorem}\label{thm:PIA_convergence}
    Under the conditions of Theorem \ref{thm:existence_regu} (ii), suppose further that Assumption \ref{assume.PIA} holds. There exist constants $T_0 > 0$ and $\bar{K}_6 > 0$ such that, {if the time horizon $T \leq T_0$ and $K_6 \leq \bar{K}_6$}, then for any  $0 < \lambda \leq \lambda_0$, the sequence $(J^{k}, m^{k})$ generated by the PIA converges to a unique limit $(J^*_\lambda, m^*_\lambda)$.
    Moreover, the corresponding limit pair $(\pi^*_\lambda, m^*_\lambda)$ constitutes a regularized equilibrium with entropy weight $\lambda$, where $\pi^*_\lambda$ is explicitly given by the Gibbs measure:
    \begin{equation*}
        \pi^*_\lambda(t,x,a) = \Gamma_{\lambda}(t,x,D_x J^*(t,x),m^*_{\lambda,t},a), \quad \forall (t, x, a) \in [0, T] \times \mathbb{R}^d \times U.
    \end{equation*}
\end{Theorem}
\begin{Remark}\label{rmk:PIA_assumptions}
    We conclude this section with a discussion of the short time horizon assumption ($T \leq T_0$) and the weak terminal dependence assumption ($K_6 \leq \bar{K}_6$) for the convergence of the PIA.
    \begin{itemize}
        \item First, regarding the terminal condition, to the best of our knowledge, the existing literature on the convergence of PIA for time-consistent MFGs (e.g., \cite{TangSong2024,camilli2022rates,cacace2021policy,lauriere2023policy}) does not consider dependence of the terminal cost $F$ on $m$, which corresponds to the case $K_6 = 0$ in our framework. Our result accommodates this terminal measure dependence under a weak interaction regime.

        \item Second, regarding the short-time-horizon assumption, we note that several recent studies (e.g., \cite{cacace2021policy, TangSong2024}) have successfully bypassed the small-time restriction in the context of time-consistent MFGs, mainly for systems with \textit{separable} Hamiltonians and constant diffusion terms. Their methodology typically relies on a compactness argument to extract a convergent subsequence, while convergence of the entire PIA sequence is subsequently guaranteed by the \textit{uniqueness} of the mean field equilibrium (often ensured by the Lasry-Lions monotonicity condition or a potential game structure). However, uniqueness of equilibria for time-inconsistent control problems, let alone time-inconsistent MFGs, remains a challenging open problem. Consequently, extending such compactness arguments to our time-inconsistent setting is fundamentally hindered by the absence of a suitable uniqueness theory.

         \item Finally, for time-consistent MFGs with general \textit{non-separable} Hamiltonians, existing results (such as \cite{lauriere2023policy}) also rely on contraction mapping arguments, which inherently require a short-time assumption to guarantee convergence. From this perspective, the contraction conditions established in our work are consistent with the existing literature, while extending the PIA framework to the time-inconsistent setting. Furthermore, our framework imposes weaker restrictions on the underlying SDE coefficients; for instance, while the aforementioned works often restrict the diffusion term $\sigma$ to be a constant matrix, our analysis allows $\sigma(t,x,m)$ to be fully state and measure-dependent.
    \end{itemize}
\end{Remark}

\subsection{Proof of Theorem \ref{thm:PIA_convergence}}
From Lemma \ref{lm:iteration.estimateentropy} and Lemma \ref{lm:Phi.continuous}, we conclude that $(J^k,m^k)\in \mathcal{M}_\lambda \times \mathscr{M}_{\nu}^{\kappa,C^*}$ and $V^{k+1}$ is the unique classical solution to the following linear parabolic PDE:
\begin{align}\label{eq:iteraPDEk}
		\partial_s V^{k+1} (t,s, x) &+ \frac{1}{2}\operatorname{tr}\left((\sigma\sigma^T)(s,x, m^k_s) D^2_x V^{k+1} (t,s, x)\right)\notag \\
		&+ \tilde{b}(s,x, m^k_s, \pi^{k+1}(s,x)) \cdot \left[ D_x V^{k+1}(t,s, x) - \delta(s-t)D_x V^{k}(s,s,x) \right]\notag \\
		&+ \tilde{r}(s-t,x, m^k_s,\pi^{k+1}(s,x)) - \delta(s-t)\tilde{r}(0,x,m^k_s,\pi^{k+1}(s,x))\notag \\
		&+ \delta(s-t)H(s,x, D_x V^{k}(s,s,x), m^k_s) = 0, \\
		V^{k+1} (t,T, x) &= F(t,x,m^k_T),\notag
\end{align}
satisfying the uniform bounds
\begin{equation}\label{eq:iteraPDEbound}
  \sup_{k \in \mathbb{N}} \|V^{k}\|_{\widetilde{\Cc}^{0, 1}_{\alpha,[0, T]}} < A^*,\quad    \sup_{k \in \mathbb{N}} \|V^{k}\|_{\widetilde{\Cc}^{1, 2}_{\alpha,[0, T]}} < A_{\lambda},
\end{equation}
where the constants $A^*, A_\lambda$ are stated in Theorem \ref{thm:existence_regu} Part (ii), and $H$ is defined by
\begin{equation*}
    H(s,x,p,m) := \lambda\ln\left(\int_U \exp\left( \frac{1}{\lambda}\left[b(s,x,m,a')\cdot p+r(0,x,m,a') \right]\right) da'\right).
\end{equation*}
\begin{Lemma}\label{lm:derivatives_bounds}
   Suppose that Assumption \ref{assume.r} holds. Then, the functions $\Gamma_{\lambda}$ and $H$ are differentiable with respect to $p$, and their derivatives satisfy the uniform bounds:
    \begin{align}
        |D_p \Gamma_{\lambda}(t,x,p,m,a)| &\leq \frac{2K_1}{\lambda}\Gamma_{\lambda}(t,x,p,m,a), \label{est:Gamma}\\
        |D_p H(t,x,p,m)| &\leq K_1, \label{est:H}
    \end{align}
    where $K_1$ is the uniform bound in Assumption \ref{assume.r}.  Moreover, $\Gamma_\lambda$ is Lipschitz continuous with respect to $(x,m)$ in the following sense: for any bounded measurable function $\phi: U \rightarrow \mathbb{R}^d$, it holds that
    \begin{align}\label{est:Gamma_m}
        &\left| \int_{U} \phi(a) \big( \Gamma_{\lambda}(t,x,p,m,a) - \Gamma_{\lambda}(t,y,p,\rho,a) \big) da \right|\nonumber \\&\qquad \leq \|\phi\|_{L^{\infty}(U)} \frac{2K_2}{\lambda} (|p|+1) \exp\left(\frac{2K_1}{\lambda}(|p|+1)\right) (|x-y|+\mathcal{W}_2(m,\rho)),
    \end{align}
    where $K_2$ is the Lipschitz constant  in Assumption \ref{assume.r}. Furthermore, $H$ is Lipschitz continuous with respect to the measure $m$ such that
    \begin{equation}\label{est:H_m}
        |H(t,x,p,m)-H(t,x,p,\rho)|\leq K_2(|p|+1)\mathcal{W}_2(m,\rho).
    \end{equation}
\end{Lemma}
\begin{proof}
Estimates \eqref{est:Gamma} and \eqref{est:H} follow from straightforward differentiation. We only present the proof of \eqref{est:Gamma_m} and \eqref{est:H_m}.

Let $g(t,x,m,p,a) := \frac{1}{\lambda}[b(t,x,m,a) \cdot p + r(0,x,m,a)]$. By Assumption \ref{assume.r}, for any $m, \rho \in \mathscr{P}_2(\mathbb{R}^d)$ and $(x,y)\in \R^d\times\R^d$, we have
\begin{align*}
    |g(t,x,m,p,a) - g(t,y,\rho,p,a)| &\leq \frac{K_2}{\lambda}(|p|+1) \left[|x-y|+\mathcal{W}_2(m, \rho)\right], \\
    |g(t,x,m,p,a)| &\leq \frac{K_1}{\lambda}(|p|+1).
\end{align*}
Direct computation yields
\begin{align*}
    &|\Gamma_{\lambda}(t,x,p,m, a) - \Gamma_{\lambda}(t,y,p,\rho, a)| \\
    &= \left| \frac{\exp[g(t,x,m,p,a)] \int_U \exp[g(t,y,\rho,p,a')]da' - \exp[g(t,y,\rho,p,a)] \int_U \exp[g(t,x,m,p,a')]da'}{\left(\int_U \exp[g(t,x,m,p,a')]da'\right) \left(\int_U \exp[g(t,y,\rho,p,a')]da'\right)} \right| \\
    &= \left| \frac{\exp[g(t,x,m,p,a)] \left( \int_U \exp[g(t,y,\rho,p,a')]da' - \int_U \exp[g(t,x,m,p,a')]da' \right)}{\left(\int_U \exp[g(t,x,m,p,a')]da'\right) \left(\int_U \exp[g(t,y,\rho,p,a')]da'\right)} \right. \\
    &\quad + \left. \frac{\left( \exp[g(t,x,m,p,a)] - \exp[g(t,y,\rho,p,a)] \right) \int_U \exp[g(t,x,m,p,a')]da'}{\left(\int_U \exp[g(t,x,m,p,a')]da'\right) \left(\int_U \exp[g(t,y,\rho,p,a')]da'\right)} \right| \\
    &\leq \Gamma_{\lambda}(t,x,p,m, a) \left| \int_U \frac{\exp[g(t,y,\rho,p,a')] - \exp[g(t,x,m,p,a')]}{\int_U \exp[g(t,y,\rho,p,a'')]da''} da' \right| \\
    &\quad + \frac{\left| \exp[g(t,x,m,p,a)] - \exp[g(t,y,\rho,p,a)] \right|}{\int_U \exp[g(t,y,\rho,p,a')]da'} \\
    &= \Gamma_{\lambda}(t,x,p,m, a) \left| \int_U \left( 1 - \exp[g(t,x,m,p,a') - g(t,y,\rho,p,a')] \right) \Gamma_{\lambda}(t,y,p,\rho, a') da' \right| \\
    &\quad + \left| \exp[g(t,x,m,p,a) - g(t,y,\rho,p,a)] - 1 \right| \Gamma_{\lambda}(t,y,p,\rho, a).
\end{align*}
Using the inequality $|e^z - 1| \leq |z| e^{ |z|}$ and noting $\int_U \Gamma_{\lambda}(t,y,p,\rho, a') da' = 1$, we obtain
\begin{align*}
    &|\Gamma_{\lambda}(t,x,p,m, a) - \Gamma_{\lambda}(t,y,p,\rho, a)| \\
    &\leq \left( \frac{K_2}{\lambda}(|p|+1) \left[|x-y|+\mathcal{W}_2(m, \rho)\right] \exp\left[\frac{2K_1}{\lambda}(|p|+1)\right] \right) \big( \Gamma_{\lambda}(t,x,p,m, a) + \Gamma_{\lambda}(t,y,p,\rho, a) \big).
\end{align*}
Finally, integrating the above against $\phi(a)$ over $U$ yields \eqref{est:Gamma_m}:
\begin{align*}
    &\left| \int_U \phi(a) \big( \Gamma_{\lambda}(t,x,p,m, a) - \Gamma_{\lambda}(t,x,p,\rho, a) \big) da \right| \\
    &\leq \|\phi\|_{L^\infty(U)} \left( \frac{K_2}{\lambda}(|p|+1) e^{\frac{2K_1}{\lambda}(|p|+1)} \left[|x-y|+\mathcal{W}_2(m, \rho)\right] \right) \int_U \big( \Gamma_{\lambda}(t,x,p,m, a) + \Gamma_{\lambda}(t,x,p,\rho, a) \big) da \\
    &= 2 \|\phi\|_{L^\infty(U)} \frac{K_2}{\lambda}(|p|+1) \exp\left[\frac{2K_1}{\lambda}(|p|+1)\right] \left[|x-y|+\mathcal{W}_2(m, \rho)\right].
\end{align*}
We now prove \eqref{est:H_m}. By the definition of $H(t,x,p,m)$, we have:
\begin{align*}
    &|H(t,x,p,m) - H(t,x,p,\rho)|
    = \lambda \left| \ln\left( \frac{\int_U \exp[g(t,x,m,p,a')] da'}{\int_U \exp[g(t,x,\rho,p,a')] da'} \right) \right| \\
    =& \lambda \left| \ln\left( \int_U \exp[g(t,x,m,p,a') - g(t,x,\rho,p,a')] \Gamma_{\lambda}(t,x,p,\rho, a') da' \right) \right|.
\end{align*}
Then,
\begin{align*}
    |H(t,x,p,m) - H(t,x,p,\rho)|
    &\leq K_2(|p|+1) \mathcal{W}_2(m, \rho),
\end{align*}
which establishes \eqref{est:H_m} and completes the proof.
\end{proof}
Let $(\pi^*_{\lambda},m^*_{\lambda})$ be a regularized equilibrium with entropy weight $\lambda$, whose existence is ensured by Theorem \ref{thm:existence_regu}. Denote $V^*(t,s,x)=V^{\pi^*_{\lambda},m^*_{\lambda}}_{\lambda}(t,s,x)$. Then $(V^*,\pi^*_{\lambda},m^*_{\lambda})$ satisfies \eqref{eq:HJBn1}-\eqref{eq:consisn}, with $V^n$, $\pi^n$, $m^n$, and $\lambda_n$ replaced by $V^*$, $\pi^*_{\lambda}$, $m^*_{\lambda}$, and $\lambda$, respectively.
\begin{Lemma}\label{lm:measure_contraction}
    Suppose Assumption \ref{assume.r} holds. Then, there exists a non-decreasing function $C_1(T)$ satisfying $\lim_{T\to 0} C_1(T) = 0$ such that
    \begin{equation*}
        d(m^{k+1}, m^*_{\lambda}) \leq C_1(T) \left[ \sup_{(s,x)\in[0,T]\times\mathbb{R}^d} |D_x V^k(s,s,x) - D_x V^{*}(s,s,x)| + d(m^k, m^*_{\lambda}) \right].
    \end{equation*}
\end{Lemma}
\begin{proof}
    Note that
    \begin{equation*}
        d^2(m^{k+1}, m^*_{\lambda}) = \sup_{t \in [0,T]} \mathcal{W}_2^2(m^{k+1}_t,m^*_{\lambda,t}) \leq \mathbb{E} \left[ \sup_{t \in [0,T]} |\tilde{X}^{k+1}_t - \tilde{X}^*_t|^2 \right],
    \end{equation*}
   where
    \begin{align*}
        d\tilde{X}^{k+1}_t &= \tilde{b}(t, \tilde{X}^{k+1}_t, m^k_t, \pi^{k+1}(t, \tilde{X}^{k+1}_t)) dt + \sigma(t, \tilde{X}^{k+1}_t,m^k_t) dW_t, \\
        d\tilde{X}^{*}_t &= \tilde{b}(t, \tilde{X}^{*}_t, m^{*}_{\lambda,t}, \pi^{*}_{\lambda}(t, \tilde{X}^{*}_t)) dt + \sigma(t, \tilde{X}^*_t,m^*_{\lambda,t}) dW_t,\\
        \tilde{X}^{k+1}_0 &= \tilde{X}^{*}_0 = \xi \sim \nu.
    \end{align*}
    Using the BDG inequality, there exists a constant $A_0 > 0$ such that
    \begin{align}\label{eq:dx_gronwall_base}
        \mathbb{E} \left[ \sup_{s \in[0,t]} |\tilde{X}^{k+1}_s-\tilde{X}^{*}_s|^2 \right] &\leq A_0 \mathbb{E} \bigg[ \int_0^t \left| \tilde{b}(s, \tilde{X}^{k+1}_s, m^k_s, \pi^{k+1}(s, \tilde{X}^{k+1}_s))- \tilde{b}(s, \tilde{X}^{*}_s, m^{*}_{\lambda,s}, \pi^{*}_{\lambda}(s, \tilde{X}^{*}_s)) \right|^2 ds \notag \\
        &\qquad\qquad + \int_0^t \left| \sigma(s, \tilde{X}^{k+1}_s,m^k_s) - \sigma(s, \tilde{X}^*_s,m^*_{\lambda,s}) \right|^2 ds \bigg].
    \end{align}
    For the drift term, we decompose the difference into three parts:
    \begin{align}
        &\left| \tilde{b}(s, \tilde{X}^{k+1}_s, m^k_s, \pi^{k+1}(s, \tilde{X}^{k+1}_s))- \tilde{b}(s, \tilde{X}^{*}_s, m^{*}_{\lambda,s}, \pi^{*}_{\lambda}(s, \tilde{X}^{*}_s)) \right|\nonumber \\
        &\leq |\tilde{b}(s, \tilde{X}^{k+1}_s, m^k_s, \pi^{k+1}(s, \tilde{X}^{k+1}_s)) - \tilde{b}(s, \tilde{X}^*_s, m^k_s, \pi^{k+1}(s, \tilde{X}^{k+1}_s))|\nonumber \\
        &\quad+ |\tilde{b}(s, \tilde{X}^*_s, m^k_s,\pi^{k+1}(s, \tilde{X}^{k+1}_s)) - \tilde{b}(s, \tilde{X}^*_s, m^{*}_{\lambda,s}, \pi^{k+1}(s, \tilde{X}^{k+1}_s))|\nonumber \\
        &\quad + |\tilde{b}(s, \tilde{X}^*_s, m^{*}_{\lambda,s}, \pi^{k+1}(s, \tilde{X}^{k+1}_s)) - \tilde{b}(s, \tilde{X}^*_s, m^{*}_{\lambda,s}, \pi^{*}_{\lambda}(s, \tilde{X}^{*}_s))|.\label{est:tb}
    \end{align}
    By the  Lipschitz property of $b$ in Assumption \ref{assume.r}, the first two terms are bounded by $K_2 |\tilde{X}^{k+1}_s - \tilde{X}^{*}_s|$ and $K_2 \mathcal{W}_2(m^k_s, m^{*}_{\lambda,s})$, respectively. For the third term, using the regularity of $\Gamma_{\lambda}$ established in Lemma \ref{lm:derivatives_bounds}, we have
    \begin{align*}
        &\left| \tilde{b}(s, \tilde{X}^*_s, m^{*}_{\lambda,s}, \pi^{k+1}(s, \tilde{X}^{k+1}_s)) - \tilde{b}(s, \tilde{X}^*_s, m^{*}_{\lambda,s}, \pi^*(s, \tilde{X}^{*}_s)) \right| \\
        &\leq K_1 \int_U \left| \Gamma_\lambda(s, \tilde{X}^{k+1}_s, D_x V^k(s,s,\tilde{X}^{k+1}_s), m^k_s, a) - \Gamma_\lambda(s, \tilde{X}^*_s, D_x V^{*}(s,s,\tilde{X}^*_s), m^{*}_{\lambda,s}, a) \right| da \\
        &\leq C_\lambda \left( |\tilde{X}^{k+1}_s - \tilde{X}^*_s| + |D_x V^k(s,s,\tilde{X}^{k+1}_s) - D_x V^{*}(s,s,\tilde{X}^*_s)| + \mathcal{W}_2(m^k_s, m^{*}_{\lambda,s}) \right),
    \end{align*}
    where $C_\lambda$ is a finite constant depending on $\lambda$, $K_1$, $K_2$, and the first uniform bound $A^*$ in \eqref{eq:iteraPDEbound}. Noting the second bound in \eqref{eq:iteraPDEbound}, we further expand the gradient difference:
    \begin{align*}
        |D_x V^k(s,s,\tilde{X}^{k+1}_s) - D_x V^{*}(s,s,\tilde{X}^*_s)| &\leq |D_x V^k(s,s,\tilde{X}^{k+1}_s) - D_x V^k(s,s,\tilde{X}^*_s)| \\
        &\quad + |D_x V^k(s,s,\tilde{X}^*_s) - D_x V^{*}(s,s,\tilde{X}^*_s)| \\
        &\leq A_\lambda |\tilde{X}^{k+1}_s - \tilde{X}^*_s| + \sup_{x \in \mathbb{R}^d} |D_x V^k(s,s,x) - D_x V^{*}(s,s,x)|.
    \end{align*}
    For the diffusion term, Assumption \ref{assume.r} (iii) and (iv) together yield that $|\sigma(s, \tilde{X}^{k+1}_s,m^{k}_s) - \sigma(s, \tilde{X}^*_s,m^{*}_{\lambda,s})| \leq K_2 (|\tilde{X}^{k+1}_s - \tilde{X}^*_s|+\mathcal{W}_2(m^k_s, m^{*}_{\lambda,s}))$.

    Substituting all these bounds into \eqref{eq:dx_gronwall_base}, we obtain
    \begin{align*}
        \mathbb{E} \left[ \sup_{s \in[0,t]} |\tilde{X}^{k+1}_s - \tilde{X}^{*}_s|^2 \right] &\leq \tilde{C}_\lambda \int_0^t \mathbb{E} \left[ \sup_{r \in [0,s]} |\tilde{X}^{k+1}_r - \tilde{X}^{*}_r|^2 \right] ds \\
        &\quad + \tilde{C}_\lambda \int_0^t \left( \mathcal{W}_2^2(m^k_s, m^{*}_{\lambda,s}) + \sup_{x \in \mathbb{R}^d} |D_x V^k(s,s,x) - D_x V^{*}(s,s,x)|^2 \right) ds,
    \end{align*}
    where $\tilde{C}_\lambda > 0$ is a generic constant depending on $A_0$, $C_\lambda$, and the bound $A_\lambda$. Applying Gronwall's inequality yields
    \begin{align*}
        &\mathbb{E} \left[ \sup_{t \in [0,T]} |\tilde{X}^{k+1}_t - \tilde{X}^*_t|^2 \right] \\
        &\leq \tilde{C}_\lambda T e^{\tilde{C}_\lambda T} \left( \sup_{t \in[0,T]} \mathcal{W}_2^2(m^k_t, m^{*}_{\lambda,t}) + \sup_{(t,x) \in [0,T] \times \mathbb{R}^d} |D_x V^k(t,t,x) - D_x V^{*}(t,t,x)|^2 \right).
    \end{align*}
    Taking the square root on both sides and setting $C_1(T) := \sqrt{\tilde{C}_\lambda T e^{\tilde{C}_\lambda T}}$, it is evident that $\lim_{T \to 0} C_1(T) = 0$, which completes the proof.
\end{proof}

\begin{Lemma}\label{lm:V_contraction}
   Under the assumptions of Theorem \ref{thm:PIA_convergence}, there exist non-decreasing functions $C_2(T)$ and $C_3(T)$ satisfying $\lim_{T\rightarrow 0}C_3(T)=0$ such that
    \begin{align*}
        \|V^{k+1}-V^{*}\|_{\mathcal{C}^{0,1}_{[0,T]}} &\leq K_6 C_2(T) \mathcal{W}_2(m^{k}_T, m^{*}_{\lambda,T}) \\
        &\quad + C_{3}(T) \left( d(m^k, m^{*}_{\lambda}) + \|V^{k+1}-V^{*}\|_{\mathcal{C}^{0,1}_{[0,T]}} + \|V^{k}-V^{*}\|_{\mathcal{C}^{0,1}_{[0,T]}} \right).
    \end{align*}
\end{Lemma}
\begin{proof}
    To establish these estimates, we will use probabilistic representation formulae, as proposed in \cite{ma2025convergence}. To this end, fix $(s, x) \in[0, T) \times \mathbb{R}^d$ and consider the auxiliary state process
\begin{equation*}
	X_l^{s, x} = x + \int_s^l \sigma\left(r, X_r^{s, x},m^*_{\lambda,r}\right) d W_r, \quad l \in [s, T].
\end{equation*}
As $V^{k+1}$ is the classical solution to \eqref{eq:iteraPDEk}, Feynman--Kac's formula implies
\begin{align}
	V^{k+1}(t, s, x) &= \mathbb{E} \left[ F(t, X_T^{s, x},m^k_T) + \int_s^T f^{k+1}(t, r,  X_r^{s, x}) dr \right], \label{eq:rep1}
\end{align}
where
\begin{align}
	f^{k+1}(t, s, x) &:= \tilde{b}(s,x, m^k_s, \pi^{k+1}(s,x)) \cdot \left[ D_x V^{k+1}(t,s, x) - \delta(s-t)D_x V^{k}(s,s,x) \right]\nonumber \\
    &\quad+\frac{1}{2} \operatorname{tr}\left(((\sigma\sigma^T)(s,x, m^k_s)-(\sigma\sigma^T)(s,x, m^*_{\lambda,s})) D^2_x V^{k+1} (t,s, x)\right)     \nonumber        \\
		&\quad + \tilde{r}(s-t,x, m^k_s,\pi^{k+1}(s,x)) - \delta(s-t)\tilde{r}(0,x,m^k_s,\pi^{k+1}(s,x))\nonumber \\
		&\quad + \delta(s-t)H(s,x, D_x V^{k}(s,s,x), m^k_s).\nonumber
\end{align}
By applying the Bismut--Elworthy--Li formula (see, e.g., \cite{Bismut1984LargeDeviation, elworthy1994formulae, ma2002representation}) and following the arguments in \cite{ma2025convergence}, we obtain a representation formula for $D_xV^{k+1}$:
\begin{align}
	D_xV^{k+1}(t,s,x) &= \mathbb{E}\left[ (\nabla X_T^{s, x})^\top D_xF(t,  X_T^{s, x},m^k_T) + \int_s^T f^{k+1}(t, r, X_r^{s, x}) N_r^{s, x} d r \right]. \label{rep_first1}
\end{align}
Here, $\nabla X^{s,x}\in\mathbb{R}^{d\times d}$ with the $i$-th column $\partial_{x_i} X$. Moreover, using Einstein summation for repeated indices, we have
\begin{equation*}
	\begin{aligned}
		& \nabla X_l^{s, x} = I_d + \int_s^l D_x\sigma^i\left(r, X_r^{s, x},m^{*}_{\lambda,r}\right) \nabla X_r^{s, x} d W_r^i,
	\end{aligned}
\end{equation*}
where $\sigma^i$ is the $i$-th column of $\sigma$. Similarly, by setting $\check{\sigma}:=\sigma^{-1}$ to be the inverse matrix, we define
\begin{equation*}
	\begin{aligned}
		& N_l^{s, x}:=\frac{1}{l-s} \int_s^l\left(\check{\sigma}\left(r, X_r^{s, x},m^{*}_{\lambda,r}\right) \nabla X_r^{s, x}\right)^{\top} d W_r.
	\end{aligned}
\end{equation*}
Furthermore, it is straightforward to check that
\begin{equation}\label{estimate_1}
	\mathbb{E}\left[\left|\nabla X_l^{s, x}\right|^2\right] \leq C e^{C(l-s)}, \quad \mathbb{E}\left[\left|N_l^{s, x}\right|^2\right] \leq \frac{C e^{C(l-s)}}{l-s}.
\end{equation}
Similarly, we obtain representation formulae for $V^*$ and $D_x V^*$:
\begin{align}
   V^{*}(t, s, x) &= \mathbb{E} \left[ F(t, X_T^{s, x},m^*_{\lambda,T}) + \int_s^T f^{*}(t, r,  X_r^{s, x}) dr \right]\label{rep:v*}, \\
   D_xV^{*}(t,s,x) &= \mathbb{E}\left[ (\nabla X_T^{s, x})^\top D_xF(t,  X_T^{s, x},m^*_{\lambda,T}) + \int_s^T f^{*}(t, r, X_r^{s, x}) N_r^{s, x} d r \right],\label{rep:dv*}
\end{align}
where
\begin{align*}
    f^*(t,s,x)&=\tilde{b}(s,x, m^*_{\lambda,s}, \pi^{*}_{\lambda}(s,x)) \cdot \left[ D_x V^{*}(t,s, x) - \delta(s-t)D_x V^{*}(s,s,x) \right]\nonumber \\
		&\quad + \tilde{r}(s-t,x, m^*_{\lambda,s},\pi^{*}_{\lambda}(s,x)) - \delta(s-t)\tilde{r}(0,x,m^*_{\lambda,s},\pi^{*}_{\lambda}(s,x))\nonumber \\
		&\quad + \delta(s-t)H(s,x, D_x V^{*}(s,s,x), m^*_{\lambda,s}).
\end{align*}
For simplicity, denote $ \Delta \phi^{k+1}:=\phi^{k+1}-\phi^*$ for $\phi=f,V$. Then
\begin{align*}
    &|\Delta f^{k+1}(t,s,x)|\\
    &\leq \left| \bigg(\tilde{b}(s,x, m^k_s, \pi^{k+1}(s,x))-\tilde{b}(s,x, m^{*}_{\lambda,s}, \pi^{*}_{\lambda}(s,x))\bigg)\cdot \left[ D_x V^{k+1}(t,s, x) - \delta(s-t)D_x V^{k}(s,s,x) \right] \right| \\&\quad + \left| \tilde{b}(s,x, m^{*}_{\lambda,s}, \pi^{*}_{\lambda}(s,x)) \cdot \left[ D_x \Delta V^{k+1}(t,s,x) - \delta(s-t)D_x\Delta V^{k}(s,s,x) \right] \right| \\
    &\quad + \left|\frac{1}{2} \operatorname{tr}\left(((\sigma\sigma^T)(s,x, m^k_s)-(\sigma\sigma^T)(s,x, m^*_{\lambda,s})) D^2_x V^{k+1} (t,s, x)\right) \right| \\
    &\quad + \left| \tilde{r}(s-t,x, m^k_s,\pi^{k+1}(s,x)) - \tilde{r}(s-t,x, m^{*}_{\lambda,s},\pi^{*}_{\lambda}(s,x)) \right| \\
    &\quad + \delta(s-t)\left| \tilde{r}(0,x,m^k_s,\pi^{k+1}(s,x)) - \tilde{r}(0,x,m^{*}_{\lambda,s},\pi^{*}_{\lambda}(s,x)) \right| \\
    &\quad + \delta(s-t)\left| H(s,x, D_x V^{k}(s,s,x), m^k_s) - H(s,x, D_x V^{*}(s,s,x), m^{*}_{\lambda,s}) \right|.
\end{align*}
In the same fashion as \eqref{est:tb}, and using Lemma \ref{lm:derivatives_bounds}, we have
\begin{align*}
     &\bigg|\tilde{b}(s,x, m^k_s, \pi^{k+1}(s,x))-\tilde{b}(s,x, m^{*}_{\lambda,s}, \pi^{*}_{\lambda}(s,x))\bigg| + \left|\tilde{r}(s-t,x, m^k_s,\pi^{k+1}(s,x))-\tilde{r}(s-t,x, m^{*}_s,\pi^{*}_{\lambda}(s,x))\right| \\&\quad + \left|(\sigma\sigma^T)(s,x, m^k_s)-(\sigma\sigma^T)(s,x, m^*_{\lambda,s})\right| + \left|\tilde{r}(0,x,m^k_s,\pi^{k+1}(s,x))-\tilde{r}(0,x,m^{*}_{\lambda,s},\pi^{*}_{\lambda}(s,x))\right| \\
     &\quad + \left|H(s,x, D_x V^{k}(s,s,x), m^k_s)-H(s,x, D_x V^{*}(s,s,x), m^{*}_{\lambda,s})\right| \\
     &\leq C_{\lambda} \left( \mathcal{W}_2 (m^k_s,m^{*}_{\lambda,s}) + |D_{x} V^k(s,s,x)-D_x V^{*}(s,s,x)| \right),
\end{align*}
where $C_\lambda$ is a finite constant depending on $\lambda$, $K_1$, $K_2$, and the uniform bound $A_{\lambda}$. Therefore,
\begin{equation*}
     |\Delta f^{k+1}(t,s,x)| \leq C_{\lambda} \left( \mathcal{W}_2 (m^k_s,m^{*}_{\lambda,s}) + |D_{x} \Delta V^{k}(s,s,x)| + |D_x \Delta V^{k+1}(t,s,x)| \right).
\end{equation*}
It follows from \eqref{eq:rep1}, \eqref{rep_first1}, \eqref{rep:v*} and \eqref{rep:dv*} as well as the estimate \eqref{estimate_1} that
\begin{align*}
     &|\Delta V^{k+1}(t,s,x)| \\
     &\leq \mathbb{E} \left[ |F(t, X_T^{s, x},m^k_T)-F(t, X_T^{s, x},m^{*}_{\lambda,T})| + \int_s^T |\Delta f^{k+1}(t, r,  X_r^{s, x})| dr \right]\\
     &\leq K_6 \mathcal{W}_2(m^k_T,m^{*}_{\lambda,T}) \\
     &\quad + C_{\lambda}T\bigg( d (m^k,m^{*}_{\lambda}) + \sup_{(s,x)\in[0,T]\times\R^d}|D_{x} \Delta V^{k}(s,s,x)| + \sup_{(t,s,x)\in\Delta_{[0,T]}\times\R^d}|D_x \Delta V^{k+1}(t,s,x)| \bigg),\\
    &|D_x\Delta V^{k+1}(t,s,x)|\\
    &\leq \mathbb{E} \left[ |\nabla X_T^{s, x}||D_xF(t, X_T^{s, x},m^k_T)-D_xF(t, X_T^{s, x},m^{*}_{\lambda,T})| + \int_s^T |\Delta f^{k+1}(t, r,  X_r^{s, x})| |N_r^{s, x}| dr \right]\\
    &\leq C_{\lambda} e^{C T} \Bigg( K_6 \mathcal{W}_2(m^k_T,m^{*}_{\lambda,T}) \\
    &\quad + \sqrt{T}\bigg( d (m^k,m^{*}_{\lambda}) + \sup_{(s,x)\in[0,T]\times\R^d}|D_{x} \Delta V^{k}(s,s,x)| + \sup_{(t,s,x)\in\Delta_{[0,T]}\times\R^d}|D_x \Delta V^{k+1}(t,s,x)| \bigg) \Bigg).
\end{align*}
Taking the supremum over all $(t,s,x) \in \Delta_{[0,T]} \times \mathbb{R}^d$ on both sides of the estimates for $|\Delta V^{k+1}|$ and $|D_x \Delta V^{k+1}|$, and summing them, we obtain the following upper bound in the $\mathcal{C}^{0,1}_{[0,T]}$ norm:
\begin{align*}
    &\|V^{k+1} - V^*\|_{\mathcal{C}^{0,1}_{[0,T]}} \\&= \sup_{(t,s,x)\in \Delta_{[0,T]} \times \mathbb{R}^d} |V^{k+1}(t,s,x) - V^*(t,s,x)| + \sup_{(t,s,x)\in \Delta_{[0,T]} \times \mathbb{R}^d} |D_x V^{k+1}(t,s,x) - D_x V^*(t,s,x)| \\
    &\leq (1 + C_{\lambda} e^{C T}) K_6 \mathcal{W}_2(m^k_T, m^{*}_{\lambda,T}) \\
    &\quad + C_{\lambda} (T + e^{C T} \sqrt{T}) \bigg( d(m^k, m^{*}_{\lambda}) + \sup_{(s,x)\in[0,T]\times\R^d} |D_x \Delta V^{k}(s,s,x)| + \sup_{(t,s,x)\in\Delta_{[0,T]}\times\R^d} |D_x \Delta V^{k+1}(t,s,x)| \bigg)\\
&\leq (1 + C_{\lambda} e^{C T}) K_6 \mathcal{W}_2(m^k_T, m^{*}_{\lambda,T}) \\
    &\quad + C_{\lambda} (T + e^{C T} \sqrt{T}) \Big( d(m^k, m^{*}_{\lambda}) + \|V^{k+1} - V^*\|_{\mathcal{C}^{0,1}_{[0,T]}} + \|V^k - V^{*}\|_{\mathcal{C}^{0,1}_{[0,T]}} \Big).
\end{align*}
Define $C_2(T)$ and $C_3(T)$ by
\begin{equation*}
    C_2(T) := 1 + C_{\lambda} e^{C T}, \quad \text{and} \quad C_3(T) := C_{\lambda} (T + e^{C T} \sqrt{T}).
\end{equation*}
Clearly, both $C_2(T)$ and $C_3(T)$ are non-decreasing functions of $T$, and it is immediate that $ \lim_{T \to 0} C_3(T) = 0$, which completes the proof.
\end{proof}

\begin{proof}[{\bf Proof of Theorem \ref{thm:PIA_convergence}:}]
Again, let $\Delta V^{k+1} := V^{k+1} - V^*$. It follows from Lemmas \ref{lm:measure_contraction} and \ref{lm:V_contraction} that
\begin{align*}
    &d(m^{k+1},m^*_{\lambda}) + \|\Delta V^{k+1}\|_{\Cc^{0,1}_{[0,T]}} \\
    &\leq K_6 C_{2}(T)\mathcal{W}_2(m^{k}_T, m^{*}_{\lambda,T}) + \tilde{C}(T) \left( d(m^k, m^{*}_{\lambda}) + \|\Delta V^{k+1}\|_{\Cc^{0,1}_{[0,T]}} + \|\Delta V^{k}\|_{\Cc^{0,1}_{[0,T]}} \right),
\end{align*}
where $\tilde{C}(T) := \max\{C_1(T), C_3(T)\}$ is a non-decreasing function satisfying $\lim_{T\rightarrow0}\tilde{C}(T)=0$.

Choose $T_0 > 0$ sufficiently small such that $\tilde{C}(T_0) \leq \frac{1}{3}$. Then, for any $T \leq T_0$, we have
\begin{align*}
    \frac{2}{3} \Big( d(m^{k+1},m^*_{\lambda}) + \|\Delta V^{k+1}\|_{\Cc^{0,1}_{[0,T]}} \Big) &\leq d(m^{k+1},m^*_{\lambda}) + (1-\tilde{C}(T_0))\|\Delta V^{k+1}\|_{\Cc^{0,1}_{[0,T]}} \\
    &\leq \left( K_6 C_2(T_0) + \frac{1}{3} \right) d(m^k, m^{*}_{\lambda}) + \frac{1}{3} \|\Delta V^{k}\|_{\Cc^{0,1}_{[0,T]}}.
\end{align*}
Multiplying both sides by $\frac{3}{2}$, we obtain
\begin{align*}
    d(m^{k+1},m^*_{\lambda}) + \|\Delta V^{k+1}\|_{\Cc^{0,1}_{[0,T]}} &\leq \frac{3}{2}\left( K_6 C_2(T_0) + \frac{1}{3} \right) \Big( d(m^k, m^{*}_{\lambda}) + \|\Delta V^{k}\|_{\Cc^{0,1}_{[0,T]}} \Big).
\end{align*}
Then, we choose $\bar{K}_6 > 0$ small enough such that $\frac{3}{2} \bar{K}_6 C_2(T_0) \leq \frac{1}{3}$. For any $K_6 \leq \bar{K}_6$, it holds that
\begin{align}\label{ineq:contraction}
    d(m^{k+1},m^*_{\lambda}) + \|V^{k+1}-V^{*}\|_{\Cc^{0,1}_{[0,T]}} \leq \frac{5}{6} \Big( d(m^k, m^{*}_{\lambda}) + \|V^{k}-V^{*}\|_{\Cc^{0,1}_{[0,T]}} \Big).
\end{align}
Therefore,  
\begin{equation*}
    \lim_{k\rightarrow\infty} \left( \|V^{k}-V^*\|_{\Cc^{0,1}_{[0,T]}} + d(m^k,m^*_\lambda) \right) = 0.
\end{equation*}
Recall that $J^k(t,x) = V^k(t,t,x)$. This uniform convergence directly implies that the sequence $(J^{k}, m^{k})$ generated by the PIA converges to a unique limit $(J^*_\lambda, m^*_\lambda)$, where $J^*_\lambda(t,x) := V^*(t,t,x)$. Moreover, the corresponding limit pair coincides with $(\pi^*_\lambda, m^*_\lambda)$ and constitutes a regularized equilibrium with entropy weight $\lambda$.
\end{proof}

{\begin{Corollary}\label{cor:uniqueness}
    Under the assumptions of Theorem \ref{thm:PIA_convergence}, $(J^*_\lambda, m^*_\lambda)$ is the unique fixed point of $\Phi_\lambda$ in the space $\Cc^{0,2}([0, T]\times \R^d)\times \mathscr{M}_\nu$, and $(V^*, m^*_\lambda)$ is the unique classical solution to the system \eqref{eq:cha.HJBentropy_prime}--\eqref{eq:cha.m} in the space $\Cc^{1,2}_{[0, T]}\times \mathscr{M}_\nu$.
\end{Corollary}
\begin{proof}
  It suffices to prove uniqueness. First, under the assumptions in Theorem \ref{thm:PIA_convergence}, \eqref{ineq:contraction} implies that
$(J^*_\lambda, m^*_{\lambda})$ is the unique fixed point of $\Phi_\lambda$ within the space $\Mc_\lambda\times \mathscr{M}_{\nu}^{\kappa,C^*}$.
Moreover, 
Lemma \ref{lm:iteration.estimateentropy} implies that any pair $(v, m)\in \Cc^{1,2}_{[0, T]}\times \mathscr{M}_\nu$ solving \eqref{eq:cha.HJBentropy_prime}--\eqref{eq:cha.m} corresponds to a fixed point $(J, m)\in  \Cc^{0,2}([0, T]\times \R^d)\times \mathscr{M}_\nu$ of $\Phi_\lambda$ with $J(t,x)=v(t,t,x)$.

Let $(\tilde J, \tilde m)\in  \Cc^{0,2}([0, T]\times \R^d)\times \mathscr{M}_\nu$ be an arbitrary fixed point of $\Phi_\lambda$.  Choosing the initial data in the PIA as $(J^0, m^0)= (\tilde J, \tilde m)$, Step 2 in the proof of Lemma \ref{lm:iteration.estimateentropy} yields that $m^1\in \mathscr{M}_\nu^{\kappa, C^*}$. As $(\tilde J,\tilde m)$ is a fixed point, we have $(J^1,m^1)=(J^0,m^0)$, and therefore $$\tilde m\in\mathscr{M}_\nu^{\kappa, C^*}.$$ Also, Step 3 in the same proof shows that $\|J^{1}\|_{\Cc^{0,1}_{\alpha}([0,T]\times\R^d)}\leq A^*$. Otherwise, the sublinear growth estimate  \eqref{est:2} would imply $$\|J^{1}\|_{\Cc^{0,1}_{\alpha}([0,T]\times\R^d)}<\|J^0\|_{\Cc^{0,1}_{\alpha}([0,T]\times\R^d)},$$
which contradicts the fixed-point identity $(J^1,m^1)=(J^0,m^0)$. Then, Step 3 further implies that $\tilde J = J^0\in \Mc_\lambda$. As a result, any fixed point of $\Phi_\lambda$ in $\Cc^{0,2}([0, T]\times \R^d)\times \mathscr{M}_\nu$ must belong to the subset $\Mc_\lambda\times \mathscr{M}_\nu^{\kappa, C^*}$. Since $\Phi_\lambda$ admits a unique fixed point in this subset, $(J^*_\lambda, m^*_\lambda)$ is the unique fixed point in the space $\Cc^{0,2}([0, T]\times \R^d)\times \mathscr{M}_\nu$. Moreover, $(V^*, m^*_\lambda)$ is the unique classical solution to \eqref{eq:cha.HJBentropy_prime}--\eqref{eq:cha.m} in the space $\Cc^{1,2}_{[0, T]}\times \mathscr{M}_\nu$.
\end{proof}}

\vspace{0.5in}
\noindent
\textbf{Acknowledgements}:
Erhan Bayraktar is supported in part by the National Science Foundation under grant DMS-2507940 and by the
Susan M. Smith Chair. Zhenhua Wang  is supported by Shandong Excellent Young Scientists Fund Program (Overseas) under grant No. 2025HWYQ--022 and National Natural Science Foundation of China (NSFC) under grant no.12501659.
Xiang Yu is supported by the Hong Kong RGC General Research Fund (GRF) under grant no. 15211524 and grant no. 15214125.

\ \\
{
\small
\bibliographystyle{siam}
\bibliography{ref}
}
\end{document}